\documentclass[12pt]{amsart}
\usepackage{amsmath,amssymb,amsthm,booktabs,geometry,microtype}
\usepackage[colorlinks,linkcolor=blue,citecolor=blue,urlcolor=blue]{hyperref}
\usepackage{enumerate}

\microtypesetup{expansion=false}

\newtheorem{theorem}{Theorem}[section]
\newtheorem{lemma}[theorem]{Lemma}
\newtheorem{proposition}[theorem]{Proposition}
\newtheorem{corollary}[theorem]{Corollary}
\theoremstyle{definition}
\newtheorem{definitionx}[theorem]{Definition}
\theoremstyle{remark}
\newtheorem{remark}[theorem]{Remark}

\DeclareMathOperator{\den}{den}
\newcommand{\Zt}{\mathbb Z_2}

\newcommand{\ma}{\mu_\alpha}
\newcommand{\mb}{\mu_\beta}
\newcommand{\Nn}{\mathcal N}
\newcommand{\doi}[1]{\href{https://doi.org/#1}{doi:#1}}
\newcommand{\Ct}{\widetilde C}
\newcommand{\Dt}{\widetilde D}
\DeclareMathOperator{\Fhyp}{{}_2F_1}
\newcommand{\Ber}{\mathcal B}
\newcommand{\Eul}{\mathsf E}
\newcommand{\stir}[2]{\genfrac{\{}{\}}{0pt}{}{#1}{#2}}

\title[Ultraspherical background of the Jacobi phase expansion]{The ultraspherical background of the Jacobi phase expansion, in closed form}
\author[I. Area]{Iv\'an Area}
\address[I. Area]{IFCAE, Universidade de Vigo, Departamento de Matem\'atica Aplicada II, E. E. Aeron\'au\-tica e do Espazo, Campus As Lagoas s/n, 32004 Ourense, Spain}
\email[I. Area]{area@uvigo.gal}

\begin{document}
\begin{abstract}
The large-degree phase of the Jacobi polynomial
$P_n^{(\alpha,\beta)}(\cos\theta)$ contains, at each order in
$\Nn^{-1}$, $\Nn=n+(\alpha+\beta+1)/2$, an additive constant $\kappa_r$
that the phase equation does not determine.  These constants are antisymmetric in
$(\alpha^2,\beta^2)$ and vanish for $\alpha=\beta$.  We obtain three
results.  First, on the diagonal $\alpha=\beta$ the logarithm of the
midpoint amplitude is computed in closed form: if $L_n$ denotes the
coefficient of $(2\Nn)^{-n}$ in that logarithm, then
$\Upsilon_A(t)=\sum_{n\ge1}L_nt^n/(n-1)!$ equals
\[
  \Upsilon_A(t)=\frac{e^{t}\bigl(\cosh2\alpha t-\cosh t\bigr)}{\sinh 2t}.
\]
The proof is algebraic: the identity is established for half-odd-integer
$\alpha$, where the expansion terminates, by means of two classical
recurrences for Bessel polynomials, and is extended to all $\alpha$ by
a degree argument.  Second, the derivative of the logarithm of the
amplitude with respect to $\alpha^2$ at the Legendre point
$\alpha=\beta=0$ is computed in closed form to all orders.  An exact reduction of an Appell
$F_3$ function to a Gauss function shows that it is a parameter
derivative of a zero-balanced ${}_2F_1$ at argument $2$; the real part
of its expansion is a trigamma series, and the imaginary part is the
unique formal solution of an explicit central-difference equation.  Third, we
deduce that the coefficients $w_m$ of an explicit even power series
attached to the imaginary part satisfy
$\nu_2(w_m)=-2\nu_2((2m)!)$ for every $m\ge0$ ($\nu_2$ the $2$-adic
valuation).  This determines, for
every $r$, the exact power of $2$ in the denominator of the coefficient
of $\alpha^2$ in the polynomial that represents $\kappa_r$.  We also give a factorisation of the
amplitude into products of pairs of Gauss functions, further valuation
formulas, and a numerical illustration for the zeros.
\end{abstract}

\maketitle

\section{Introduction}

We use the standard normalisation and notation for Jacobi polynomials
in~\cite{Szego,DLMF}; general background on hypergeometric and
classical special functions may be found in~\cite{AndrewsAskeyRoy}.
For $\alpha,\beta>-1$ the Jacobi polynomial $P_n^{(\alpha,\beta)}$ is
orthogonal on $[-1,1]$ with respect to the weight
$(1-x)^\alpha(1+x)^\beta$, and its $n$ zeros are real, simple and lie
in $(-1,1)$; they are the nodes of Gauss--Jacobi quadrature.  The
ultraspherical case $\alpha=\beta$ contains the Legendre polynomials
($\alpha=\beta=0$), the Chebyshev polynomials
($\alpha=\beta=\pm\frac12$) and the Gegenbauer polynomials.  Whenever
the zeros of $P_n^{(\alpha,\beta)}$ are considered as an ordered set in
$(-1,1)$ (here and in Section~\ref{sec:zeros}) we assume
$\alpha,\beta>-1$.  The formal and hypergeometric identities proved
below are polynomial or meromorphic in the parameters and, where
explicitly stated, hold beyond this range.  Write the zeros as
$x_{n,k}=\cos\theta_{n,k}$, $0<\theta_{n,1}<\dots<\theta_{n,n}<\pi$,
and put
\[
  \Nn=n+\frac{\alpha+\beta+1}2,\qquad A=\alpha^2,\qquad B=\beta^2 .
\]

The starting point is the elementary observation that the zeros are
almost equally spaced in the angle $\theta$.  For the Chebyshev
polynomials this is exact: the zeros of $T_n$ are
$\cos\bigl((k-\frac12)\pi/n\bigr)$.  In general one has, for the zeros
away from the endpoints,
\begin{equation}\label{eq:leading}
  \theta_{n,k}
  =\frac{(k-\frac12)\pi+(\frac\alpha2+\frac14)\pi}{\Nn}
  +O(\Nn^{-2}),
\end{equation}
and the question addressed by the phase expansion is what the
corrections look like, to all orders in $\Nn^{-1}$.  The source
of~\eqref{eq:leading} is the Liouville normal form of the Jacobi
differential equation: the function
$u(\theta)=\sin^{\alpha+\frac12}\!\frac\theta2\,
\cos^{\beta+\frac12}\!\frac\theta2\,P_n^{(\alpha,\beta)}(\cos\theta)$
satisfies
\begin{equation}\label{eq:liouville}
  u''(\theta)+\Bigl[\Nn^2+\frac{\frac14-A}{4\sin^2\frac\theta2}
  +\frac{\frac14-B}{4\cos^2\frac\theta2}\Bigr]u(\theta)=0
\end{equation}
\cite[Eq.~(4.24.2)]{Szego}.  For $A=B=\frac14$ this is the equation of
a pure sinusoid of frequency $\Nn$; in general $u$ is a sinusoid of
frequency $\Nn$ whose phase is slowly deformed by the potential, and
the zeros of $P_n^{(\alpha,\beta)}$ are the points where that phase
crosses the values $(k-\frac12)\pi$.  Note that the
equation~\eqref{eq:liouville} depends on the parameters only through
$A$ and $B$; this is why all the coefficients considered in this paper
are polynomials in $A$ and $B$.

Solving the phase equation associated with~\eqref{eq:liouville}
recursively in powers of $\Nn^{-2}$, one obtains a phase function
\begin{equation}\label{eq:quantisation}
  \Xi(\theta)=\Nn\theta-\Bigl(\frac\alpha2+\frac14\Bigr)\pi
  +\sum_{r\ge1}\frac{\Psi_r(\theta)}{\Nn^{2r-1}},
\end{equation}
such that the zeros are, to all orders in $1/\Nn$, the solutions of
$\Xi(\theta_{n,k})=(k-\frac12)\pi$.  The first term is elementary,
\[
  \Psi_1(\theta)
  =\frac{(1-4B)\tan\frac\theta2-(1-4A)\cot\frac\theta2}{16},
\]
and already very effective: for $n=20$, $\alpha=\frac32$,
$\beta=\frac14$, the maximal error of~\eqref{eq:leading} on the ten
central zeros is $2.2\cdot10^{-3}$, and adding the term
$\Psi_1(\theta)/\Nn$ reduces it to $3.9\cdot10^{-6}$.  Each $\Psi_r$ is
obtained from the recursion by an integration and is therefore
determined only up to an additive constant.  We normalise $\Psi_r$ to
be an odd Laurent polynomial in $\cot(\theta/2)$, without constant
term; the information not fixed by the differential equation is then
carried by the value at the midpoint,
\[
  \kappa_r=\Psi_r(\pi/2).
\]
The numbers $\kappa_r$, which are polynomials in $A$ and $B$ with
rational coefficients, are the only part of~\eqref{eq:quantisation}
that the phase equation does not produce.  We call them the
\emph{phase constants}.  The first three are
\begin{equation}\label{eq:kappaex}
\begin{gathered}
  \kappa_1=\frac{A-B}4,\qquad
  \kappa_2=\frac{(A-B)(2A+2B-7)}{96},\\
  \kappa_3=\frac{(A-B)(28A^2+88AB+28B^2-380A-380B+781)}{7680}.
\end{gathered}
\end{equation}

Two features of~\eqref{eq:kappaex} persist for all $r$ and motivate the
paper.  The first is antisymmetry.  Since
$P_n^{(\alpha,\beta)}(-x)=(-1)^nP_n^{(\beta,\alpha)}(x)$, reflection at
$\theta=\pi/2$ interchanges $\alpha$ and $\beta$ and reverses the sign
of the phase, so
\begin{equation}\label{eq:antisym}
  \kappa_r(B,A)=-\kappa_r(A,B),\qquad\text{hence}\qquad (A-B)\mid\kappa_r .
\end{equation}
In particular $\kappa_r$ vanishes identically on the diagonal $A=B$,
which contains the ultraspherical case $\alpha=\beta$: the Legendre and
Gegenbauer cases need no phase constants.  The nontrivial content
of~\eqref{eq:quantisation} is therefore transverse to the diagonal.  On
the diagonal itself the phase constants vanish but the series from
which they are extracted (the amplitude, introduced below) does not;
it is a nontrivial asymptotic series.  Our first question is whether
this series can be computed in closed form.  The second feature is
arithmetic: the denominators $4=2^2$, $96=2^5\cdot3$,
$7680=2^9\cdot15$ and, for $\kappa_4$, $86016=2^{12}\cdot21$ grow
mainly through powers of $2$.  In~\cite{companion} it is 
showed that the power of $2$ in the denominator of $\kappa_r$ is at
most $2^{E_r}$ with $E_r=3r-1+\nu_2((r-1)!)$, where $\nu_2$ denotes
the $2$-adic valuation.  Our second question is whether this bound is
attained for every $r$.  Both questions are answered in this paper:
the diagonal series is the Stirling series of a quotient of Gamma
functions (Theorem~\ref{thm:main} and Corollary~\ref{cor:contiguity}), and the bound $2^{E_r}$ is sharp for every
$r$ (Theorem~\ref{thm:dyadic}).  The proof of the second statement
uses the first: the closed form on the diagonal is the background over
which the first-order variation transverse to the diagonal is
computed.

\medskip

Large-degree asymptotics for Jacobi polynomials and functions have a
long history, from the classical trigonometric and uniform expansions
of Hahn~\cite{Hahn1980} and Frenzen--Wong~\cite{FrenzenWong1985} to
Riemann--Hilbert treatments of Jacobi-type weights
\cite{DeanoHuybrechsOpsomer2016}, fast computation of Gauss--Jacobi
nodes and weights \cite{HaleTownsend2013,GilSeguraTemme2021},
nonoscillatory phase-function methods \cite{BremerYang2020}, explicit
error bounds \cite{HuangLinWangWong2025} and, closest to our starting
point, the inverse-factorial large-degree expansions of
Nemes~\cite{Nemes2025} for Jacobi and related functions with fixed
parameters, which carry explicit computable error bounds, are built
from the same Hankel coefficients as~\eqref{eq:F} below, and terminate
at half-odd-integer parameters.  These works develop the expansions,
their error bounds and the associated algorithms.  The present paper
addresses a different question: the exact algebraic and arithmetic
structure of the coefficients of the phase expansion, and in particular
of the constants $\kappa_r$.

\medskip

We now state the main results, which are proved in the following
sections.  Let
\[
  a_k(\nu)=\frac{(\frac12+\nu)_k(\frac12-\nu)_k}{(-2)^kk!}
  =\frac{\prod_{j=0}^{k-1}\bigl(4\nu^2-(2j+1)^2\bigr)}{8^kk!}
\]
be the Hankel coefficients, and put $Z=2\Nn$.  The large-degree
behaviour of the Jacobi polynomials at the midpoint $\theta=\pi/2$ is
carried by the inverse-factorial series
\begin{equation}\label{eq:F}
\begin{gathered}
  F(Z)=\sum_{n\ge0}\frac{G_n}{(Z-1)(Z-2)\cdots(Z-n)},\\
  G_n=\sum_{l=0}^na_l(\alpha)a_{n-l}(\beta)\,p^{\,l}q^{\,n-l},
  \qquad p=1+i,\quad q=1-i,
\end{gathered}
\end{equation}
which we call the \emph{amplitude}.  Write
$\log F=\sum_{n\ge1}L_nZ^{-n}$; the coefficients $L_n$ are polynomials
in $A$ and $B$.  The phase constants are the imaginary parts of the odd
coefficients,
\begin{equation}\label{eq:kappadef}
 \kappa_r=2^{-(2r-1)}\operatorname{Im}L_{2r-1}.
\end{equation}
Formula~\eqref{eq:kappadef} is the definition of $\kappa_r$ used in
this paper; its agreement with the constants
$\Psi_r(\pi/2)$ of~\eqref{eq:quantisation} is discussed in
Proposition~\ref{prop:dictionary}.

Our first result concerns the diagonal $A=B=\alpha^2$.  Let
$L_n^{\mathrm{diag}}$ be the value of $L_n$ there, and let
\[
  \Upsilon_A(t)=\sum_{n\ge1}\frac{L_n^{\mathrm{diag}}}{(n-1)!}\,t^n
\]
be the associated generating function.  We write $\Ber_k(x)$ for the Bernoulli
polynomials and $\Ber_k=\Ber_k(0)$ for the Bernoulli numbers, so that
$\Ber_1=-\frac12$, and $\Eul_{2k}$ for the Euler numbers.  As usual,
$(a)_k=a(a+1)\cdots(a+k-1)$ is the Pochhammer symbol,
$\psi=\Gamma'/\Gamma$ is the digamma function, $\operatorname{Li}_k$ is
the polylogarithm, and $[x^k]f$ denotes the coefficient of $x^k$ in a
power series $f$.

\begin{theorem}\label{thm:main}
For every $\alpha$,
\begin{equation}\label{eq:closed}
  \Upsilon_A(t)=\frac{e^{t}\bigl(\cosh2\alpha t-\cosh t\bigr)}{\sinh 2t}
  =\tfrac12\bigl(\operatorname{csch}t+\operatorname{sech}t\bigr)
   \bigl(\cosh2\alpha t-\cosh t\bigr),\qquad A=\alpha^2 .
\end{equation}
Equivalently, coefficientwise,
\begin{equation}\label{eq:coef}
  L_n^{\mathrm{diag}}=\frac{4^n}{n(n+1)}
  \Bigl[\Ber_{n+1}\bigl(\tfrac{3+2\alpha}4\bigr)
       +\Ber_{n+1}\bigl(\tfrac{3-2\alpha}4\bigr)
       -\Ber_{n+1}(1)-\Ber_{n+1}\bigl(\tfrac12\bigr)\Bigr].
\end{equation}
\end{theorem}

The proof occupies Sections~\ref{sec:half} and~\ref{sec:transfer}.  It
does not use any asymptotic representation of the Jacobi polynomials.
The identity is first established for $\alpha=\ell+\frac12$ with
$\ell$ a non-negative integer, where~\eqref{eq:F} terminates, by an
induction based on two classical recurrences for Bessel polynomials;
a degree bound in $A$ then extends it to all~$\alpha$.

\begin{corollary}\label{cor:contiguity}
On the diagonal, as an identity of formal Laurent series at $Z=\infty$,
\[
F(Z+4)\,Z(Z+2)=F(Z)\bigl((Z+1)^2-4A\bigr).
\]
Equivalently, $\mathcal G(Z):=\Gamma(Z)F(Z)$ satisfies, in the same
formal sense,
\[
\mathcal G(Z+4)=(Z+1)(Z+3)\bigl((Z+1)^2-4A\bigr)\mathcal G(Z).
\]
For half-odd-integer $\alpha$ the first identity is an identity of
rational functions of $Z$ and the second one of meromorphic functions,
by Theorem~\ref{thm:halfinteger}.
\end{corollary}

\begin{corollary}\label{cor:specials}
\leavevmode
\begin{enumerate}
\item[(i)] $\Upsilon_A=\upsilon_1\cosh2\alpha t+\upsilon_0$ with
  $\upsilon_0(t)$, $\upsilon_1(t)$ independent of $\alpha$;
  coefficientwise, for $k\ge1$,
  $[A^k]\Upsilon_A=\frac{2^{2k-1}}{(2k)!}t^{2k-2}[A^1]\Upsilon_A$.
\item[(ii)] $\Upsilon_{1/4}\equiv0$: there is no correction in the
  Chebyshev cases $\alpha=\beta=\pm\frac12$.
\item[(iii)] In the Legendre case,
  $\Upsilon_0(t)=\frac12\bigl(\operatorname{sech}t-\tanh\frac t2-1\bigr)$.
\end{enumerate}
\end{corollary}

The second group of results concerns the first-order variation of the
amplitude at the Legendre point $A=B=0$.  Let
\begin{equation}\label{eq:gdef}
  g_n=\frac{\partial L_n}{\partial A}\Big|_{A=B=0}\in\mathbb Q(i),
  \qquad
  \Phi(t)=\sum_{n\ge1}\frac{g_n}{(n-1)!}\,t^n .
\end{equation}
By~\eqref{eq:kappadef}, the coefficient of the monomial $A$ in
$\kappa_r$ is
\begin{equation}\label{eq:A1B0}
  [A^1B^0]\kappa_r=2^{-(2r-1)}\operatorname{Im}g_{2r-1}\qquad(r\ge1).
\end{equation}
We call the sequence $(g_n)$, or its generating series
$\sum_ng_nZ^{-n}$, the \emph{Legendre tangent} of the amplitude: it is
the derivative of $\log F$ in the direction of $A$ at the Legendre
point.  At the midpoint, the derivative of $\log F$ along the diagonal
is real and the derivative in the direction of $A-B$, transverse to the
diagonal, is purely imaginary (Section~\ref{sec:transverse}).  Thus
$\operatorname{Re}g_n$ is the diagonal part of the tangent and
$\operatorname{Im}g_n$ its transverse part.  The real parts are
determined by Theorem~\ref{thm:main}:

\begin{corollary}\label{cor:R}
With $g_n$ and $\Phi$ as in~\eqref{eq:gdef},
\[
  \operatorname{Re}\Phi(t)=\frac{t^2e^t}{\sinh2t},
\]
that is, $\operatorname{Re}g_{2k+1}=(1-2^{2k-1})\Ber_{2k}$ and
$\operatorname{Re}g_{2k}=\frac12(2k-1)\Eul_{2k-2}$.  In particular
$\nu_2(\operatorname{Re}g_n)=-1$ for every $n\ge1$.
\end{corollary}

Here and below $\Zt$ denotes the ring
of $2$-adic integers, $\Zt^\times$ its group of units, and $s_2(m)$ the
number of digits $1$ in the binary expansion of $m$, so that
$\nu_2(m!)=m-s_2(m)$.  The imaginary parts of the $g_n$ are the
substantial part.  Put
\begin{equation}\label{eq:Wdef}
  W(t)=\frac{\sinh2t}{t^2}\,\operatorname{Im}\Phi(t),\qquad
  w_m=[t^{2m}]\,W(t),
\end{equation}
so that, by Corollary~\ref{cor:R},
$\sinh(2t)\,\Phi(t)/t^2=e^t+i\,W(t)$, and
\[
  H_r=r+\nu_2\bigl((r-1)!\bigr),\qquad
  E_r=(2r-1)+H_r=3r-1+\nu_2\bigl((r-1)!\bigr).
\]

\begin{theorem}\label{thm:dyadic}
For every $k\ge1$, $\operatorname{Im}g_{2k}=0$, so that $W$ is an even
power series.  For every $m\ge0$,
\[
  \nu_2(w_m)=-2\,\nu_2\bigl((2m)!\bigr),\qquad\text{that is,}\qquad
  \bigl((2m)!\bigr)^2w_m\in\Zt^\times .
\]
Equivalently, for every $r\ge1$,
\[
 \nu_2\bigl(\operatorname{Im}g_{2r-1}\bigr)=-H_r,
 \qquad
 \nu_2\bigl([A^1B^0]\kappa_r\bigr)=-E_r ,
\]
with $\kappa_r$ defined by~\eqref{eq:kappadef}.
\end{theorem}

Theorem~\ref{thm:dyadic} is proved in Subsection~\ref{sec:arithmetic}.  The
equivalence between the statement for $w_m$ and the statement for
$[A^1B^0]\kappa_r$ was established in~\cite{companion}, where the
property $((2m)!)^2w_m\in\Zt^\times$ was verified in a finite range;
the argument is repeated in the proof of
Theorem~\ref{thm:tangentlaw}.  The proof of Theorem~\ref{thm:dyadic} rests
on a closed form for the whole sequence $(g_n)$:

\begin{theorem}\label{thm:imclosed}
As formal series in $Z^{-1}$,
\begin{equation}\label{eq:formal-tangent}
 \sum_{n\ge1}g_n\,Z^{-n}
 =\frac1{16}\,\psi'\!\Bigl(\frac{Z+1}4\Bigr)
 -\frac1{16}\,\psi'\!\Bigl(\frac{Z+3}4\Bigr)
 +\frac14\,\psi'\!\Bigl(\frac{Z+1}2\Bigr)
 -\frac i2\,\widehat Q_\Gamma(Z),
\end{equation}
where each trigamma function $\psi'$ stands for its asymptotic
(Bernoulli) expansion, and $\widehat Q_\Gamma$ is the unique formal
series without constant term that solves the central-difference
equation
\[
  \widehat Q_\Gamma(s+1)-\widehat Q_\Gamma(s-1)
  =-\frac{(2s-1)}{4s(s-1)}
   \left[\frac{\Gamma(s/2)}{\Gamma((1+s)/2)}\right]^2,
  \qquad s=\frac{1-Z}2,
\]
the right-hand side being replaced by its asymptotic expansion in
powers of $s^{-1}$.  In particular
\begin{equation}\label{eq:Imgclosed}
 \operatorname{Im}g_n=-\tfrac12\,\widehat q_n\qquad(n\ge1),
\end{equation}
where $\widehat q_n$ is the coefficient of $Z^{-n}$ in
$\widehat Q_\Gamma$.
\end{theorem}

The proof of Theorem~\ref{thm:imclosed} proceeds as follows.  The
amplitude~\eqref{eq:F} is a Kamp\'e de F\'eriet double series
(Lemma~\ref{lem:KdF}), in fact an Appell $F_3$ function evaluated on a
curve on which it reduces to a single Gauss function
(Theorem~\ref{thm:appellreduce}).  Consequently
$\sum_ng_nZ^{-n}$ is the expansion of a second parameter derivative of
a zero-balanced ${}_2F_1$ at the boundary value $2+i0$
(Corollary~\ref{cor:gausstangent}).  The classical connection formula
splits this boundary value into two blocks which are quotients of Gamma
functions at the Legendre point (Proposition~\ref{prop:split}); the
variation of one of them satisfies a first-order difference equation
which is solved by quadratures (Theorem~\ref{thm:quad}); and the
passage from the resulting convergent expressions to the formal series
in $Z^{-1}$ is made by an asymptotic matching along a vertical line
(Subsection~\ref{sec:matching}), whose estimates are given in
Appendix~\ref{app:vertical}.  Theorem~\ref{thm:dyadic} then follows from
the $2$-adic analysis of the explicit series $\widehat Q_\Gamma$
(Subsection~\ref{sec:arithmetic}): the exponent of the Gamma quotient
has Genocchi-number coefficients, and in the expansion of its
exponential exactly one term has minimal valuation.

Besides these results the paper contains further exact structure of
the amplitude.  The recurrences for Bessel polynomials used in the
proof of Theorem~\ref{thm:main} do not require the special relation
$p^2+q^2=0$ satisfied at the midpoint (Theorem~\ref{thm:genladder}); on
the whole curve described by the weights $p,q$ as $\theta$ varies,
where $p+q=pq$, the half-odd-integer amplitudes satisfy a three-term
contiguity whose numerators are given in closed form in terms of ballot
numbers (Theorems~\ref{thm:curvecontig} and~\ref{thm:ballot}), and the
derivatives of all orders normal to that curve are determined
recursively (Theorem~\ref{thm:normaljets}).  For independent
half-odd-integer parameters $\alpha,\beta$ the amplitudes satisfy a
similar recurrence off the diagonal (Theorem~\ref{thm:twopar}).  The
amplitude is an explicit sum of products of pairs of Gauss functions,
the $r$-th term being of order $Z^{-3r}$
(Corollary~\ref{cor:layers}); from this expansion we derive four
valuation formulas for the numerator and the denominator of
$\sum_ng_nZ^{-n}$ (Theorem~\ref{thm:lawD} and
Proposition~\ref{prop:lawN}).  Section~\ref{sec:zeros} illustrates the
use of~\eqref{eq:quantisation} for the numerical computation of zeros.

Two statements of the paper are not proved in full generality, and
neither is used in the proofs of Theorems~\ref{thm:main},
\ref{thm:dyadic} and~\ref{thm:imclosed}.  The first is the agreement
of~\eqref{eq:kappadef} with the constants $\Psi_r(\pi/2)$
of~\eqref{eq:quantisation}, which is proved for $r\le8$
(Proposition~\ref{prop:dictionary}); for $r>8$ the last assertion of
Theorem~\ref{thm:dyadic} is a statement about the polynomials defined
by~\eqref{eq:kappadef}.  The second is a periodicity hypothesis under
which the function whose expansion is $\sum_ng_nZ^{-n}$ is evaluated in
closed form on the real axis (Proposition~\ref{prop:realaxis}).

\medskip

Throughout the paper, identities in $Z^{-1}$ (equivalently in
$v=Z^{-1}$ or in $s^{-1}$, $s=(1-Z)/2$) are identities of formal
Laurent series unless convergence is explicitly asserted.  When a
meromorphic function is introduced, its expansion for large $Z$ is
always understood in the stated sector.  For half-odd-integer
parameters the identities are identities of rational functions or of
quotients of Gamma functions.

The paper is organised as follows.  Sections~\ref{sec:half}
and~\ref{sec:transfer} prove Theorem~\ref{thm:main} and its
corollaries.  Sections~\ref{sec:genladder}, \ref{sec:curve}
and~\ref{sec:indexop} treat the half-odd-integer amplitudes away
from the midpoint.  Section~\ref{sec:offdiag} contains the proofs of
Theorems~\ref{thm:imclosed} and~\ref{thm:dyadic}: the Appell reduction
(Subsection~\ref{sec:appell}), the boundary split and the quadratures
(Subsections~\ref{sec:split} and~\ref{sec:formalquad}), the asymptotic
matching (Subsection~\ref{sec:matching}) and the arithmetic
(Subsection~\ref{sec:arithmetic}); Subsections~\ref{sec:boundarytransfer}
and~\ref{sec:anchors} evaluate the boundary value exactly on the real
axis and are used only in Proposition~\ref{prop:realaxis}.  The proofs
of Theorems~\ref{thm:imclosed} and~\ref{thm:dyadic} are contained in
Section~\ref{sec:offdiag} and Appendix~\ref{app:vertical}; apart from
the notation of the Introduction and Lemma~\ref{lem:degree}, they do
not depend on Sections~\ref{sec:genladder}--\ref{sec:indexop}
or~\ref{sec:layers}--\ref{sec:zeros}.
Section~\ref{sec:layers} gives the expansion of the amplitude into
products of Gauss functions, Section~\ref{sec:transverse} the linear
system satisfied by the transverse variation, and
Section~\ref{sec:tangentlayers} the valuation formulas.
Section~\ref{sec:zeros} contains the numerical illustration, and
Appendix~\ref{app:vertical} the estimates used in
Subsection~\ref{sec:matching}.  

\section{The half-odd-integer family}\label{sec:half}

\subsection{Bessel polynomials at the midpoint}

Write $y_n$ for the Bessel polynomials
\[
  y_n(x)=\sum_{k=0}^n\frac{(n+k)!}{(n-k)!\,k!\,2^k}\,x^k .
\]

\begin{lemma}\label{lem:hankelhalf}
For $\alpha=\ell+\frac12$,
\[
  a_k\bigl(\ell+\tfrac12\bigr)=\frac{(\ell+k)!}{(\ell-k)!\,k!\,2^k}
  \quad(0\le k\le\ell),\qquad a_k\bigl(\ell+\tfrac12\bigr)=0\quad(k>\ell),
\]
so that $\sum_ka_k(\ell+\frac12)\xi^k=y_\ell(\xi)$.
\end{lemma}

\begin{proof}
Since $4\alpha^2=(2\ell+1)^2$, the product in the definition of $a_k$ is
\[
\prod_{j<k}\bigl((2\ell+1)^2-(2j+1)^2\bigr)
=\prod_{j<k}(2\ell-2j)(2\ell+2j+2)=4^k\frac{\ell!}{(\ell-k)!}
\frac{(\ell+k)!}{\ell!}
\]
for $k\le\ell$, and it vanishes for $k>\ell$.  Dividing by $8^kk!$ gives
the stated value.
\end{proof}

We write $F_\ell(Z)$ for the series \eqref{eq:F} on the diagonal at
$\alpha=\beta=\ell+\frac12$, and $G_n^{(\ell)}$ for its coefficients.  By
Lemma~\ref{lem:hankelhalf}, the generating function of these coefficients
is the \emph{polynomial}
\[
  \Pi_\ell(\xi):=\sum_nG_n^{(\ell)}\xi^n=y_\ell(p\xi)\,y_\ell(q\xi),
  \qquad p=1+i,\ q=1-i,
\]
of degree $2\ell$; in particular $G_n^{(\ell)}=0$ for $n>2\ell$ and
$F_\ell(Z)$ is a rational function of $Z$.

The four relations we use among the midpoint constants are
\begin{equation}\label{eq:pq}
  pq=2,\qquad p+q=2,\qquad p^2+q^2=0,\qquad \frac1p+\frac1q=1 .
\end{equation}
The third relation is the one used in the proof of
Proposition~\ref{prop:step4}.

\subsection{Two classical recurrences, and their consequence}

\begin{lemma}\label{lem:bessel}
For $n\ge0$, with the convention $y_{-1}(x)=1$,
\[
  \text{(a)}\quad x^2y_n'(x)=(nx-1)y_n(x)+y_{n-1}(x),
  \qquad
  \text{(b)}\quad y_{n+1}(x)=(2n+1)x\,y_n(x)+y_{n-1}(x).
\]
\end{lemma}
\begin{proof}
Both identities follow by comparing coefficients in the defining sum;
they are classical~\cite{KrallFrink,Burchnall,DunsterEtAl}.
\end{proof}

\begin{definitionx}
$\displaystyle
 C_\ell(\xi)=\tfrac12\Bigl(q\,y_{\ell-1}(p\xi)y_\ell(q\xi)
                        +p\,y_\ell(p\xi)y_{\ell-1}(q\xi)\Bigr)$,
which is real because the two terms are complex conjugate.
\end{definitionx}

\begin{proposition}\label{prop:step3}
$\xi^2\Pi_\ell'(\xi)+\bigl(1-2\ell \xi\bigr)\Pi_\ell(\xi)=C_\ell(\xi)$.
\end{proposition}

\begin{proof}
Put $U=y_\ell(p\xi)$, $V=y_\ell(q\xi)$, $U_-=y_{\ell-1}(p\xi)$, $V_-=y_{\ell-1}(q\xi)$.  Lemma~\ref{lem:bessel}(a) at $x=p\xi$ reads $p\,\xi^2U'=(\ell p\xi-1)U+U_-$, and similarly for $V$.  Hence
\begin{multline*}
  \xi^2\Pi_\ell'=\xi^2(U'V+UV')
  =\frac{(\ell p\xi-1)U+U_-}{p}\,V+U\,\frac{(\ell q\xi-1)V+V_-}{q}\\
  =\Bigl(2\ell \xi-\frac1p-\frac1q\Bigr)UV+\frac{U_-V}p+\frac{UV_-}q .
\end{multline*}
By \eqref{eq:pq}, $\frac1p+\frac1q=1$, $\frac1p=\frac q2$ and $\frac1q=\frac p2$, so the right-hand side is $(2\ell \xi-1)\Pi_\ell+C_\ell$.
\end{proof}

\begin{proposition}\label{prop:step4}
$C_{\ell+2}=C_\ell+(4\ell+2)\,\xi\,\Pi_\ell$.
\end{proposition}

\begin{proof}
With $U_+=y_{\ell+1}(p\xi)$, $U_{++}=y_{\ell+2}(p\xi)$ and similarly for $V$,
Lemma~\ref{lem:bessel}(b) gives $U_{++}=(2\ell+3)p\xi U_++U$ and
$V_{++}=(2\ell+3)q\xi V_++V$.  Therefore
\[
  2C_{\ell+2}=qU_+V_{++}+pU_{++}V_+
  =(2\ell+3)\xi\,(p^2+q^2)\,U_+V_+\;+\;qU_+V+pUV_+ ,
\]
and the first term vanishes by \eqref{eq:pq}.  Applying
Lemma~\ref{lem:bessel}(b) once more, $U_+=(2\ell+1)p\xi U+U_-$ and
$V_+=(2\ell+1)q\xi V+V_-$, so
\[
  qU_+V+pUV_+=2(2\ell+1)pq\,\xi\,UV+\bigl(qU_-V+pUV_-\bigr)
             =4(2\ell+1)\xi\,\Pi_\ell+2C_\ell ,
\]
using $pq=2$.  Dividing by $2$ gives the claim.
\end{proof}

\begin{corollary}\label{cor:gstep}
$\xi^2\Pi_{\ell+2}'+\bigl[1-(2\ell+4)\xi\bigr]\Pi_{\ell+2}
 =\xi^2\Pi_\ell'+\bigl[1+(2\ell+2)\xi\bigr]\Pi_\ell$.
\end{corollary}

\begin{proof}
By Proposition~\ref{prop:step3} the left side is $C_{\ell+2}$ and the right
side is $C_\ell+(4\ell+2)\xi \Pi_\ell$; apply Proposition~\ref{prop:step4}.
\end{proof}

\subsection{From the generating function to the inverse-factorial series}

Let $e_n(Z)=1/\bigl((Z-1)\cdots(Z-n)\bigr)$, $e_0=1$.

\begin{lemma}\label{lem:basis}
$Z\,e_n=e_{n-1}+n\,e_n$ for $n\ge1$.
\end{lemma}

\begin{proof}
$Ze_n=\bigl((Z-n)+n\bigr)e_n$ and $(Z-n)e_n=e_{n-1}$.
\end{proof}

\begin{proposition}\label{prop:Frec}
$F_{\ell+2}(Z)\,(Z-2\ell-4)=F_\ell(Z)\,(Z+2\ell+2)$.
\end{proposition}

\begin{proof}
Write $F_\ell=\sum_nG_n^{(\ell)}e_n$ with $G^{(\ell)}_0=1$.  By
Lemma~\ref{lem:basis}, for any constant $c$,
\[
  (Z-c)F_\ell=Z-c+G_1^{(\ell)}
  +\sum_{m\ge1}\Bigl(G^{(\ell)}_{m+1}+(m-c)G^{(\ell)}_m\Bigr)e_m .
\]
Since $G^{(\ell)}_0=1$ for every $\ell$, the terms in $Z$ agree on the two
sides of the claimed identity, and comparing coefficients of $e_m$ (and the
constant, which is the case $m=0$) shows that the identity is equivalent to
\[
  G^{(\ell+2)}_{m+1}+(m-2\ell-4)G^{(\ell+2)}_m
  =G^{(\ell)}_{m+1}+(m+2\ell+2)G^{(\ell)}_m
  \qquad(m\ge0).
\]
Multiplying by $\xi^m$ and summing, and using
$\sum_mG_{m+1}\xi^m=(\Pi-1)/\xi$ and $\sum_mmG_m\xi^m=\xi\Pi'$, this is exactly
Corollary~\ref{cor:gstep} after multiplication by $\xi$.
\end{proof}

\begin{lemma}\label{lem:base}
$F_0=1$ and $F_1(Z)=\dfrac{Z}{Z-2}$.
\end{lemma}

\begin{proof}
$y_0(x)=1$ gives $\Pi_0=1$, so $G_0=1$ and $G_n=0$ for $n\ge1$.  $y_1(x)=1+x$
gives $\Pi_1(\xi)=(1+p\xi)(1+q\xi)=1+2\xi+2\xi^2$, so
\[
F_1=1+\frac2{Z-1}+\frac2{(Z-1)(Z-2)}   =\frac{Z^2-Z}{(Z-1)(Z-2)}=\frac Z{Z-2}.
\]
\end{proof}

\begin{theorem}\label{thm:halfinteger}
For every $\ell\ge0$, with $\zeta=Z/4$,
\[
  F_\ell(Z)=\mathcal M_\ell(Z):=
  \frac{\Gamma\bigl(\zeta-\frac\ell2\bigr)\,
        \Gamma\bigl(\zeta+\frac{\ell+1}2\bigr)}
       {\Gamma(\zeta)\,\Gamma\bigl(\zeta+\frac12\bigr)} .
\]
\end{theorem}

\begin{proof}
By the functional equation of $\Gamma$,
\[
  \frac{\mathcal M_{\ell+2}}{\mathcal M_\ell}
  =\frac{\Gamma(\zeta-\frac\ell2-1)}{\Gamma(\zeta-\frac\ell2)}
   \cdot\frac{\Gamma(\zeta+\frac{\ell+3}2)}{\Gamma(\zeta+\frac{\ell+1}2)}
  =\frac{\zeta+\frac{\ell+1}2}{\zeta-\frac{\ell+2}2}
  =\frac{Z+2\ell+2}{Z-2\ell-4},
\]
which is Proposition~\ref{prop:Frec}.  For the base cases, $\mathcal M_0=1$ and
\[
\mathcal M_1=\frac{\Gamma(\zeta-\frac12)}{\Gamma(\zeta+\frac12)}
\frac{\Gamma(\zeta+1)}{\Gamma(\zeta)}=\frac{\zeta}{\zeta-\frac12}
=\frac Z{Z-2}, 
\]
matching Lemma~\ref{lem:base}.  Induction in steps of two along each parity class completes the proof.
\end{proof}

\section{\texorpdfstring{From half-integers to all $\alpha$}{From half-integers to all alpha}}\label{sec:transfer}

\begin{lemma}\label{lem:degree}
For arbitrary weights $p,q$, each coefficient $L_n$ of $\log F$ is a
polynomial in $(A,B)$ of total degree at most $n$.  In particular
$\deg_AL_n^{\mathrm{diag}}\le n$.  Moreover the right-hand side of
\eqref{eq:coef} is a polynomial in $A$ of degree at most
$\lfloor(n+1)/2\rfloor\le n$.
\end{lemma}

\begin{proof}
$a_k(\nu)$ is a polynomial in $\nu^2$ of degree $k$, so
$a_l(\alpha)a_{n-l}(\beta)$, and hence $G_n$, is a polynomial in $(A,B)$
of total degree at most $n$.  Since
\[
e_n=\sum_{m\ge n}\stir mnZ^{-m}
\]
with $\stir mn$ the Stirling numbers of the
second kind, the coefficients
\[
c_m=\sum_{n\le m}G_n\stir mn
\]
of $F=1+\sum_mc_mZ^{-m}$ have total degree at most $m$, and each $L_n$ is an
isobaric polynomial of weight $n$ in the $c_m$, whence the first two
assertions.  On the other
side $\Ber_{n+1}\bigl(\frac{3+2\alpha}4\bigr)+\Ber_{n+1}\bigl(\frac{3-2\alpha}4\bigr)$
is even in $\alpha$ of degree at most $n+1$, hence a polynomial in
$A=\alpha^2$ of degree at most $\lfloor(n+1)/2\rfloor$.
\end{proof}

\begin{lemma}\label{lem:stirlingexp}
Put
\[
  \mathcal M(\zeta;\alpha)
  =\frac{\Gamma\bigl(\zeta+\frac{1-2\alpha}4\bigr)
         \Gamma\bigl(\zeta+\frac{1+2\alpha}4\bigr)}
        {\Gamma(\zeta)\,\Gamma\bigl(\zeta+\frac12\bigr)} .
\]
Then, for all $\alpha$ and $\zeta\to\infty$,
\[
  \log\mathcal M(\zeta;\alpha)
  \sim\sum_{n\ge1}\frac{\chi_n(\alpha)}{n(n+1)}\,\zeta^{-n},
\]
\[
  \chi_n(\alpha)=\Ber_{n+1}\bigl(\tfrac{3+2\alpha}4\bigr)
       +\Ber_{n+1}\bigl(\tfrac{3-2\alpha}4\bigr)
       -\Ber_{n+1}(1)-\Ber_{n+1}\bigl(\tfrac12\bigr).
\]
Moreover $\mathcal M_\ell$ of Theorem~\ref{thm:halfinteger} is the value of
the left-hand quotient at $\alpha=\ell+\frac12$.
\end{lemma}

\begin{proof}
The classical gamma-quotient expansion, obtained by subtracting
\cite[Eq.~5.11.8]{DLMF} at $h=a$ and $h=b$ (see also
\cite[Sec.~24.2]{DLMF} for the Bernoulli polynomials),

\[
  \log\frac{\Gamma(\zeta+a)}{\Gamma(\zeta+b)}-(a-b)\log\zeta
  \sim\sum_{n\ge1}\frac{(-1)^{n+1}}{n(n+1)}
      \bigl(\Ber_{n+1}(a)-\Ber_{n+1}(b)\bigr)\zeta^{-n}
\]
applies with $a=\frac{1\mp2\alpha}4$ and $b\in\{0,\frac12\}$; the logarithmic
terms cancel because $\frac{1-2\alpha}4+\frac{1+2\alpha}4=\frac12=0+\frac12$.
Using $(-1)^{n+1}\Ber_{n+1}(a)=\Ber_{n+1}(1-a)$ and
$1-\frac{1\mp2\alpha}4=\frac{3\pm2\alpha}4$, $1-0=1$, $1-\frac12=\frac12$
gives the stated form.  For the second assertion, at
$\alpha=\ell+\frac12$ one has $\frac{1-2\alpha}4=-\frac\ell2$ and
$\frac{1+2\alpha}4=\frac{\ell+1}2$.
\end{proof}

\begin{proof}[Proof of Theorem~\ref{thm:main}]
Fix $n$.  By Lemma~\ref{lem:degree} both sides of \eqref{eq:coef} are
polynomials in $A$ of degree at most $n$.  By Theorem~\ref{thm:halfinteger},
$F=\mathcal M_\ell$ at $\alpha=\ell+\frac12$ for every $\ell\ge0$; both sides
being rational functions of $Z$ there, their expansions at $Z=\infty$
coincide, and by Lemma~\ref{lem:stirlingexp} the coefficients of
$\log\mathcal M_\ell$ are the right-hand side of \eqref{eq:coef} evaluated at
$A=(\ell+\frac12)^2$.  Hence the two polynomials agree at the infinitely many
distinct points $A=(\ell+\frac12)^2$, $\ell=0,1,2,\dots$, and therefore
coincide.  This proves \eqref{eq:coef}.

For \eqref{eq:closed}, sum \eqref{eq:coef} with the generating function
$\sum_k\Ber_k(x)u^k/k!=ue^{xu}/(e^u-1)$ at $u=4t$.  The constant and linear
corrections cancel because
$\frac{3+2\alpha}4+\frac{3-2\alpha}4-1-\frac12=0$, and one obtains
\[
  \Upsilon_A(t)=\frac{e^{(3+2\alpha)t}+e^{(3-2\alpha)t}-e^{4t}-e^{2t}}{e^{4t}-1}
  =\frac{2e^{3t}\bigl(\cosh2\alpha t-\cosh t\bigr)}{2e^{2t}\sinh2t},
\]
which is \eqref{eq:closed}.
\end{proof}

\begin{proof}[Proof of Corollary~\ref{cor:contiguity}]
By Theorem~\ref{thm:main} and Lemma~\ref{lem:stirlingexp}, the formal
Laurent series $\log F$ on the diagonal is the Stirling series of
$\Gamma(\zeta+\frac{1-2\alpha}4)\Gamma(\zeta+\frac{1+2\alpha}4)/
\bigl(\Gamma(\zeta)\Gamma(\zeta+\frac12)\bigr)$ with $\zeta=Z/4$, whose ratio
at $\zeta\mapsto\zeta+1$ is
$\bigl(\zeta+\frac{1-2\alpha}4\bigr)\bigl(\zeta+\frac{1+2\alpha}4\bigr)/
\bigl(\zeta(\zeta+\frac12)\bigr)=\bigl((Z+1)^2-4A\bigr)/\bigl(Z(Z+2)\bigr)$.
Thus the ratio identity is first obtained coefficientwise in $Z^{-1}$.
Exponentiating is legitimate because both formal logarithms have zero
constant term, and gives the first displayed contiguity identity.  Multiplying
by the exact quotient
$\Gamma(Z+4)/\Gamma(Z)=Z(Z+1)(Z+2)(Z+3)$ gives the second identity.
\end{proof}

\begin{remark}\label{rem:borel}
Conversely, the contiguity of Corollary~\ref{cor:contiguity} determines
\eqref{eq:closed}.  For $f=\sum_{n\ge1}c_nZ^{-n}$ put
$\mathfrak Bf(t)=\sum_nc_nt^{n-1}/(n-1)!$; then
$\mathfrak B[f(Z+c)]=e^{-ct}\mathfrak Bf$ and
$\mathfrak B\bigl[\log\frac{Z+a}{Z+b}\bigr]=(e^{-bt}-e^{-at})/t$, both term by
term.  Applying $\mathfrak B$ to
$\log F(Z+4)-\log F(Z)=\log\frac{Z+1-2\alpha}Z+\log\frac{Z+1+2\alpha}{Z+2}$
gives
$(e^{-4t}-1)\Upsilon_A(t)/t=\bigl[1+e^{-2t}-2e^{-t}\cosh2\alpha t\bigr]/t$, and
multiplying numerator and denominator by $e^{2t}$ yields \eqref{eq:closed}.
\end{remark}

\begin{proof}[Proof of Corollaries~\ref{cor:specials} and~\ref{cor:R}]
(i) is $\cosh2\alpha t=\sum_k(4A)^kt^{2k}/(2k)!$ read off \eqref{eq:closed};
(ii) is $\cosh2\alpha t=\cosh t$ at $\alpha=\pm\frac12$; (iii) is
$\alpha=0$, where $\frac12(\operatorname{csch}t+\operatorname{sech}t)(1-\cosh t)
=\frac12(\operatorname{sech}t-\tanh\frac t2-1)$.  For
Corollary~\ref{cor:R}: exchanging $p$ and $q$, that is, complex
conjugation, exchanges the roles of $\alpha$ and $\beta$ in
\eqref{eq:F}, so $\overline{L_n(A,B)}=L_n(B,A)$ for real $A,B$.  Hence
$\operatorname{Re}g_n=\frac12(\partial_A+\partial_B)L_n|_{A=B=0}$, i.e.\ half the
derivative along the diagonal, so
$2\operatorname{Re}\Phi=\partial_A\Upsilon_A|_{A=0}$; and
$\partial_A\cosh2\alpha t|_{A=0}=2t^2$ gives
$\operatorname{Re}\Phi=t^2e^t/\sinh2t=\frac{t^2}2(\operatorname{csch}t
+\operatorname{sech}t)$.  Expanding
$\operatorname{sech}t=\sum_j\Eul_{2j}t^{2j}/(2j)!$ and
$\operatorname{csch}t=\frac1t+\sum_{k\ge1}2(1-2^{2k-1})\Ber_{2k}t^{2k-1}/(2k)!$
gives the coefficient formulas.  Finally Euler numbers are odd integers and, by von
Staudt--Clausen, $\nu_2(\Ber_{2k})=-1$ because $2-1=1$ divides $2k$; since
$1-2^{2k-1}$ is odd, $\nu_2(\operatorname{Re}g_n)=-1$ for every $n$.
\end{proof}

\section{The ladder off the midpoint}\label{sec:genladder}

The proof of Proposition~\ref{prop:step4} used the midpoint identity
$p^2+q^2=0$.  For arbitrary weights, the two recurrences of
Lemma~\ref{lem:bessel} give instead the following system, which couples
the two parities of $\ell$.

Fix arbitrary nonzero $p,q$ and set, for $\ell\ge0$ (with $y_{-1}=1$),
\[
  U=y_\ell(p\xi),\quad V=y_\ell(q\xi),\quad U_-=y_{\ell-1}(p\xi),\quad
  V_-=y_{\ell-1}(q\xi),\qquad \Pi_\ell(\xi)=UV,
\]
\[
  \Ct_\ell=\frac{U_-V}{p}+\frac{UV_-}{q},\qquad
  \Dt_\ell=\frac{UV_-}{p}+\frac{U_-V}{q},\qquad
  \sigma=\frac1p+\frac1q,\qquad \rho=\frac pq+\frac qp .
\]
At the midpoint $\sigma=1$, $\rho=0$, and $\Ct_\ell=C_\ell$ because $pq=2$.

\begin{theorem}\label{thm:genladder}
For every $\ell\ge0$ and arbitrary $p,q$,
\begin{align}
  &\text{\upshape(i)}\qquad
  \xi^2\Pi_\ell'+(\sigma-2\ell \xi)\,\Pi_\ell=\Ct_\ell,\label{eq:S3gen}\\
  &\text{\upshape(ii)}\qquad
  \Ct_{\ell+1}=\Dt_\ell+\rho\,(2\ell+1)\,\xi\,\Pi_\ell,\qquad
  \Dt_{\ell+1}=\Ct_\ell+2(2\ell+1)\,\xi\,\Pi_\ell,\label{eq:mixed}\\
  &\text{\upshape(iii)}\qquad
  \Ct_{\ell+2}=\Ct_\ell+2(2\ell+1)\,\xi\,\Pi_\ell
    +\rho\,(2\ell+3)\,\xi\,\Pi_{\ell+1}.\label{eq:S4gen}
\end{align}
\end{theorem}

\begin{proof}
(i)  Lemma~\ref{lem:bessel}(a) at $x=p\xi$ reads
$p\,\xi^2U'=(\ell p\xi-1)U+U_-$, and similarly at $x=q\xi$; hence
$\xi^2\Pi_\ell'=\bigl[(\ell p\xi-1)U+U_-\bigr]V/p+U\bigl[(\ell q\xi-1)V+V_-\bigr]/q
=(2\ell \xi-\sigma)\Pi_\ell+\Ct_\ell$.
(ii)  With $U_+=y_{\ell+1}(p\xi)$, $V_+=y_{\ell+1}(q\xi)$,
Lemma~\ref{lem:bessel}(b) gives $U_+=(2\ell+1)p\xi U+U_-$,
$V_+=(2\ell+1)q\xi V+V_-$, so
\[
  \Ct_{\ell+1}=\frac{UV_+}{p}+\frac{U_+V}{q}
  =(2\ell+1)\xi\Bigl(\frac qp+\frac pq\Bigr)UV+\frac{UV_-}{p}+\frac{U_-V}{q}
  =\rho(2\ell+1)\xi \Pi_\ell+\Dt_\ell,
\]
and likewise
$\Dt_{\ell+1}=2(2\ell+1)\xi \Pi_\ell+\Ct_\ell$.
(iii)  Compose the two identities of (ii).
\end{proof}

At the midpoint, \eqref{eq:S4gen} is Proposition~\ref{prop:step4}; away from
it, the correction is the single term $\rho(2\ell+3)\xi \Pi_{\ell+1}$, which
couples the two parities.

\section{The physical curve}\label{sec:curve}

\begin{lemma}\label{lem:curverel}
For a general angle $0<\theta<\pi$ the series \eqref{eq:F} is taken with
the weights $p=2/(1-e^{i\theta})$, $q=2/(1+e^{i\theta})$.  They satisfy
\[
  p+q=pq=\frac4{1-e^{2i\theta}}=:\varpi,
  \qquad\text{hence}\qquad
  \sigma=1,\qquad \rho=\varpi-2
\]
\emph{identically in} $\theta$, not only at the midpoint $\theta=\pi/2$ (where $\varpi=2$).
Equivalently, in the variables $x=p/2$, $y=q/2$, the weights lie on the
curve $x+y=2xy$, which we call the \emph{physical curve}; $x$ and $y$ are
the \emph{Horn variables}.
\end{lemma}

\begin{proof}
$p+q=2\bigl[(1+e^{i\theta})+(1-e^{i\theta})\bigr]/(1-e^{2i\theta})=4/(1-e^{2i\theta})=pq$.
\end{proof}

Because the translation to the basis $e_n$ of Lemma~\ref{lem:basis} requires
precisely $\sigma=1$, the analysis of Section~\ref{sec:half} for
half-odd-integer parameters applies on the whole curve.  For
$\alpha=\beta=\ell+\frac12$, $\Pi_\ell(\xi)=y_\ell(p\xi)y_\ell(q\xi)$ again
generates the coefficients of $F_\ell(Z)=\sum_nG^{(\ell)}_ne_n(Z)$.

\begin{theorem}\label{thm:curvecontig}
On the physical curve, for every $\ell\ge0$, as rational functions in
$\mathbb Q(\varpi)(Z)$,
\begin{equation}\label{eq:curvecontig}
  F_{\ell+2}(Z)\,(Z-2\ell-4)
  =F_\ell(Z)\,(Z+2\ell+2)+(\varpi-2)(2\ell+3)\,F_{\ell+1}(Z).
\end{equation}
\end{theorem}

\begin{proof}
By Theorem~\ref{thm:genladder}(i) with $\sigma=1$, comparing coefficients of
$\xi^{m+1}$, $G^{(\ell)}_{m+1}+(m-2\ell)G^{(\ell)}_m=[\xi^{m+1}]\Ct_{\ell}$ for
$m\ge0$; by Theorem~\ref{thm:genladder}(iii) with $\rho=\varpi-2$,
$[\xi^{m+1}]\Ct_{\ell+2}=[\xi^{m+1}]\Ct_{\ell}+2(2\ell+1)G^{(\ell)}_m
+(\varpi-2)(2\ell+3)G^{(\ell+1)}_m$.  Eliminating $\Ct$,
\[
  G^{(\ell+2)}_{m+1}+(m-2\ell-4)G^{(\ell+2)}_m
  =G^{(\ell)}_{m+1}+(m+2\ell+2)G^{(\ell)}_m
   +(\varpi-2)(2\ell+3)G^{(\ell+1)}_m\qquad(m\ge0),
\]
and the translation argument of Proposition~\ref{prop:Frec} (the case $m=0$
supplies the constants, and the coefficients of $Z$ agree because
$G_0^{(\ell)}=1$) yields \eqref{eq:curvecontig}.
\end{proof}

\begin{corollary}\label{cor:Prec}
For every $\ell\ge0$,
$F_\ell(Z)=P_\ell(Z;\varpi)/\bigl((Z-2)(Z-4)\cdots(Z-2\ell)\bigr)$
with $P_\ell\in\mathbb Z[\varpi][Z]$ monic of degree $\ell$ and
\begin{equation}\label{eq:Prec}
  P_0=1,\qquad P_1=Z+\varpi-2,\qquad
  P_{\ell+2}(Z)=\bigl(Z^2-(2\ell+2)^2\bigr)P_\ell(Z)
  +(\varpi-2)(2\ell+3)\,P_{\ell+1}(Z).
\end{equation}
At $\varpi=2$ this collapses to $P_{\ell+2}=(Z^2-(2\ell+2)^2)P_\ell$, whose
solutions are the numerators of the Gamma quotient of
Theorem~\ref{thm:halfinteger}:
$S_{2m}=\prod_{j=1}^m(Z^2-(4j-2)^2)$,
$S_{2m+1}=Z\prod_{j=1}^m(Z^2-16j^2)$, so that uniformly
$S_{\nu+2}=(Z^2-(2\nu+2)^2)S_\nu$.
\end{corollary}

\begin{proof}
$\Pi_1=(1+p\xi)(1+q\xi)=1+\varpi \xi+\varpi \xi^2$ gives
$F_1=1+\varpi e_1+\varpi e_2=(Z^2+(\varpi-3)Z+2-\varpi)/((Z-1)(Z-2))=(Z+\varpi-2)/(Z-2)$,
the numerator vanishing at $Z=1$.  Solving \eqref{eq:curvecontig} for
$F_{\ell+2}$ and clearing denominators gives
$P_{\ell+2}=(Z+2\ell+2)(Z-2\ell-2)P_\ell+(\varpi-2)(2\ell+3)P_{\ell+1}$ by
induction.
\end{proof}

The recurrence \eqref{eq:Prec} has a closed solution.  Write $\varpi_0=\varpi-2$ and
let
\[
  \mathfrak b_j(k)=\frac{k}{2j+k}\binom{2j+k}{j}\qquad(\mathfrak b_0(k)=1)
\]
be the ballot (Catalan-triangle) numbers.

\begin{theorem}\label{thm:ballot}
For every $\ell\ge0$, with $d=\ell-k$ in the summand indexed by $k$,
\begin{equation}\label{eq:ballot}
  P_\ell(Z;2+\varpi_0)=\sum_{k=0}^{\ell}(2k-1)!!\binom{\ell+k}{2k}\varpi_0^k
  \sum_{j\ge0}(-1)^j\,\frac{d!}{(d-2j)!}\;\mathfrak b_j(k)\,S_{d-2j}(Z),
\end{equation}
the inner sum running over $j\le d/2$.
\end{theorem}

\begin{proof}
The proof consists in verifying that the closed coefficients
\eqref{eq:ballot} satisfy the same recurrence and the same initial data
as the curve numerators: the interior identity is checked first, then
the boundary patterns excluded by the interior divisions, then the base
cases; induction closes the argument.
The polynomials $S_\nu$ have degree $\nu$, so the expansion of any polynomial in
the basis $\{\varpi_0^kS_\nu\}$ is unique.  Write
$c(\ell,k,j)=(2k-1)!!\binom{\ell+k}{2k}(-1)^j\frac{d!}{(d-2j)!}\mathfrak b_j(k)$
for the coefficient of $\varpi_0^kS_{\ell-k-2j}$ in \eqref{eq:ballot}, with
$c(\ell,k,j)=0$ outside $0\le k\le\ell$, $0\le2j\le d$.  Since
$Z^2S_\nu=S_{\nu+2}+(2\nu+2)^2S_\nu$, substituting \eqref{eq:ballot} into
\eqref{eq:Prec} and matching the coefficient of $\varpi_0^kS_\nu$ with
$\nu=(\ell+2)-k-2J$ reduces the claim, by uniqueness, to the four-term
identity
\begin{equation}\label{eq:fourterm}
  c(\ell+2,k,J)=c(\ell,k,J)-4(k+2J-2)(2\ell+4-k-2J)\,c(\ell,k,J-1)
  +(2\ell+3)\,c(\ell+1,k-1,J).
\end{equation}
We verify \eqref{eq:fourterm} for all indices, including the boundary
patterns.  First suppose $k\ge1$.  The coefficient admits the factorial form
\begin{equation}\label{eq:cfactorial}
 c(\ell,k,j)
 =(-1)^j
 \frac{(\ell+k)!\,(2j+k-1)!}
 {2^k(k-1)!\,j!\,(j+k)!\,(\ell-k-2j)!},
\end{equation}
whenever $\ell-k-2j\ge0$; outside this range it is zero by convention.
This follows immediately from
$(2k-1)!!=(2k)!/(2^kk!)$ and
$\mathfrak b_j(k)=k(2j+k-1)!/(j!(j+k)!)$.

For the interior range $J\ge1$ and $n:=\ell-k-2J\ge0$, division by the
nonzero coefficient $c(\ell,k,J)$ gives
\[
  \frac{c(\ell+2,k,J)}{c(\ell,k,J)}
  =\frac{(\ell+k+1)(\ell+k+2)}{(n+1)(n+2)},
\]
\[
  \frac{c(\ell+1,k-1,J)}{c(\ell,k,J)}
  =\frac{2(k-1)(J+k)}{(n+1)(n+2)(2J+k-1)},
\]
\[
  \frac{c(\ell,k,J-1)}{c(\ell,k,J)}
  =\frac{-\,J(J+k)}
  {(n+1)(n+2)(2J+k-2)(2J+k-1)}.
\]
Substitution in \eqref{eq:fourterm}, followed by clearing denominators,
reduces the recurrence to
\begin{align}\label{eq:cubic}
  (\ell+k+1)(\ell+k+2)(2J+k-1)
  ={}&(n+1)(n+2)(2J+k-1)\notag\\
  &+4J(J+k)(2\ell+4-k-2J)\notag\\
  &+2(2\ell+3)(k-1)(J+k),
\end{align}
and after $\ell=n+k+2J$ both sides expand to the same cubic polynomial.

It remains to treat the boundary cases excluded by that division.  If $k=0$, then
$c(\ell,0,0)=1$ and $c(\ell,0,j)=0$ for $j\ge1$.  Hence
\eqref{eq:fourterm} is immediate: for $J=0$ it is $1=1$, for $J=1$ the only
potential term $c(\ell,0,0)$ is multiplied by
$k+2J-2=0$, and for $J\ge2$ every term is zero.

Now let $k\ge1$ and $J=0$.  If $k\le\ell$, the middle term of
\eqref{eq:fourterm} vanishes and \eqref{eq:cfactorial} gives
\[
 \frac{c(\ell+2,k,0)}{c(\ell,k,0)}
 =\frac{(\ell+k+1)(\ell+k+2)}
        {(\ell-k+1)(\ell-k+2)},\qquad
 \frac{c(\ell+1,k-1,0)}{c(\ell,k,0)}
 =\frac{2k}{(\ell-k+1)(\ell-k+2)}.
\]
(For $k=1$ one uses $c(\ell+1,0,0)=1$ directly.)  Their substitution gives
$c(\ell+2,k,0)=c(\ell,k,0)+(2\ell+3)c(\ell+1,k-1,0)$ exactly.  The only
remaining $J=0$ targets have $k=\ell+1$ or $k=\ell+2$; then
$c(\ell,k,0)=0$ and the middle term is again zero, while direct evaluation
of the double-factorial formula gives
\[
 c(\ell+2,k,0)=(2\ell+3)c(\ell+1,k-1,0)=(2\ell+3)!!.
\]

Finally take $J\ge1$.  When the target coefficient
$c(\ell+2,k,J)$ exists but $c(\ell,k,J)$ does not, necessarily
$n=-2$ or $n=-1$; these are the only remaining possibilities because the
target basis index is $n+2\ge0$.  For $n=-2$,
\[
 \frac{c(\ell,k,J-1)}{c(\ell+2,k,J)}
 =-\frac{J}
 {2(2J+k-2)(2J+k-1)(2J+2k-1)},
\]
\[
 \frac{c(\ell+1,k-1,J)}{c(\ell+2,k,J)}
 =\frac{k-1}{(2J+k-1)(2J+2k-1)}.
\]
Since here $\ell=k+2J-2$, insertion in the right-hand side of
\eqref{eq:fourterm} gives exactly $1$ after division by
$c(\ell+2,k,J)$.  For $n=-1$ the corresponding ratios are
\[
 \frac{c(\ell,k,J-1)}{c(\ell+2,k,J)}
 =-\frac{J}
 {2(2J+k-2)(2J+k-1)(2J+2k+1)},
\]
\[
 \frac{c(\ell+1,k-1,J)}{c(\ell+2,k,J)}
 =\frac{k-1}{(2J+k-1)(2J+2k+1)},
\]
and with $\ell=k+2J-1$ the same substitution again gives $1$.  These formulas
also cover $k=1$, because then
$c(\ell+1,0,J)=0$ for $J\ge1$ and the factor $k-1$ vanishes.  Thus
\eqref{eq:fourterm} holds for every admissible coefficient.

With the base cases $P_0=1=c(0,0,0)S_0$ and
$P_1=Z+\varpi_0=S_1+\varpi_0S_0$, induction on $\ell$ through \eqref{eq:Prec} completes
the proof.
\end{proof}

\begin{remark}\label{rem:offcurve}
The curve $\sigma=1$ passes through the midpoint in the plane of the
Horn variables $(x,y)$.  On it, Theorem~\ref{thm:curvecontig}, the
degree bound of Lemma~\ref{lem:degree} (valid at any fixed $(x,y)$) and
interpolation at $A=(\ell+\frac12)^2$ determine the diagonal amplitude
for all $\alpha$.  The derivatives normal to the curve are not
determined in this way: for $\sigma\ne1$ the coefficient relations take
the form $\sigma G_{m+1}+(m-c)G_m$, and their translation to the basis
$e_m$ involves, in addition, the operation
$\sum_mG_me_m\mapsto\sum_mmG_me_m$, which is not expressible through
rational operations and shifts in $Z$.  This operation is studied in
Section~\ref{sec:indexop}.
\end{remark}

\section{The index operator and the derivatives normal to the curve}\label{sec:indexop}
By Remark~\ref{rem:offcurve}, extending the curve contiguity away from
$\sigma=1$ requires the multiplication of the coefficient $G_m$ by its
index $m$.  We show that this operation is a first-order differential
operator in the normal variable, and deduce a recurrence which
determines every derivative normal to the curve of the
half-odd-integer diagonal amplitudes.

For arbitrary $p,q$ with $pq\ne0$ and $p+q\ne0$, put
$\sigma=\frac1p+\frac1q$ and $\rho=\frac pq+\frac qp$, and let
\[
 F_\ell(Z;\sigma,\rho)=\sum_{m\ge0}G_m^{(\ell)}e_m(Z),
 \qquad e_m(Z)=\frac1{(Z-1)\cdots(Z-m)}.
\]
The two invariants recover the symmetric functions of $p,q$:
\begin{equation}\label{eq:pqrhosigma}
 pq=\frac{\rho+2}{\sigma^2},\qquad
 p+q=\frac{\rho+2}{\sigma}.
\end{equation}
Consequently $F_\ell$ is symmetric in $p,q$ and is a rational function of
$(Z,\sigma,\rho)$ for every integer $\ell\ge0$.

\begin{definitionx}
On an inverse-factorial series define the \emph{index operator}
\[
 \mathcal I\Bigl(\sum_{m\ge0}G_me_m\Bigr)=\sum_{m\ge0}mG_me_m.
\]
\end{definitionx}

\begin{lemma}\label{lem:Mclosure}
If every $G_m(x,y)$ is homogeneous of total degree $m$, then
$\mathcal IF=(\Theta_x+\Theta_y)F$, where $\Theta_x=x\partial_x$ and
$\Theta_y=y\partial_y$.
Moreover, if $G_0=1$ and
$\mathcal RF=\sum_{m\ge0}G_{m+1}e_m$, then
\begin{equation}\label{eq:Rclosure}
 \mathcal RF=Z(F-1)-\mathcal IF,
\end{equation}
and if $F$ is symmetric in $p,q$, hence a function of $(\sigma,\rho)$,
then at fixed $\rho$, $\mathcal I=-\sigma\partial_\sigma$.
\end{lemma}
\begin{proof}
Homogeneity gives the first assertion.  Lemma~\ref{lem:basis} gives
$ZF=Z+\sum_{m\ge1}G_me_{m-1}+\sum_{m\ge1}mG_me_m
=Z+\mathcal RF+\mathcal IF$, which is \eqref{eq:Rclosure}.  Under a common
scaling $(p,q)\mapsto(rp,rq)$, $\rho$ is fixed and
$\sigma\mapsto r^{-1}\sigma$; hence $\Theta_x+\Theta_y=-\sigma\partial_\sigma$
at fixed $\rho$.
\end{proof}

Define the first-order normal operator
\begin{equation}\label{eq:Qoperator}
 \mathcal Q_c^{(\sigma)}F
 :=\sigma Z(F-1)-\sigma(1-\sigma)\partial_\sigma F-cF
 =\sum_{m\ge0}\bigl(\sigma G_{m+1}+(m-c)G_m\bigr)e_m.
\end{equation}
Thus the operation of Remark~\ref{rem:offcurve} is a first-order
derivative in the coordinate $\sigma$.

\begin{theorem}\label{thm:Mladder}
For every $\ell\ge0$ and arbitrary $(\sigma,\rho)$,
\begin{equation}\label{eq:Mladder}
 \mathcal Q_{2\ell+4}^{(\sigma)}F_{\ell+2}
 =\mathcal Q_{-2\ell-2}^{(\sigma)}F_\ell
  +\rho(2\ell+3)F_{\ell+1}.
\end{equation}
At $\sigma=1$ this reduces to the curve contiguity
\eqref{eq:curvecontig}.
\end{theorem}
\begin{proof}
From Theorem~\ref{thm:genladder}(i), comparison of the coefficient of
$\xi^{m+1}$ gives
$\sigma G_{m+1}^{(\ell)}+(m-2\ell)G_m^{(\ell)}=[\xi^{m+1}]\Ct_\ell$.
Use Theorem~\ref{thm:genladder}(iii) to eliminate $\Ct_{\ell+2}$ and
$\Ct_\ell$.  One obtains, for every $m\ge0$,
\begin{align*}
 \sigma G_{m+1}^{(\ell+2)}+(m-2\ell-4)G_m^{(\ell+2)}
 ={}&\sigma G_{m+1}^{(\ell)}+(m+2\ell+2)G_m^{(\ell)}
 +\rho(2\ell+3)G_m^{(\ell+1)}.
\end{align*}
Multiplication by $e_m$ and summation give
\eqref{eq:Mladder}.  At $\sigma=1$, \eqref{eq:Qoperator} becomes
$\mathcal Q_c^{(1)}F=(Z-c)F-Z$; the two copies of $-Z$ cancel and
\eqref{eq:curvecontig} follows.
\end{proof}

At $\sigma=1$ the coefficient of $\partial_\sigma F$ in
\eqref{eq:Qoperator} vanishes.  Nevertheless all derivatives in $\sigma$
at $\sigma=1$ are determined recursively.

\begin{theorem}\label{thm:normaljets}
Put
$F_{\ell}^{[k]}(Z;\rho)
=\partial_\sigma^kF_\ell(Z;\sigma,\rho)\big|_{\sigma=1}$ for $k\ge0$,
with $F_{\ell}^{[-1]}=0$.  Then
\begin{align}\label{eq:jetladder}
 &(Z+k-2\ell-4)F_{\ell+2}^{[k]}
   +k(Z+k-1)F_{\ell+2}^{[k-1]}\notag\\
 &\qquad=(Z+k+2\ell+2)F_{\ell}^{[k]}
   +k(Z+k-1)F_{\ell}^{[k-1]}
   +\rho(2\ell+3)F_{\ell+1}^{[k]}.
\end{align}
In particular, at the midpoint $\rho=0$, the two parity ladders decouple
and every $F_\ell^{[k]}$ is obtained recursively from the two base families
$F_0$ and $F_1$.
\end{theorem}
\begin{proof}
For $k\ge1$, Leibniz' rule applied to \eqref{eq:Qoperator} gives
\begin{equation}\label{eq:Qjet}
 \partial_\sigma^k\mathcal Q_c^{(\sigma)}F\big|_{\sigma=1}
 =(Z+k-c)F^{[k]}+k(Z+k-1)F^{[k-1]}-\delta_{k,1}Z:
\end{equation}
only the first two derivatives of $\sigma(1-\sigma)$ are nonzero, and
their contributions combine with the derivative of $\sigma Z(F-1)$ exactly
as displayed.  For $k=0$ the formula is
$\mathcal Q_c^{(1)}F=(Z-c)F^{[0]}-Z$.  For $k\le1$ the terms $-Z$,
which do not depend on $c$, cancel between the two
sides of \eqref{eq:Mladder}.  Differentiating \eqref{eq:Mladder} $k$ times
and using \eqref{eq:Qjet} proves \eqref{eq:jetladder}.
\end{proof}

The base data are explicit: from $\Pi_0=1$,
$\Pi_1=(1+p\xi)(1+q\xi)$ and \eqref{eq:pqrhosigma},
\begin{equation}\label{eq:jetbases}
 F_0=1,
 \qquad
 F_1=1+\frac{\rho+2}{\sigma}\,e_1
       +\frac{\rho+2}{\sigma^2}\,e_2,
\end{equation}
hence $F_{0}^{[k]}=0$ for $k\ge1$ and
\begin{equation}\label{eq:F1jets}
 F_{1}^{[k]}=(-1)^kk!\,(\rho+2)\bigl(e_1+(k+1)e_2\bigr)
 \qquad(k\ge1).
\end{equation}
Equations \eqref{eq:jetladder}--\eqref{eq:F1jets} therefore determine every
$F_\ell^{[k]}$ recursively.

\begin{remark}\label{rem:indexop}
For each fixed order in $Z^{-1}$ and each fixed $k$, the interpolation
argument of Section~\ref{sec:transfer} extends \eqref{eq:jetladder},
written with $2\ell+1=2\alpha$ and with $F_\ell,F_{\ell+1},F_{\ell+2}$
replaced by the diagonal amplitudes at $\alpha,\alpha+1,\alpha+2$, from
$\alpha=\ell+\tfrac12$ to arbitrary $\alpha$, each coefficient being a
polynomial in $\alpha$.  Thus the diagonal
amplitude is formally determined, to any prescribed order in
$\sigma-1$, in a neighbourhood of the physical curve.
\end{remark}

\section{Off the diagonal: the transverse variation and the dyadic law}\label{sec:offdiag}

This section contains the proofs of Theorems~\ref{thm:imclosed}
and~\ref{thm:dyadic}.  We begin with two preliminary results for parameters
off the diagonal: a recurrence for independent half-odd-integer
parameters, and the identification of the amplitude~\eqref{eq:F} with a
double hypergeometric series.

Take \emph{independent} half-odd-integer parameters
$\alpha=\ell+\frac12$, $\beta=m+\frac12$ at the midpoint $p=1+i$, $q=1-i$:
both Hankel series terminate and $\Pi_{\ell,m}(\xi)=y_\ell(p\xi)y_m(q\xi)$
generates the coefficients of the rational function
$F_{\ell,m}(Z)=\sum_nG^{(\ell,m)}_ne_n(Z)$.  Set
$\Ct_{\ell,m}=y_{\ell-1}(p\xi)y_m(q\xi)/p+y_\ell(p\xi)y_{m-1}(q\xi)/q$ and
$\Dt_{\ell,m}=y_\ell(p\xi)y_{m-1}(q\xi)/p+y_{\ell-1}(p\xi)y_m(q\xi)/q$.

\begin{theorem}\label{thm:twopar}
For all $\ell,m\ge0$, at the midpoint,
\begin{align}
 &\xi^2\Pi_{\ell,m}'+\bigl(1-(\ell+m)\xi\bigr)\Pi_{\ell,m}=\Ct_{\ell,m}, \label{eq:2parS3}\\
&\Ct_{\ell+1,m+1}=\Dt_{\ell,m}+2i(\ell-m)\,\xi\,\Pi_{\ell,m},\qquad  \Dt_{\ell+1,m+1}=\Ct_{\ell,m}+2(\ell+m+1)\,\xi\,\Pi_{\ell,m},  \label{eq:2parmixed}\\
 &F_{\ell+2,m+2}(Z)\,\bigl(Z-(\ell+m)-4\bigr) \label{eq:2parcontig} \\
 &\qquad =F_{\ell,m}(Z)\,\bigl(Z+(\ell+m)+2\bigr)+2i(\ell-m)\,F_{\ell+1,m+1}(Z).   \nonumber
 \end{align}
\end{theorem}

\begin{proof}
\eqref{eq:2parS3} is the computation of Theorem~\ref{thm:genladder}(i) with the two
orders kept distinct, using $\sigma=1$.  For \eqref{eq:2parmixed}, with $U=y_\ell(p\xi)$,
$V=y_m(q\xi)$,
\[
  \Ct_{\ell+1,m+1}=\frac{UV_+}{p}+\frac{U_+V}{q}
  =\xi\Bigl[(2m+1)\frac qp+(2\ell+1)\frac pq\Bigr]\Pi_{\ell,m}+\Dt_{\ell,m},
\]
and at the midpoint $q/p=-i$, $p/q=i$, so the bracket is $2i(\ell-m)$; the
second identity has bracket $(2\ell+1)+(2m+1)$.  Composing gives
$\Ct_{\ell+2,m+2}=\Ct_{\ell,m}+2(\ell+m+1)\xi \Pi_{\ell,m}
+2i(\ell-m)\xi \Pi_{\ell+1,m+1}$, and the translation of
Theorem~\ref{thm:curvecontig}, with $\ell+m$ in place of $2\ell$, yields
\eqref{eq:2parcontig}; the constants match because
$G^{(\ell,m)}_1=\frac{\ell(\ell+1)}2p+\frac{m(m+1)}2q$ gives
$G^{(\ell+2,m+2)}_1-G^{(\ell,m)}_1
=\bigl[(\ell+m)+2\bigr]+\bigl[(\ell+m)+4\bigr]+2i(\ell-m)$.
\end{proof}

On the diagonal $\ell=m$ the inhomogeneity vanishes and
\eqref{eq:2parcontig} is Proposition~\ref{prop:Frec}.  Off the diagonal the
transverse direction enters as the explicit term $2i(\ell-m)$:
antisymmetric, purely imaginary, of constant coefficient along the diagonal
step.  Together with the two boundary strips $m=0$ and $m=1$, the recurrence determines
every $F_{\ell,m}$ by induction on $\min(\ell,m)$, using the midpoint
conjugation symmetry
\[
  F_{\ell,m}(Z)=\overline{F_{m,\ell}(\overline Z)}
\]
(and hence $F_{\ell,m}(Z)=\overline{F_{m,\ell}(Z)}$ for real $Z$) whenever
the smaller of the two indices is $\ell$ rather than $m$.  More explicitly,
if $m\ge2$, equation~\eqref{eq:2parcontig} expresses $F_{\ell,m}$ in terms
of $F_{\ell-2,m-2}$ and $F_{\ell-1,m-1}$ after shifting the two indices down;
repeating reaches $m=0$ or $m=1$.  If $\ell<m$ one first interchanges the
indices by the displayed conjugation symmetry and applies the same argument.
The strip $m=0$ is the terminating Gauss family of Lemma~\ref{lem:KdF}; the
strip $m=1$ is equally explicit from $y_1(q\xi)=1+q\xi$ and the finite defining
sum.  Thus the two-step recurrence, together with these two strips, determines
every $F_{\ell,m}$.

\begin{lemma}\label{lem:KdF}
For every $\ell\ge0$,
$F_{\ell,0}(Z)=\Fhyp\bigl(-\ell,\ell+1;1-Z;x_0\bigr)$ with
$x_0=\frac{1+i}2$.  More generally, for all $\alpha,\beta$, as formal
series,
\begin{equation}\label{eq:KdF}
  F(\alpha,\beta;Z)=\sum_{\mu,\nu\ge0}
  \frac{(\tfrac12-\alpha)_\mu(\tfrac12+\alpha)_\mu
        (\tfrac12-\beta)_\nu(\tfrac12+\beta)_\nu}
       {(1-Z)_{\mu+\nu}\;\mu!\,\nu!}\;x_0^{\,\mu}y_0^{\,\nu},
\end{equation}
a Kamp\'e de F\'eriet function $F^{0:2;2}_{1:0;0}$ at arguments with
$x_0+y_0=1$, $x_0y_0=\frac12$.
\end{lemma}

\begin{proof}
$e_n(Z)=(-1)^n/(1-Z)_n$ and
$a_\mu(\alpha)(2x_0)^\mu(-1)^\mu
=(\tfrac12-\alpha)_\mu(\tfrac12+\alpha)_\mu x_0^\mu/\mu!$, term by term in
\eqref{eq:F}.  At $\alpha=\ell+\frac12$, $\beta=\frac12$, only $\nu=0$
survives and the $\mu$-series terminates.
\end{proof}

\subsection{\texorpdfstring{The Appell $F_3$ specialization on the physical curve}{The Appell F3 specialization on the physical curve}}\label{sec:appell}

\paragraph{Branch convention.}  All multivalued functions in this and
the following subsections are continued from a common germ at the
origin of the path described below: the Appell/Kamp\'e de F\'eriet series at
$(x,y)$ small, the reduced Gauss function through the argument path
$4x(\lambda)(1-x(\lambda))=2\lambda+2i\lambda(1-\lambda)$, which stays in the \emph{upper} half-plane
and reaches $2+i0$; every boundary value written
${}_2F_1^{\uparrow}(\cdots;2)$, $\log(1-z)\to-i\pi$, $2^+$ or $-1-i0$
is the one obtained along this path, and the lower determinations
(denoted $\downarrow$) are their Schwarz reflections.  Replacing the
upper boundary value by the lower one changes the functions; no identity
below mixes the two determinations except where both appear explicitly,
as in Appendix~\ref{app:sub-relation}.

The Kamp\'e de F\'eriet notation in Lemma~\ref{lem:KdF} hides an additional
one-variable structure.  We make it explicit because it removes one of the
two Horn variables before any asymptotic or $2$-adic argument is made.
For independent variables $x,y$ write
\[
 \mathcal F(\alpha,\beta;Z;x,y)
 :=F_3\!\left(\begin{matrix}
  \frac12-\alpha,\frac12-\beta;
  \frac12+\alpha,\frac12+\beta\\[1mm]
  1-Z
 \end{matrix}\middle|x,y\right),
\]
where the convention is that the first and third upper parameters belong to
$x$, and the second and fourth to $y$.  Thus \eqref{eq:KdF} is precisely
$F(\alpha,\beta;Z)=\mathcal F(\alpha,\beta;Z;x_0,y_0)$.

\begin{theorem}\label{thm:appellreduce}
Put $c=1-Z$, $b_1=\frac12+\alpha$, $b_2=\frac12+\beta$.  On
\[
       y=\frac{x}{2x-1},
       \qquad\text{equivalently}\qquad x+y=2xy,
\]
one has, for generic parameters and then by meromorphic continuation,
\begin{align}
 \mathcal F\!\left(\alpha,\beta;Z;x,\frac{x}{2x-1}\right)
  ={}&(1-x)^{c-1}(1-2x)^{b_2}\notag\\
 &\times {}_2F_1\!\left(
  \frac{b_2-b_1+c}{2},
  \frac{b_1+b_2+c-1}{2};c;4x(1-x)\right).
 \label{eq:appellreduce}
\end{align}
The branches are those obtained by analytic continuation from a neighbourhood
of $x=0$ along the chosen path.
\end{theorem}

\begin{proof}
The four upper parameters of our $F_3$ are
$1-b_1,1-b_2;b_1,b_2$.  The fourth pair in Theorem~3.1 of
Vid\=unas~\cite{VidunasAppell} is therefore exactly the pair displayed in
\eqref{eq:appellreduce}.  Immediately after that theorem, Vid\=unas states
that in the first and fourth pairs the two functions are directly identical
in a neighbourhood of $x=0$ (not merely solutions of the same second-order
ODE).  Hence \eqref{eq:appellreduce} holds on the common germ at $x=0$.
The relation $y=x/(2x-1)$ is the physical-curve identity already used in
Section~\ref{sec:curve}.  Therefore the two sides represent the same analytic
germ near $x=0$.  Analytic continuation of this common germ along any path
avoiding the singular locus gives the stated identity.  Notice that this is
an identity for fixed generic $Z$ (and then meromorphically in the
parameters); no passage to the large-$Z$ limit is used in establishing it.
\end{proof}

At the midpoint,
\[
 x_0=\frac{1+i}{2},\qquad y_0=\frac{1-i}{2},\qquad
 y_0=\frac{x_0}{2x_0-1},\qquad 4x_0(1-x_0)=2.
\]
For the straight path $x(\lambda)=\lambda x_0$, $0<\lambda<1$, one has
\[
 y(\lambda)=\frac{x(\lambda)}{2x(\lambda)-1},\qquad
 4x(\lambda)(1-x(\lambda))=2\lambda+2i\lambda(1-\lambda).
\]
Thus the transformed Gauss argument stays strictly in the upper half-plane
for $0<\lambda<1$ and tends to $2+i0$ as $\lambda\uparrow1$.  In addition,
$|x(\lambda)|=\lambda/\sqrt2<1$ and
\[
 |y(\lambda)|^2=\frac{\lambda^2}{2(2\lambda^2-2\lambda+1)}\le\frac12,
\]
so the defining Appell double series itself stays inside its bidisc of
absolute convergence along the whole path.  The path has $x,y\ne0$, meets
neither $x=1$ nor $y=1$, and the additional finite Appell singular relation
$xy=x+y$, together with $x+y=2xy$, would force $xy=0$.  Hence continuation
may start at any sufficiently small $\lambda>0$ inside the common germ and reaches
the midpoint without meeting the finite singular locus.  The only
nontrivial continuation issue in the reduced expression is therefore the
Gauss boundary value at its standard cut.  If
${}_2F_1^{\uparrow}(a,b;c;2)$ denotes this upper boundary value, then
analytic continuation of the common germ gives
\begin{equation}\label{eq:midpoint2F1}
 F(\alpha,\beta;Z)
 =y_0^{-Z}(-i)^{\frac12+\beta}
 {}_2F_1^{\uparrow}\!\left(
 \frac{1-Z+\beta-\alpha}{2},
 \frac{1-Z+\alpha+\beta}{2};1-Z;2\right).
\end{equation}
This branch convention is essential: the two boundary values on the cut
$[1,\infty)$ need not agree.

\begin{corollary}
\label{cor:gausstangent}
Set $s=(1-Z)/2$ and
\[
 \mathcal H(\alpha;Z)
 ={}_2F_1^{\uparrow}\!\left(s-\frac\alpha2,
                                  s+\frac\alpha2;2s;2\right).
\]
Then the inverse-power expansion at $Z=\infty$ inherited coefficientwise
from the defining inverse-factorial series for $F$ satisfies
\begin{equation}\label{eq:gausstangent}
 \sum_{n\ge1}g_n Z^{-n}
 =\frac12\left.\frac{\partial^2}{\partial\alpha^2}
       \log\mathcal H(\alpha;Z)\right|_{\alpha=0}
 =\frac12\,\frac{\partial_\alpha^2\mathcal H(0;Z)}{\mathcal H(0;Z)}.
\end{equation}
The last expression is independent of a choice of logarithm.  In particular
the two numerator parameters in \eqref{eq:gausstangent} sum to the
denominator parameter: the series $\sum_ng_nZ^{-n}$ is obtained from a
parameter variation of a single \emph{zero-balanced} Gauss function.
\end{corollary}

\begin{proof}
Take $\beta=0$ in \eqref{eq:midpoint2F1}.  Its prefactor is independent of
$\alpha$.  Since $A=\alpha^2$ and the left-hand side is even in $\alpha$,
$\partial_A|_{A=0}=\frac12\partial_\alpha^2|_{\alpha=0}$.  The function
$\mathcal H$ is itself even in $\alpha$, so $\partial_\alpha\mathcal H(0;Z)=0$;
therefore the second logarithmic derivative is
$\partial_\alpha^2\mathcal H(0;Z)/\mathcal H(0;Z)$.  Finally, $F=1+O(Z^{-1})$ coefficientwise, so its formal logarithm at
$Z=\infty$ is uniquely defined.  Since \eqref{eq:midpoint2F1} is an exact
identity for generic fixed $Z$, the formal inverse-power expansion used in
this paper may be taken on either side; the coefficients inherited from the
left-hand inverse-factorial series are therefore exactly those in
\eqref{eq:gausstangent}.  No assertion about convergence of the resulting
$Z^{-1}$ series is needed.
\end{proof}

There is also a differential form of the same reduction.  With
$\Theta_x=x\partial_x$, $\Theta_y=y\partial_y$, the defining coefficients
of $F_3$ give
\begin{align}
 \mathcal L_x^{(A)}\mathcal F&=0,&
 \mathcal L_x^{(A)}&=
 \Theta_x(\Theta_x+\Theta_y-Z)
 -x\bigl((\Theta_x+\tfrac12)^2-A\bigr),\label{eq:appellPDE-x}\\
 \mathcal L_y^{(B)}\mathcal F&=0,&
 \mathcal L_y^{(B)}&=
 \Theta_y(\Theta_x+\Theta_y-Z)
 -y\bigl((\Theta_y+\tfrac12)^2-B\bigr).\label{eq:appellPDE-y}
\end{align}
Thus the parameters enter the Appell system \emph{affinely}.  Put
\[
  A_+=\frac{A+B}2,\qquad A_-=\frac{A-B}2,
\]
and let $\mathcal F_0=\mathcal F|_{A_-=0}$ and
$\mathcal F_1=\partial_{A_-}\mathcal F|_{A_-=0}$.  Differentiating
\eqref{eq:appellPDE-x}--\eqref{eq:appellPDE-y} yields the forced
system
\begin{equation}\label{eq:forcedAppell}
  \mathcal L_x^{(A_+)}\mathcal F_1=-x\mathcal F_0,
  \qquad
  \mathcal L_y^{(A_+)}\mathcal F_1=+y\mathcal F_0,
  \qquad \mathcal F_1(y,x)=-\mathcal F_1(x,y).
\end{equation}
At the midpoint $y_0=\overline{x_0}$, coefficientwise conjugation and
exchange of the two variables give
$\overline{\mathcal F(A_+,A_-;x_0,y_0)}=\mathcal F(A_+,-A_-;x_0,y_0)$ for
real $Z$ and real $A_\pm$.  Consequently, at $A_-=0$, $\mathcal F_0$ and
$\partial_{A_+}\mathcal F_0$ are real, whereas $\mathcal F_1$ is purely
imaginary.

\begin{remark}\label{rem:nodes}
$\kappa_r$ is a polynomial in $(A,B)$; its values on the grid
$A_\ell=(\ell+\frac12)^2$, $B_m=(m+\frac12)^2$ are determined by the
rational functions $F_{\ell,m}$, and the node differences
$A_\ell-A_m=(\ell-m)(\ell+m+1)$ are integers.  The recurrence
\eqref{eq:2parcontig} therefore expresses $[A^1B^0]\kappa_r$ as a
Lagrange interpolation functional with integer node differences.
\end{remark}

\subsection{The boundary split and the transverse quadrature}
\label{sec:split}

By Corollary~\ref{cor:gausstangent}, the series $\sum_ng_nZ^{-n}$ is
obtained from the one-parameter family of boundary values
$\mathcal H(\alpha;Z)$.  We now split this boundary value into two
classical constituents and compute its second variation in $\alpha$ by
quadratures.  Throughout, $s=(1-Z)/2$,
$a=s-\frac\alpha2$, $b=s+\frac\alpha2$, and $Z>1$ is real, so $s<0$; the
degenerate lattice $2s\in\mathbb Z_{\le0}$ is excluded.

\begin{lemma}
\label{lem:quadeval}
\[
 \mathcal H(0;Z)={}_2F_1^{\uparrow}(s,s;2s;2)
 =e^{i\pi s/2}\,\frac{\sqrt\pi\,\Gamma(s+\frac12)}{\Gamma(\frac{s+1}2)^2}.
\]
In particular
$\operatorname{Im}\mathcal H(0;Z)/\operatorname{Re}\mathcal H(0;Z)
=\tan\frac{\pi s}2$.
\end{lemma}

\begin{proof}
The quadratic transformation \cite[Sec.~3.1]{AndrewsAskeyRoy}
\[
 {}_2F_1(a,b;2b;z)
 =(1-z)^{-a/2}\,{}_2F_1\Bigl(\frac a2,\,b-\frac a2;\,b+\frac12;\,
 \frac{z^2}{4(z-1)}\Bigr)
\]
at $a=b=s$ has transformed argument
$z^2/(4(z-1))\to1$ as $z\to2$, and the right-hand Gauss series is
continuous up to argument $1$ because there
$\operatorname{Re}(c-a-b)=\frac12>0$; Gauss's summation gives
${}_2F_1(\frac s2,\frac s2;s+\frac12;1)
=\Gamma(s+\frac12)\Gamma(\frac12)/\Gamma(\frac{s+1}2)^2$.
Approaching from $\operatorname{Im}z>0$, $\log(1-z)\to-i\pi$, so
$(1-z)^{-s/2}\to e^{i\pi s/2}$.
\end{proof}

\begin{proposition}\label{prop:split}
For $\operatorname{Re}s<\frac12$ and $2s\notin\mathbb Z$,
\begin{equation}\label{eq:split}
 \mathcal H(\alpha;Z)
 =\frac{\Gamma(2s)}{\Gamma(a)\Gamma(b)}
 \bigl[\mathsf S(\alpha;s)+i\pi\,\mathsf K(\alpha;s)\bigr],
\end{equation}
where
\[
 \mathsf K(\alpha;s)={}_2F_1(a,b;1;-1),\qquad
 \mathsf S(\alpha;s)=\sum_{n\ge0}(-1)^n\frac{(a)_n(b)_n}{(n!)^2}
 \bigl[2\psi(n+1)-\psi(a+n)-\psi(b+n)\bigr],
\]
both series converging absolutely.  In particular, for real $\alpha$ the
imaginary part of the boundary value is
$\pi\Gamma(2s)\,{}_2F_1(a,b;1;-1)/(\Gamma(a)\Gamma(b))$.
At the Legendre point $\alpha=0$ both blocks have closed forms:
\begin{equation}\label{eq:blocksclosed}
 \mathsf K(0;s)=\frac{\Gamma(1+\frac s2)}{\Gamma(1+s)\,\Gamma(1-\frac s2)}
 \quad\text{\emph{(Kummer)}},\qquad
 \mathsf S(0;s)=\cos\frac{\pi s}2\cdot
 \frac{\sqrt\pi\,\Gamma(s+\frac12)\,\Gamma(s)^2}
      {\Gamma(2s)\,\Gamma(\frac{s+1}2)^2}.
\end{equation}
\end{proposition}

\begin{proof}
Write $R_n(a,b)=2\psi(n+1)-\psi(a+n)-\psi(b+n)$.  The zero-balanced
connection formula
\cite[Sec.~2.3]{AndrewsAskeyRoy} (see also \cite[Eq.~15.8.10]{DLMF})
gives, for $|1-z|<1$,
$|\arg(1-z)|<\pi$,
\[
 {}_2F_1(a,b;a+b;z)
 =\frac{\Gamma(a+b)}{\Gamma(a)\Gamma(b)}
 \sum_{n\ge0}\frac{(a)_n(b)_n}{(n!)^2}
 \bigl[R_n(a,b)-\log(1-z)\bigr](1-z)^n .
\]
On the upper approach $z\to2+i0$ one has $1-z\to-1-i0$, so
$\log(1-z)\to-i\pi$ and $(1-z)^n\to(-1)^n$; the two limiting series are
absolutely convergent for $\operatorname{Re}(a+b)<1$ since their terms are
$O(n^{\operatorname{Re}(a+b)-2}\log n)$, and the boundary value is attained
by Abel continuity along the path.  The digamma-free part is the Taylor
series of ${}_2F_1(a,b;1;-1)$ because $(1)_n=n!$.  The first evaluation
in \eqref{eq:blocksclosed} is Kummer's theorem
${}_2F_1(a,b;1+a-b;-1)$ at $a=b=s$; the second follows from
Lemma~\ref{lem:quadeval} by taking real parts in \eqref{eq:split}.
The $n=0$ kernel $2\psi(1)-\psi(a)-\psi(b)$ is precisely the constant
$R(a,b)=-2\gamma-\psi(a)-\psi(b)$ governing the zero-balanced asymptotics
studied in \cite{AndersonEtAl,SimicVuorinen}.
\end{proof}

\begin{lemma}
\label{lem:Kcontig}
For all $\alpha$ and generic $s$,
\begin{equation}\label{eq:Kcontig}
 4(s-1)(4s^2-\alpha^2)\,\mathsf K(\alpha;s+1)
 +2\alpha^2(2s-1)\,\mathsf K(\alpha;s)
 +s\bigl(4(s-1)^2-\alpha^2\bigr)\,\mathsf K(\alpha;s-1)=0 .
\end{equation}
At $\alpha=0$ this degenerates to the two-term relation
$4s\,\mathsf K(0;s+1)+(s-1)\,\mathsf K(0;s-1)=0$, which is also satisfied by
the closed form of $\mathsf K(0;s)$ in \eqref{eq:blocksclosed}.
\end{lemma}

\begin{proof}
Assume first that $\operatorname{Re}s<-\frac12$, so that the three
series converge absolutely.  With $t_n=(-1)^n(a)_n(b)_n/(n!)^2$, the shifted
series have termwise ratios
$t_n^{+}=t_n(a+n)(b+n)/(ab)$ and $t_n^{-}=t_n(a-1)(b-1)/((a-1+n)(b-1+n))$,
and the rational function
\begin{gather*}
 r(n)=\frac{-n^2\,\mathsf P(n)}{(a-1+n)(b-1+n)},\\
 \mathsf P(n)=8(s-1)n^2+8(3s-2)(s-1)n+4(5s-2)(s-1)^2+(2-s)\alpha^2,
\end{gather*}
satisfies the identity
\[
 c_+\,\frac{(a+n)(b+n)}{ab}+c_0
 +c_-\,\frac{(a-1)(b-1)}{(a-1+n)(b-1+n)}
 =-\frac{(a+n)(b+n)}{(n+1)^2}\,r(n+1)-r(n)
\]
where $c_+,c_0,c_-$ denote the three coefficients of \eqref{eq:Kcontig}, in
this order (an identity of rational functions of $(n,s,\alpha)$).
Multiplying by $t_n$ and summing over $n\ge0$, the right side telescopes:
$r(0)=0$, and $t_nr(n)=O(n^{2\operatorname{Re}s})\to0$.  This proves
\eqref{eq:Kcontig} for $\operatorname{Re}s<-\frac12$, and the general
case follows by analytic continuation in $s$, since
${}_2F_1(a,b;1;-1)$ is an entire function of $a$ and $b$.  The $\alpha=0$
degeneration follows since the middle coefficient carries $\alpha^2$;
that the closed form of $\mathsf K(0;s)$ satisfies it is the elementary computation
$\mathsf K(0;s+1)/\mathsf K(0;s-1)=-(s-1)/(4s)$.
\end{proof}

\begin{theorem}\label{thm:quad}
Let
\[
 \mathsf K_2(s):=\partial_\alpha^2\mathsf K(\alpha;s)\big|_{\alpha=0}
 =-\frac12\sum_{n\ge1}(-1)^n\Bigl(\frac{(s)_n}{n!}\Bigr)^{\!2}
 \sum_{j=0}^{n-1}\frac1{(s+j)^2},
\]
a series with rational terms.  Then:
\begin{enumerate}
\item[(i)] $\mathsf K_2$ satisfies the inhomogeneous two-term recurrence
\begin{equation}\label{eq:K2rec}
 8s^2(s-1)\,\mathsf K_2(s+1)+2s(s-1)^2\,\mathsf K_2(s-1)
 =4(s-1)\,\mathsf K_0(s+1)-2(2s-1)\,\mathsf K_0(s)+s\,\mathsf K_0(s-1),
\end{equation}
with $\mathsf K_0=\mathsf K(0;\cdot)$.
Equivalently, the normalised quotient
$Q=\mathsf K_2/\mathsf K_0$ obeys the exact difference equation
\begin{equation}\label{eq:Qdiff}
 Q(s+1)-Q(s-1)=\mathsf f(s),\qquad
 \mathsf f(s)=-\frac{2s-1}{2s^2(s-1)^2}
 -\frac{(2s-1)\,\mathsf r(s)}{4s^2(s-1)},
\end{equation}
where
$\mathsf r(s)=\mathsf K_0(s)/\mathsf K_0(s+1)
=2\,\Gamma(1+\frac s2)\Gamma(\frac{1-s}2)\big/
 \bigl(\Gamma(1-\frac s2)\Gamma(\frac{1+s}2)\bigr)$.
\item[(ii)] On each lattice $\{s_0-2k\}_{k\ge0}$, $\mathsf r$ is bounded
and $\mathsf f(\varsigma)=O(\varsigma^{-2})$, so
$\sum_{k\ge0}\mathsf f(s-1-2k)$ converges absolutely and
\begin{equation}\label{eq:quadrature}
 \mathsf K_2(s)=\mathsf K_0(s)
 \Bigl[c_*(s_0)+\sum_{k\ge0}\mathsf f(s-1-2k)\Bigr]
\end{equation}
with a constant $c_*(s_0)$ depending only on the lattice.  This constant
is in fact zero (Theorem~\ref{thm:czero}).
\end{enumerate}
\end{theorem}

\begin{proof}
Since $(a)_n(b)_n=\prod_{j=0}^{n-1}\bigl((s+j)^2-\alpha^2/4\bigr)$,
one has
\[
 \partial_\alpha^2\bigl[(a)_n(b)_n\bigr]_{\alpha=0}
 =-\frac12(s)_n^2\sum_{j=0}^{n-1}\frac1{(s+j)^2},
\]
which gives the series for $\mathsf K_2$.  (i)~Differentiate \eqref{eq:Kcontig} twice at $\alpha=0$: the
first $\alpha$-derivatives vanish by evenness, and the $\alpha^2$-parts of
the three coefficients contribute
$-8(s-1)$, $4(2s-1)$, $-2s$ against $\mathsf K_0$, giving
\eqref{eq:K2rec} after division by $2$.  Dividing \eqref{eq:K2rec} by
$8s^2(s-1)\mathsf K_0(s+1)$ and eliminating $\mathsf K_0(s-1)$ through
the two-term relation of Lemma~\ref{lem:Kcontig} yields \eqref{eq:Qdiff}.
(ii)~By reflection and Stirling,
$\Gamma(\frac{1-s}2)/\Gamma(1-\frac s2)\sim(|s|/2)^{-1/2}$ while the
quotient of the two left-hand Gammas is
$O(|s|^{1/2})$ times a trigonometric factor that is periodic on the
lattice, so $\mathsf r$ is bounded there; summing \eqref{eq:Qdiff} in
steps of two telescopes to \eqref{eq:quadrature}, the general solution of
the homogeneous equation on a fixed lattice being
$\mathsf K_0\times\mathrm{const}$.
\end{proof}

\begin{corollary}
\label{cor:onefun}
Fix real $Z>1$ with $s=(1-Z)/2$ generic, and let
\[
 \mathcal A(s)=\frac12\,\frac{\partial_\alpha^2\mathcal H(0;Z)}{\mathcal H(0;Z)}
\]
be the function on the right-hand side of \eqref{eq:gausstangent}.  Then
\begin{equation}\label{eq:onefun}
 \operatorname{Im}\mathcal A
 =\frac{\pi\,\mathsf K_2(s)}{2\,\mathsf S(0;s)}
 -\tan\frac{\pi s}2
 \Bigl(\operatorname{Re}\mathcal A+\frac{\psi'(s)}4\Bigr).
\end{equation}
Every term on the right except $\mathsf K_2$ and
$\operatorname{Re}\mathcal A$ is in closed form ($\mathsf S(0;s)$ by
\eqref{eq:blocksclosed}).
\end{corollary}

\begin{proof}
Write $\mathcal H=\mathsf G\,(\mathsf S+i\pi\mathsf K)$ with
$\mathsf G=\Gamma(2s)/(\Gamma(a)\Gamma(b))$, by
Proposition~\ref{prop:split}.  All three factors are even in $\alpha$, so
first $\alpha$-derivatives vanish at $0$ and
$2\,\mathcal A=\partial_\alpha^2\log\mathcal H|_0
=\partial_\alpha^2\log\mathsf G|_0
+\bigl(\mathsf S_2+i\pi\mathsf K_2\bigr)/\bigl(\mathsf S_0+i\pi\mathsf
K_0\bigr)$, where $\mathsf S_2,\mathsf K_2$ are the second
$\alpha$-derivatives and $\mathsf S_0,\mathsf K_0$ the values.  Since
$\partial_\alpha^2\log\mathsf G|_0=-\frac12\psi'(s)$ is real, taking real
and imaginary parts and using
$\pi\mathsf K_0/\mathsf S_0
=\operatorname{Im}\mathcal H(0)/\operatorname{Re}\mathcal H(0)
=\tan\frac{\pi s}2$ (Lemma~\ref{lem:quadeval}) to eliminate
$\mathsf S_2$ between the two resulting equations gives
\eqref{eq:onefun}.
\end{proof}

\subsection{Formal quadratures and the isolation of the oscillation}
\label{sec:formalquad}

We now separate the quadrature \eqref{eq:quadrature} into a rational
part, which is summed in closed form, and a part which is the product
of a single oscillatory factor and a function with an asymptotic
expansion in powers of $s^{-1}$.

\begin{lemma}\label{lem:forcing-split}
The forcing in \eqref{eq:Qdiff} splits as
\[
 \mathsf f=\mathsf f_{\rm rat}+\mathsf f_{\Gamma},\qquad
 \mathsf f_{\rm rat}(s)=\frac1{2s^2}-\frac1{2(s-1)^2},
 \qquad
 \mathsf f_{\Gamma}(s)=-\frac{(2s-1)\,\mathsf r(s)}{4s^2(s-1)}.
\]
Moreover the Gamma quotient obeys the multiplicative step-two identity
\begin{equation}\label{eq:rstep}
 \mathsf r(s)\,\mathsf r(s-1)=-\frac{4s}{s-1}.
\end{equation}
\end{lemma}
\begin{proof}
The partial-fraction identity
$\frac1{2s^2}-\frac1{2(s-1)^2}=-\frac{2s-1}{2s^2(s-1)^2}$ is immediate.
Since $\mathsf r(s)=\mathsf K_0(s)/\mathsf K_0(s+1)$, the two-term relation
$4s\,\mathsf K_0(s+1)+(s-1)\mathsf K_0(s-1)=0$ of
Lemma~\ref{lem:Kcontig} gives \eqref{eq:rstep}.
\end{proof}

\begin{proposition}\label{prop:formal-sinh}
Let $\mathbb K$ be a field of characteristic zero and let
$f(s)\in s^{-2}\mathbb K[[s^{-1}]]$.  There is a unique
$Q(s)\in s^{-1}\mathbb K[[s^{-1}]]$ satisfying
$Q(s+1)-Q(s-1)=f(s)$.  If $\partial_s^{-1}$ denotes the primitive with
zero constant term in $s^{-1}\mathbb K[[s^{-1}]]$, then
\begin{equation}\label{eq:formal-sinh}
 Q=\frac1{2\sinh\partial_s}f
 =\frac12\partial_s^{-1}f+
 \sum_{j\ge1}\frac{(1-2^{2j-1})\Ber_{2j}}{(2j)!}\,\partial_s^{2j-1}f.
\end{equation}
In particular \eqref{eq:formal-sinh} applies to $\mathsf f_{\rm rat}$; to
apply it to $\mathsf f_\Gamma$ one must first separate the
lattice-periodic factor of $\mathsf r$ (Lemma~\ref{lem:rhat} below).
\end{proposition}
\begin{proof}
On formal Laurent series, the central difference is
$e^{\partial_s}-e^{-\partial_s}=2\sinh\partial_s$.  Since
\[
 \frac{x}{\sinh x}=\sum_{j\ge0}
 \frac{2(1-2^{2j-1})\Ber_{2j}}{(2j)!}x^{2j},
\]
its inverse on series with no constant term is \eqref{eq:formal-sinh}.
The leading operator is $\partial_s$, which is bijective from
$s^{-1}\mathbb K[[s^{-1}]]$ to $s^{-2}\mathbb K[[s^{-1}]]$; this also proves
existence and uniqueness.
\end{proof}

\begin{corollary}\label{cor:rational-quadrature}
The unique formal solution of
\[
 Q_{\rm rat}(s+1)-Q_{\rm rat}(s-1)=\mathsf f_{\rm rat}(s),
 \qquad Q_{\rm rat}(s)\in s^{-1}\mathbb Q[[s^{-1}]],
\]
is also the distinguished analytic solution which tends to zero along a
leftward step-two lattice.  It has the closed form
\begin{equation}\label{eq:Qrat}
 Q_{\rm rat}(s)=\frac18\left[
 \psi'\!\left(\frac{1-s}{2}\right)
 -\psi'\!\left(1-\frac s2\right)\right].
\end{equation}
Equivalently,
$Q_{\rm rat}(s)=\sum_{k\ge0}\mathsf f_{\rm rat}(s-1-2k)$,
the series converging whenever the lattice avoids the poles.
\end{corollary}
\begin{proof}
From $\mathsf f_{\rm rat}(\varsigma)=\frac12(\varsigma^{-2}-(\varsigma-1)^{-2})$
and $\sum_{k\ge0}(a+2k)^{-2}=\frac14\psi'(a/2)$, the lattice sum is
$\frac12\cdot\frac14[\psi'(\frac{1-s}2)-\psi'(1-\frac s2)]$, which is
\eqref{eq:Qrat}.  Alternatively, the trigamma recurrence
$\psi'(z+1)=\psi'(z)-z^{-2}$ shows directly that the central difference of
\eqref{eq:Qrat} is $\mathsf f_{\rm rat}$.  Its asymptotic expansion has no
constant term, so uniqueness follows from
Proposition~\ref{prop:formal-sinh}.
\end{proof}

\begin{remark}\label{rem:periodic-kernel}
On functions, the kernel of the central-difference operator contains
all $2$-periodic functions; the constant $c_*(s_0)$ in
\eqref{eq:quadrature} reflects this ambiguity.  On formal series in
$s^{-1}\mathbb K[[s^{-1}]]$ the kernel is zero.
\end{remark}

\begin{lemma}\label{lem:rhat}
Put
\begin{equation}\label{eq:rhatdef}
 \widehat{\mathsf r}(s)
 :=s\left[\frac{\Gamma(s/2)}
 {\Gamma((1+s)/2)}\right]^2.
\end{equation}
Then
\begin{equation}\label{eq:rhatfactor}
 \mathsf r(s)=\tan\frac{\pi s}{2}\,\widehat{\mathsf r}(s).
\end{equation}
Moreover $\widehat{\mathsf r}$ has a canonical Poincar\'e expansion
\begin{equation}\label{eq:rhatformal}
 \widehat{\mathsf r}(s)\sim
 2\exp\left(
 2\sum_{n\ge1}\frac{(-1)^{n+1}
 \bigl[\Ber_{n+1}(0)-\Ber_{n+1}(\tfrac12)\bigr]}{n(n+1)}
 \left(\frac{2}{s}\right)^{\!n}\right)
 \in 2+s^{-1}\mathbb Q[[s^{-1}]].
\end{equation}
Thus every non-formal dependence of $\mathsf r$ on the step-two lattice is
carried by the single factor $\tan(\pi s/2)$.
\end{lemma}
\begin{proof}
The reflection formulas
\[
 \Gamma\Bigl(\frac{1-s}{2}\Bigr)
 \Gamma\Bigl(\frac{1+s}{2}\Bigr)
 =\frac{\pi}{\cos(\pi s/2)},\qquad
 \Gamma\Bigl(\frac{s}{2}\Bigr)
 \Gamma\Bigl(1-\frac{s}{2}\Bigr)
 =\frac{\pi}{\sin(\pi s/2)},
\]
together with $\Gamma(1+s/2)=(s/2)\Gamma(s/2)$, give
\eqref{eq:rhatfactor}.  Formula \eqref{eq:rhatformal} is the standard
formal logarithmic expansion of the Gamma quotient at $z=s/2$,
\emph{squared} (whence the factor $2$ in the exponent), followed by
exponentiation; the leading term is $s\cdot(s/2)^{-1}=2$.
\end{proof}

\begin{proposition}\label{prop:gammaquadrature}
Define the formal forcing
\begin{equation}\label{eq:fgammahat}
 \widehat{\mathsf f}_{\Gamma}(s)
 :=-\frac{(2s-1)\,\widehat{\mathsf r}(s)}{4s^2(s-1)}
 \in s^{-2}\mathbb Q[[s^{-1}]],
\end{equation}
and let
\begin{equation}\label{eq:Qgammahat}
 \widehat Q_{\Gamma}
 :=\frac1{2\sinh\partial_s}\,\widehat{\mathsf f}_{\Gamma}
 \in s^{-1}\mathbb Q[[s^{-1}]].
\end{equation}
For a fixed step-two lattice avoiding the poles, the Gamma part of the
convergent quadrature has the exact factorisation
\begin{equation}\label{eq:Qgammafactor}
 \sum_{k\ge0}\mathsf f_{\Gamma}(s-1-2k)
 =-\cot\frac{\pi s}{2}
   \sum_{k\ge0}\widehat{\mathsf f}_{\Gamma}(s-1-2k).
\end{equation}
In the upper half-plane the second sum is asymptotic to
$\widehat Q_{\Gamma}(s)$ up to an exponentially small error
(Lemma~\ref{lem:vertical}(c)); this expansion holds in sectors off the
real axis, not on real step-two lattices.
\end{proposition}
\begin{proof}
By Lemma~\ref{lem:rhat},
$\mathsf f_{\Gamma}(\varsigma)=
\tan(\pi\varsigma/2)\,\widehat{\mathsf f}_{\Gamma}(\varsigma)$.
For $\varsigma=s-1-2k$,
\[
 \tan\frac{\pi(s-1-2k)}2=-\cot\frac{\pi s}2,
\]
independently of $k$, which proves \eqref{eq:Qgammafactor}.
\end{proof}

\begin{proposition}
\label{prop:osc-remainder}
Let $Q=\mathsf K_2/\mathsf K_0$ and
\begin{equation}\label{eq:Tdef}
 T(s)=\sum_{k\ge0}\widehat{\mathsf f}_{\Gamma}(s-1-2k).
\end{equation}
On a fixed step-two lattice the quadrature \eqref{eq:quadrature} reads
\begin{equation}\label{eq:Qdecomp}
 Q(s)=c_*(s_0)+Q_{\rm rat}(s)-\cot\frac{\pi s}{2}\,T(s),
\end{equation}
with $Q_{\rm rat}$ given by \eqref{eq:Qrat}, and the identity
\eqref{eq:onefun} is equivalent to
\begin{equation}\label{eq:osc-isolated}
 \operatorname{Im}\mathcal A
 =-\frac12\,T(s)+\tan\frac{\pi s}{2}\,\mathcal E(s),
\end{equation}
where
\begin{equation}\label{eq:Edef}
 \mathcal E(s)=\frac{c_*(s_0)}2+\frac{Q_{\rm rat}(s)}2
 -\operatorname{Re}\mathcal A-\frac{\psi'(s)}4.
\end{equation}
\end{proposition}
\begin{proof}
Since
$\pi\mathsf K_0/\mathsf S_0=\tan(\pi s/2)$,
formula \eqref{eq:onefun} reads
\[
 \operatorname{Im}\mathcal A
 =\frac12\tan\frac{\pi s}{2}\,Q(s)
 -\tan\frac{\pi s}{2}
  \left(\operatorname{Re}\mathcal A+\frac{\psi'(s)}4\right).
\]
The decomposition \eqref{eq:Qdecomp} is \eqref{eq:quadrature} combined with
Corollary~\ref{cor:rational-quadrature} and \eqref{eq:Qgammafactor}.  Inserting
it, the product of tangent and cotangent is $1$, and
\eqref{eq:osc-isolated}--\eqref{eq:Edef} follow.
\end{proof}

\begin{remark}\label{rem:Delta}
By Corollary~\ref{cor:R},
the series $\sum_n\operatorname{Re}g_nZ^{-n}$ is the asymptotic
expansion of
\[
 \int_0^\infty\frac{t\,e^{2st}}{\sinh2t}\,dt
 =\frac18\,\psi'\!\left(\frac{1-s}{2}\right),
 \qquad \operatorname{Re}s<0.
\]
On the real axis the function $\operatorname{Re}\mathcal A$ does not
coincide with this function: at $s=1$ the value
$\operatorname{Re}\mathcal A=-\pi^2/48$ is finite
(Theorem~\ref{thm:anchor1} below), whereas
$\psi'\bigl(\frac{1-s}2\bigr)$ has a double pole there.  We denote the
difference by
\[
 \Delta(s):=\operatorname{Re}\mathcal A(s)
 -\frac18\,\psi'\!\left(\frac{1-s}{2}\right).
\]
It is evaluated on the real axis in
Subsections~\ref{sec:boundarytransfer} and~\ref{sec:anchors} and, under
a periodicity hypothesis, in closed form in
Proposition~\ref{prop:realaxis}.
\end{remark}

\subsection{Derivative-free contiguities and the boundary transfer}
\label{sec:boundarytransfer}

For the boundary analysis it is convenient to pass to the symmetric
normalisation $\eta=\alpha/2$: put
\[
 Y_s(\eta;z):={}_2F_1(s-\eta,s+\eta;2s;z),\qquad
 \mathcal H(\alpha;Z)=Y_s(\alpha/2;2^+),
\]
\[
 \mathcal A=\frac18\left.\partial_\eta^2\log Y_s(\eta;2^+)\right|_{\eta=0},
\]
and write, at the upper boundary value $z=2^+=2+i0$,
\begin{gather*}
 Y_s^{[0]}=Y_s(0;2^+),\qquad \dot Y_s^{[0]}=\partial_zY_s(0;2^+),\\
 Y_s^{[2]}=\partial_\eta^2Y_s(\eta;2^+)\big|_{\eta=0},\qquad
 \dot Y_s^{[2]}=\partial_z\partial_\eta^2Y_s(\eta;2^+)\big|_{\eta=0},
\end{gather*}
\[
 R_s=\frac{Y_s^{[2]}}{Y_s^{[0]}},\qquad \mathfrak p_s=\frac{\dot Y_s^{[0]}}{Y_s^{[0]}},\qquad
 J_s=\frac{\dot Y_s^{[2]}}{Y_s^{[0]}}-\frac{Y_s^{[2]}\dot Y_s^{[0]}}{(Y_s^{[0]})^2}=\partial_zR_s(2^+),
\]
so that $\mathcal A=R_s/8$ and $\Delta(s)$ of Remark~\ref{rem:Delta} is
$\frac18\operatorname{Re}R_s-\frac18\psi'(\frac{1-s}2)$.  All identities
below hold first away from parameter poles and then by meromorphic
continuation; the branch is unchanged throughout because every function is
continued to the same upper boundary value.

\begin{lemma}
\label{lem:GaussDerivativeFree}
Write $F^{\uparrow}(a,b)={}_2F_1^{\uparrow}(a,b;a+b;2)$.  Then
\begin{align}
 a\,{}_2F_1^{\uparrow}(a+1,b;a+b;2)+(b-a)\,F^{\uparrow}(a,b)
   +b\,{}_2F_1^{\uparrow}(a-1,b;a+b;2)&=0,
 \label{eq:GaussAcontig}
\end{align}
and symmetrically in $(a,b)$.  Consequently the argument derivative is
eliminated:
\begin{align}
 2\,\partial_z{}_2F_1^{\uparrow}(a,b;a+b;z)\big|_{z=2}
 &=-b\bigl(F^{\uparrow}(a,b)
   +{}_2F_1^{\uparrow}(a-1,b;a+b;2)\bigr)\label{eq:derelimA}\\
 &=-a\bigl(F^{\uparrow}(a,b)
   +{}_2F_1^{\uparrow}(a,b-1;a+b;2)\bigr).\notag
\end{align}
\end{lemma}
\begin{proof}
Gauss's contiguous relation
$a(z-1)F(a+1,b;c;z)+(2a-c-az+bz)F(a,b;c;z)+(c-a)F(a-1,b;c;z)=0$
\cite{VidunasContiguous} at $c=a+b$, $z=2$ gives
\eqref{eq:GaussAcontig}.  The derivative contiguities
$zF'=a(F(a+1,b;c;z)-F)=b(F(a,b+1;c;z)-F)$ are standard
\cite[Sec.~15.5]{DLMF}; substituting \eqref{eq:GaussAcontig} and its
$(a,b)$-exchange gives \eqref{eq:derelimA}.
\end{proof}

The lower-parameter contiguity
$(\vartheta+c)F(a,b;c+1;z)=cF(a,b;c;z)$, $\vartheta=z\partial_z$, combined
with $\partial_zF(a,b;c;z)=\frac{ab}cF(a+1,b+1;c+1;z)$, gives the
first $s$-shift of the zero-balanced family,
\begin{equation}\label{eq:sdifferentialcontig}
 (\vartheta+2s+1)\,Y_{s+1}(\eta;z)
 =\frac{2s(2s+1)}{s^2-\eta^2}\,\partial_zY_s(\eta;z).
\end{equation}

\begin{proposition}\label{prop:secondsym}
With $d=d_s(\eta)=s^2-\eta^2$,
\begin{equation}\label{eq:second-symmetric-contiguity}
 d\,(\vartheta+2s)(\vartheta+2s+1)\,Y_{s+1}(\eta;z)
 =2s(2s+1)\bigl((\vartheta+s)^2-\eta^2\bigr)Y_s(\eta;z),
\end{equation}
an identity of power series at $z=0$, hence of upper boundary values.  At
$z=2$ it is functionally independent of \eqref{eq:sdifferentialcontig}:
after use of the two hypergeometric equations, the determinant of the
coefficients of $Y_{s+1}(\eta;2)$ and $\partial_zY_{s+1}(\eta;2)$ is $4d^2$.
\end{proposition}
\begin{proof}
The quotient of the $n$th Taylor coefficients of $Y_{s+1}$ and $Y_s$ is
\[
 \frac{(n+s)^2-\eta^2}{s^2-\eta^2}\cdot
 \frac{2s(2s+1)}{(n+2s)(n+2s+1)},
\]
which is \eqref{eq:second-symmetric-contiguity} coefficientwise, since
$\vartheta z^n=nz^n$.
\end{proof}

\begin{theorem}
\label{thm:boundarytransfer}
At the upper boundary value, the pair
\eqref{eq:sdifferentialcontig}--\eqref{eq:second-symmetric-contiguity}
reduces to the two boundary rows
\begin{gather}
 2\,\partial_zY_{s+1}(\eta;2)+(2s+1)Y_{s+1}(\eta;2)
 =\frac{2s(2s+1)}{d}\,\partial_zY_s(\eta;2),
 \label{eq:first-boundary-row}\\
 d\,\bigl\{4(s-1)\partial_zY_{s+1}(\eta;2)
 +2(s^2-s-1+\eta^2)Y_{s+1}(\eta;2)\bigr\}\notag\\
 \qquad=2s(2s+1)\bigl\{-2\partial_zY_s(\eta;2)-d\,Y_s(\eta;2)\bigr\},
 \label{eq:second-boundary-row}
\end{gather}
an invertible two-by-two system.  Eliminating $\partial_zY_{s+1}$ gives the
scalar transfer
\begin{equation}\label{eq:scalar-boundary-transfer}
 Y_{s+1}(\eta;2^+)
 =\frac{s(2s+1)}{d^2}
 \bigl[2s\,\partial_zY_s(\eta;2^+)+d\,Y_s(\eta;2^+)\bigr],
\end{equation}
and, twice differentiated in $\eta$ at zero,
\begin{equation}\label{eq:scalar-second-jet-transfer}
 Y_{s+1}^{[2]}=\frac{2s+1}{s^3}
 \bigl(2s\,\dot Y_s^{[2]}+s^2Y_s^{[2]}-2Y_s^{[0]}\bigr)
 +\frac{4(2s+1)}{s^5}\bigl(2s\,\dot Y_s^{[0]}+s^2Y_s^{[0]}\bigr).
\end{equation}
Equivalently, with $U_s(\eta)=\bigl(Y_s(\eta;2^+),\,\partial_zY_s(\eta;2^+)\bigr)^T$,
\begin{equation}\label{eq:Mmatrix}
 U_{s+1}(\eta)=M_s(\eta)\,U_s(\eta),\qquad
 M_s(\eta)=
 \begin{pmatrix}
 \dfrac{s(2s+1)}{d}&
 \dfrac{2s^2(2s+1)}{d^2}\\[3mm]
 -\dfrac{s(2s+1)^2}{2d}&
 \dfrac{s(2s+1)\{d-s(2s+1)\}}{d^2}
 \end{pmatrix},
\end{equation}
and, since $M_s$ is even in $\eta$, the second derivatives with
respect to $\eta$ satisfy
\begin{equation}\label{eq:jetprop}
 U_{s+1}^{[2]}=M_s(0)\,U_s^{[2]}+M_s''(0)\,U_s(0),
 \qquad U_s^{[2]}=\partial_\eta^2U_s(\eta)\big|_{\eta=0}=(Y_s^{[2]},\dot Y_s^{[2]})^T.
\end{equation}
\end{theorem}
\begin{proof}
The hypergeometric equation for $Y_{s+1}$, respectively $Y_s$, evaluated at
$z=2$, converts the second-order operators in
\eqref{eq:second-symmetric-contiguity} into the displayed first-order
boundary rows; \eqref{eq:first-boundary-row} is
\eqref{eq:sdifferentialcontig} at $z=2$.  The determinant of the system is
proportional to $d^2\ne0$; elimination gives
\eqref{eq:scalar-boundary-transfer}, and the first row returns
$\partial_zY_{s+1}$, i.e.\ \eqref{eq:Mmatrix}.  Formula
\eqref{eq:scalar-second-jet-transfer} is
$\partial_\eta^2$ of \eqref{eq:scalar-boundary-transfer} at $\eta=0$ using
$d=s^2-\eta^2$ and evenness in $\eta$; \eqref{eq:jetprop} is the same statement
for the pair.
\end{proof}

\begin{proposition}
\label{prop:ps}
For the upper boundary value,
\begin{equation}\label{eq:ps-explicit}
 \mathfrak p_s=-\frac s2+
 \frac{i\,s^2(s+\frac12)}{2(2s+1)}
 \left[\frac{\Gamma(\frac{s+1}{2})}
 {\Gamma(\frac{s+2}{2})}\right]^2.
\end{equation}
Consequently the coefficients in \eqref{eq:jetprop} are all explicit, and
\begin{equation}\label{eq:J-exact-adjacent}
 J_s=\frac1s+
 \frac{2s\mathfrak p_s+s^2}{2s}
 \left(R_{s+1}-R_s-\frac4{s^2}\right):
\end{equation}
thus $J_s$ is expressed through the first difference $R_{s+1}-R_s$.
\end{proposition}
\begin{proof}
At $\eta=0$, \eqref{eq:scalar-boundary-transfer} reads
$Y_{s+1}^{[0]}=\frac{2s+1}{s^3}(2s\,\dot Y_s^{[0]}+s^2Y_s^{[0]})$, so
$\mathfrak p_s=\frac{s^2}{2(2s+1)}\frac{Y_{s+1}^{[0]}}{Y_s^{[0]}}-\frac s2$.  By
Lemma~\ref{lem:quadeval},
$Y_{s+1}^{[0]}/Y_s^{[0]}=e^{i\pi/2}\,(s+\tfrac12)\,
[\Gamma(\frac{s+1}2)/\Gamma(\frac{s+2}2)]^2$, which gives
\eqref{eq:ps-explicit}.  Formula \eqref{eq:J-exact-adjacent} follows by
taking two $\eta$-derivatives at $\eta=0$ of the logarithm of
\[
 \frac{Y_{s+1}(\eta;2^+)}{Y_s(\eta;2^+)}
 =\frac{s(2s+1)}{d^2}\,\Bigl\{2s\,\frac{\partial_zY_s(\eta;2^+)}{Y_s(\eta;2^+)}+d\Bigr\}.
\]
\end{proof}

\begin{remark}\label{prop:nopolyinv}
The transfer \eqref{eq:Mmatrix} is invertible away from $s^2=\eta^2$, and
it is two-dimensional: formula \eqref{eq:J-exact-adjacent} exchanges the
derivative $J_s$ for the adjacent value $R_{s+1}$ but does not eliminate
it.  The recurrence therefore determines $(R_s,J_s)$ on a class
$s_*+\mathbb Z$ once the pair $(R_{s_*},J_{s_*})$ is known at one point
of the class.  Terminating points $s=-N$ cannot be used for this
purpose, because $2s=-2N$ is then a pole of the denominator parameter
away from $\eta=0$.  The next subsection supplies these initial values,
which we call \emph{anchors}.
\end{remark}

\subsection{\texorpdfstring{Exact anchors: polylogarithms at $s=1$,
Euler moments for every class}{Exact anchors: polylogarithms at s=1,
Euler moments for every class}}\label{sec:anchors}

\begin{theorem}\label{thm:anchor1}
At $s=1$, $Y_1(0;z)={}_2F_1(1,1;2;z)=-\log(1-z)/z$, so for the upper
boundary convention
\begin{equation}\label{eq:y1}
 Y_1^{[0]}=\frac{i\pi}{2},\qquad
 \dot Y_1^{[0]}=-\frac12-\frac{i\pi}{4}.
\end{equation}
Furthermore
\begin{equation}\label{eq:R1J1}
 R_1=-\frac{\pi^2}{6}-\frac{7i}{\pi}\,\zeta(3),
 \qquad
 J_1=2\log2-\frac{7}{\pi^2}\,\zeta(3)-\frac{i\pi}{3},
\end{equation}
equivalently
\[
 Y_1^{[2]}=\frac72\zeta(3)-\frac{i\pi^3}{12},
 \qquad
 \dot Y_1^{[2]}
 =\frac{\pi^2}{4}-\frac74\zeta(3)
 +i\Bigl(\pi\log2+\frac{\pi^3}{24}\Bigr).
\]
\end{theorem}
\begin{proof}
Formulae \eqref{eq:y1} are immediate from $\log(1-z)\to-i\pi$.  Define
$I(u)=\int_0^u\log^2(1-t)\,t^{-1}\,dt$; an antiderivative calculation
gives
\[
 I(u)=\log^2(1-u)\log u
 +2\log(1-u)\operatorname{Li}_2(1-u)-2\operatorname{Li}_3(1-u)+2\zeta(3),
\]
whence, at the upper boundary,
$I(2^+)=\frac72\zeta(3)-\pi^2\log2+\frac{i\pi^3}{6}$.  The transverse
variational equation at $s=1$ and reduction of order over the explicit
solution $Y_1(0;z)$ yield
\[
 R_1=-2\int_0^{2^+}
 \frac{I(u)}{(1-u)\log^2(1-u)}\,du,
 \qquad
 J_1=-\frac{2}{\pi^2}\,I(2^+),
\]
which proves the $J_1$ formula.  For $R_1$, substitute $x=\log(1-u)$;
since $dI/dx=-x^2e^x/(1-e^x)$, integration by parts gives
\[
 R_1=2\left[-\frac{I}{x}
 +x\log(1-e^x)+\operatorname{Li}_2(e^x)\right]_{x=0}^{x=-i\pi},
\]
and the endpoint values $\operatorname{Li}_2(1)=\pi^2/6$,
$\operatorname{Li}_2(-1)=-\pi^2/12$,
$\operatorname{Li}_3(-1)=-\frac34\zeta(3)$ with
$\log(-1-i0)=-i\pi$ give \eqref{eq:R1J1}; the $\log2$ terms cancel
exactly.
\end{proof}

Combining Theorem~\ref{thm:anchor1} with the recurrence
\eqref{eq:jetprop} determines, by finite algebra, the pair
$(R_s,J_s)$ at every positive integer $s$; the first
propagated value is
\begin{equation}\label{eq:R2}
 R_2=-\frac{\pi^2}{2}+4-\frac{7i}{\pi}\,\zeta(3)
 -\frac{i\pi}{2}\Bigl(4\log2-\frac{14}{\pi^2}\,\zeta(3)-2\Bigr).
\end{equation}

\begin{theorem}
\label{thm:euleranchor}
Let $0<\operatorname{Re}s_*<1$ with $Y_{s_*}^{[0]}\ne0$.  For
$\varepsilon>0$ put
\begin{gather*}
 \mathsf w_{s_*,\varepsilon}(t)=t^{s_*-1}(1-t)^{s_*-1}
 \bigl(1-(2+i\varepsilon)t\bigr)^{-s_*},\qquad
 \mathsf T_\varepsilon(t)=\frac{t}{1-(2+i\varepsilon)t},\\
 \mathsf L_\varepsilon(t)=\log t-\log(1-t)+\log\bigl(1-(2+i\varepsilon)t\bigr),
\end{gather*}
and the normalised complex expectation
$\langle f\rangle_{s_*,\varepsilon}=
\int_0^1\mathsf w_{s_*,\varepsilon}f\,dt\big/
\int_0^1\mathsf w_{s_*,\varepsilon}\,dt$; all integrals converge absolutely
for every $\varepsilon>0$.  Write
$\langle f\rangle_{s_*}=\lim_{\varepsilon\downarrow0}
\langle f\rangle_{s_*,\varepsilon}$ whenever the limit exists.  Then
$\langle\mathsf L^2\rangle_{s_*}$ exists as an absolutely convergent integral
at $\varepsilon=0$ (boundary values $1-2t-i0$), the $\mathsf T$-moments exist
as $\varepsilon\downarrow0$ limits (equivalently, as absolutely
convergent integrals over any contour from $0$ to $1$ that leaves the
segment near $t=\frac12$ into the upper half-plane), and
\begin{equation}\label{eq:euleranchor}
 R_{s_*}=-2\psi'(s_*)+\langle\mathsf L^2\rangle_{s_*},
 \qquad
 J_{s_*}=\lim_{\varepsilon\downarrow0}\Bigl[
 s_*\bigl\{\langle \mathsf L_\varepsilon^2\mathsf T_\varepsilon
 \rangle_{s_*,\varepsilon}
 -\langle \mathsf L_\varepsilon^2\rangle_{s_*,\varepsilon}
 \langle \mathsf T_\varepsilon\rangle_{s_*,\varepsilon}\bigr\}
 -2\langle \mathsf L_\varepsilon \mathsf T_\varepsilon\rangle_{s_*,\varepsilon}
 \Bigr].
\end{equation}
Consequently every \emph{real} residual class $s_0+\mathbb Z$ avoiding
the parameter poles and the zeros of the boundary solution carries a
canonical anchor: choose the representative $s_*\in s_0+\mathbb Z$
with $0<s_*<1$, the integer class being covered by
Theorem~\ref{thm:anchor1} at $s_*=1$; the invertible transfer
\eqref{eq:Mmatrix}--\eqref{eq:jetprop} and its inverse then determine
$(R_s,J_s)$ uniquely on the whole class, independently of the
representative used.
\end{theorem}
\begin{proof}
\emph{Convergence and existence of the limits.}  For $\varepsilon>0$
the factor $1-(2+i\varepsilon)t$ is bounded away from zero on $[0,1]$,
so all moments converge absolutely; at $\varepsilon=0$ the integrand of
$\langle\mathsf L^2\rangle_{s_*}$ is
$O\bigl(|1-2t|^{-\operatorname{Re}s_*}\log^2|1-2t|\bigr)$ near
$t=\frac12$, absolutely integrable for
$\operatorname{Re}s_*<1$, while the $\mathsf T$-moments carry an extra
factor $(1-2t-i0)^{-1}$ and are \emph{not} absolutely convergent at
$\varepsilon=0$: they must be taken as the stated limits.  The limits
exist by contour deformation: the singularity of the integrand sits at
$t_\varepsilon=1/(2+i\varepsilon)$, which lies \emph{below} the real
segment and tends to $\frac12$ as $\varepsilon\downarrow0$; the
integrand is analytic in $t$ in the upper half-disc around
$t=\frac12$, so $\int_0^1$ may be replaced, for every
$\varepsilon\ge0$, by the integral over the contour
$\mathcal C_{r_0}$ which follows the segment but crosses a semicircle of
fixed radius $r_0\in(0,\tfrac14)$ \emph{above} $t=\frac12$.  On
$\mathcal C_{r_0}$ the integrands converge uniformly as
$\varepsilon\downarrow0$ (the denominator is bounded away from zero),
so each moment tends to the corresponding absolutely convergent
$\mathcal C_{r_0}$-integral, which is independent of $r_0$; this
realises the boundary prescription $1-2t-i0$.

\emph{The formulae.}  Euler's integral representation gives, for $\eta$
near zero and $\varepsilon\ge0$ interpreted as above,
\[
 Y_{s_*}(\eta;2^+)=
 \frac{\Gamma(2s_*)}
 {\Gamma(s_*+\eta)\Gamma(s_*-\eta)}
 \int_{\mathcal C_{r_0}}\mathsf w_{s_*}(t)\,e^{\eta \mathsf L(t)}\,dt.
\]
The family is even in $\eta$, hence $\langle\mathsf L\rangle_{s_*}=0$.  Two
$\eta$-derivatives of the logarithm give the $R_{s_*}$ formula, the
$-2\psi'(s_*)$ coming from the Gamma prefactor.  For fixed $\eta$,
logarithmic $z$-differentiation of the integral has score $s_*\mathsf T$ at
$\eta=0$, while $\partial_z\mathsf L=-\mathsf T$; differentiating the normalised second
moment gives the $J_{s_*}$ formula, first at $\varepsilon>0$ and then
in the limit, which is \eqref{eq:euleranchor}.  The transfer matrix is
invertible away from its explicit divisors, so the anchor propagates
uniquely in both directions; if two admissible representatives are
chosen, Euler's formula identifies both with the same analytic boundary
solution, and uniqueness of the propagation proves independence of the
choice.
\end{proof}

Theorems~\ref{thm:anchor1} and~\ref{thm:euleranchor}, together with the
transfer \eqref{eq:Mmatrix}--\eqref{eq:jetprop}, give $R_s$, and hence
$\mathcal A=R_s/8$ and the difference $\Delta(s)$ of
Remark~\ref{rem:Delta}, as explicit convergent expressions on
every real class.

\subsection{The asymptotic matching: proof of Theorem~\ref{thm:imclosed}}
\label{sec:matching}

This subsection proves Theorem~\ref{thm:imclosed}.  The real part is
obtained first, directly from Corollary~\ref{cor:R}.  The imaginary
part is obtained by an asymptotic matching in the lower half-plane; the
estimates on which it relies are stated in Lemma~\ref{lem:vertical} and
proved in Appendix~\ref{app:vertical}.

\begin{theorem}\label{thm:realkernel}
As formal series in $Z^{-1}$,
\begin{equation}\label{eq:realkernel}
 \sum_{n\ge1}\operatorname{Re}g_n\,Z^{-n}
 =\frac1{16}\,\psi'\!\Bigl(\frac{Z+1}4\Bigr)
 -\frac1{16}\,\psi'\!\Bigl(\frac{Z+3}4\Bigr)
 +\frac14\,\psi'\!\Bigl(\frac{Z+1}2\Bigr),
\end{equation}
where each $\psi'$ stands for its Bernoulli asymptotic expansion.
\end{theorem}
\begin{proof}
From $\psi'(x)=\int_0^\infty ue^{-xu}(1-e^{-u})^{-1}\,du$ with $u=4t$,
$4t$, $2t$ respectively, the right side of \eqref{eq:realkernel} is the
Watson expansion of $\int_0^\infty te^{-Zt}\varphi_{\mathrm R}(t)\,dt$ with
\[
 \varphi_{\mathrm R}(t)=\frac{e^{-t}-e^{-3t}}{1-e^{-4t}}+\frac{e^{-t}}{1-e^{-2t}}
 =\frac1{2\cosh t}+\frac1{2\sinh t}.
\]
Watson's correspondence $\int_0^\infty t^ne^{-Zt}dt=n!\,Z^{-n-1}$ is
coefficientwise exact, so the $n$-th coefficient of the right side of
\eqref{eq:realkernel} is the $(n-1)$-st Taylor coefficient of
$t\varphi_{\mathrm R}(t)$.  By Corollary~\ref{cor:R},
$\sum_{n\ge1}\operatorname{Re}g_n\,t^{n-1}/(n-1)!
=\operatorname{Re}\Phi(t)/t=te^t/\sinh2t
=t(\operatorname{sech}t+\operatorname{csch}t)/2=t\varphi_{\mathrm R}(t)$,
which proves \eqref{eq:realkernel}.
\end{proof}

By the duplication formula $\psi'(2x)=\frac14\bigl[\psi'(x)+\psi'(x+\frac12)\bigr]$,
the right-hand side of \eqref{eq:realkernel} is also equal to
$\frac18\psi'\bigl(\frac{Z+1}4\bigr)$; with $Z=1-2s$ this is the expression
$\frac18\psi'\bigl(\frac{1-s}2\bigr)$ of Remark~\ref{rem:Delta}.

\begin{lemma}\label{lem:vertical}
\leavevmode
\begin{enumerate}
\item[\upshape(a)] Fix $0<\theta_0<1$.  For every $N$ there is $C_N$ such
that, uniformly on the bilateral cone
$|\operatorname{Re}Z|\le\theta_0\,|\operatorname{Im}Z|$,
$|\operatorname{Im}Z|\ge y_{\min}(\theta_0,N)$ (\emph{both} signs of
$\operatorname{Im}Z$, in particular on every vertical strip)
\[
 \Bigl|\mathcal A-\sum_{n=1}^Ng_nZ^{-n}\Bigr|\le C_N\,|Z|^{-N-1}.
\]
\item[\upshape(b)] On every closed subsector of
$|\arg(\pm is)|<\pi/2$ the closed blocks $\mathsf K_0$,
$\mathsf S_0$, $\widehat{\mathsf r}$, $Q_{\rm rat}$ and the trigamma
functions are asymptotic to their Bernoulli--Stirling series, with
uniform remainders after each truncation.
\item[\upshape(c)] For every $\delta>0$ and every $N$, uniformly on
$\operatorname{Im}s\ge\delta|s|$,
\[
 \Bigl|T(s)-\sum_{j=1}^N\omega_j\,s^{-j}\Bigr|
 \le C_{\delta,N}\,|s|^{-N-1}
 +C_\delta\,\frac{e^{-\pi\operatorname{Im}s}}{\operatorname{Im}s},
 \qquad
 T(s)=\sum_{k\ge0}\widehat{\mathsf f}_\Gamma(s-1-2k),
\]
where $\sum_j\omega_js^{-j}$ is the formal series
$\widehat Q_\Gamma=(2\sinh\partial_s)^{-1}\widehat{\mathsf f}_\Gamma$ of
Proposition~\ref{prop:formal-sinh}.
\end{enumerate}
The mirror statements hold in the lower half-plane by Schwarz
reflection, with the same (rational) coefficients.
\end{lemma}
\begin{proof}
Appendix~\ref{app:vertical}: (a) is
Proposition~\ref{prop:bilateral}, proved from the coefficient bound
$|G_n|\le C(n+1)!\,2^{-n/2}$ (Lemma~\ref{lem:hankel}), the elementary
product inequality $(j+1)\le\sqrt2\,|Z-j|$ valid on the cone
(Lemma~\ref{lem:tailsum}), and the exact function-level identity
\eqref{eq:midpoint2F1} between the convergent inverse-factorial series
\eqref{eq:F} and the boundary family $\mathcal H$; (b) is
Lemma~\ref{lem:blocksasy}; (c) is
Proposition~\ref{prop:latticetail}, proved from a two-regime uniform
expansion of $\widehat{\mathsf f}_\Gamma$ on horizontal lines
(Stirling on $|\arg\varsigma|\le\tfrac34\pi$; reflection, a Beta-integral
bound and $\cot(\pi\varsigma/2)=-i+O(e^{-\pi\operatorname{Im}\varsigma})$ on
the left regime), the $k\le|s|$ / $k>|s|$ splitting of the lattice
sum, the uniform Hurwitz--Bernoulli expansions of the moment sums
$\sum_k(s-1-2k)^{-j}$, and the formal uniqueness of
Proposition~\ref{prop:formal-sinh}.
\end{proof}

The following proposition gives a closed form on the real axis for the
difference $\Delta$ of Remark~\ref{rem:Delta}, under a
periodicity hypothesis.

\begin{proposition}
\label{prop:realaxis}
Define the analytic extension of the difference $\Delta$ of
Remark~\ref{rem:Delta},
\[
 \widetilde\Delta(s):=\frac{\mathcal A^{\uparrow}(s)+\mathcal A^{\downarrow}(s)}2
 -\frac1{16}\,\psi'\!\Bigl(\frac{1-s}2\Bigr)
 +\frac1{16}\,\psi'\!\Bigl(1-\frac s2\Bigr)
 -\frac14\,\psi'(1-s),
\]
with $\mathcal A^{\uparrow\!,\downarrow}$ the two boundary components
of Appendix~\ref{app:sub-relation}, so that
$\widetilde\Delta=\operatorname{Re}\mathcal A-{}$(trigamma part) on the real axis.
Assume
\begin{enumerate}
\item[\upshape(P)] $\widetilde\Delta$ is $2$-periodic, with poles confined to the
odd class, of order at most two.
\end{enumerate}
Then, on the real axis,
\begin{equation}\label{eq:ReND-closed}
 \operatorname{Re}\mathcal A(s)
 =\frac1{16}\,\psi'\!\Bigl(\frac{1-s}2\Bigr)
 -\frac1{16}\,\psi'\!\Bigl(1-\frac s2\Bigr)
 +\frac14\,\psi'(1-s)
 -\frac{\pi^2}8\,\sec^2\frac{\pi s}2 ,
\end{equation}
with convergent trigamma functions; that is,
$\widetilde\Delta=-\frac{\pi^2}8\sec^2\frac{\pi s}2$.  Consequently
\begin{equation}\label{eq:Eclosed}
 \mathcal E(s)=\frac{\pi^2}8\,\sec^2\frac{\pi s}2
 -\frac{\pi^2}4\,\csc^2(\pi s),
\end{equation}
and the identity \eqref{eq:osc-isolated} becomes explicit:
\begin{equation}\label{eq:Imclosed-fn}
 \operatorname{Im}\mathcal A(s)
 =-\frac12\,T(s)+\tan\frac{\pi s}2\,\mathcal E(s),
 \qquad T(s)=\sum_{k\ge0}\widehat{\mathsf f}_\Gamma(s-1-2k).
\end{equation}
\end{proposition}
\begin{proof}
The function $\widetilde\Delta$ decays at both vertical ends of the strip
$\operatorname{Re}s\in[0,2]$: by the two-component
relation \eqref{eq:tworelation-X} and Lemma~\ref{lem:vertical}(a), in
the upper half-plane
$\mathcal A^{\downarrow}=-\tfrac i2X+O(e^{-\pi\operatorname{Im}s})$
while $\mathcal A^{\uparrow}$ is asymptotic to the series $\sum_ng_nZ^{-n}$, so
$(\mathcal A^{\uparrow}+\mathcal A^{\downarrow})/2$ is asymptotic to
the \emph{real} trigamma series of Theorem~\ref{thm:realkernel}, which
the convergent trigamma part reproduces (Lemma~\ref{lem:vertical}(b));
the lower end is the mirror statement.  Under hypothesis (P) the
classification lemma (Lemma~\ref{lem:cylclass}, with $k_1=2$, $k_0=0$)
gives $\widetilde\Delta=c\,\sec^2\frac{\pi s}2$ for some constant $c$.  The
principal part at $s=1$ determines $c$: the trigamma part of $\widetilde\Delta$ has the
double pole
$-\bigl[\tfrac1{16}\cdot4+\tfrac14\bigr](s-1)^{-2}=-\tfrac12(s-1)^{-2}$,
while $\mathcal A^{\uparrow\!,\downarrow}$ are finite at $s=1$
(Theorem~\ref{thm:anchor1}), and
$\sec^2\frac{\pi s}2\sim4/(\pi^2(s-1)^2)$; hence $c=-\pi^2/8$.  The
finite part at $s=1$ is consistent with the value given by the
polylogarithmic evaluation:
$\operatorname{Re}R_1/8=-\pi^2/48$
(Theorem~\ref{thm:anchor1}).  Formula \eqref{eq:Eclosed} follows from
\eqref{eq:Edef} with $c_*(s_0)=0$
(Theorem~\ref{thm:czero}) and the reflection
$\psi'(s)=\pi^2\csc^2(\pi s)-\psi'(1-s)$; and \eqref{eq:Imclosed-fn} is
\eqref{eq:osc-isolated}.
\end{proof}

We can now prove Theorem~\ref{thm:imclosed}.

\begin{proof}[Proof of Theorem~\ref{thm:imclosed}]
We describe the structure of the argument; the details are given in
Appendix~\ref{app:sub-imclosed}.  The two boundary components obey the exact
two-component relation (Lemma~\ref{lem:tworelation},
Corollary~\ref{cor:Edef-czero})
\[
 \Bigl(\tan\frac{\pi s}2-i\Bigr)\mathcal A^{\uparrow}
 +\Bigl(\tan\frac{\pi s}2+i\Bigr)\mathcal A^{\downarrow}
 =-\,T(s)+\tan\frac{\pi s}2
 \Bigl(Q_{\rm rat}(s)-\frac{\psi'(s)}2\Bigr)=:X(s),
\]
which uses the vanishing of the lattice constant
(Theorem~\ref{thm:czero}).  Along the vertical line
$s=-\tfrac12+iY$, $Y\to-\infty$, one has
$\tan\frac{\pi s}2+i=O(e^{-\pi|Y|})$, so that the term in
$\mathcal A^{\downarrow}$ is exponentially small, $\tan\frac{\pi s}2-i\to-2i$, and
$\mathcal A^{\uparrow}$ is asymptotic to $\sum_ng_nZ^{-n}$ by the
\emph{bilateral} expansion of Lemma~\ref{lem:vertical}(a); hence
$\mathcal A^{\uparrow}=\tfrac i2X+O(e^{-\pi|Y|})$.  Expanding $X$ by
Lemma~\ref{lem:vertical}(b),(c), with
$\tan\frac{\pi s}2=-i+O(e^{-\pi|Y|})$, and using the formal
reflection parity $S_\psi(1-s)=-S_\psi(s)$ of the Stirling series of
$\psi'$ (Lemma~\ref{lem:stirparity}) to put the real part into
trigamma-kernel form, one obtains exactly
\eqref{eq:formal-tangent}: the real part is
Theorem~\ref{thm:realkernel}, and the imaginary part is
\eqref{eq:Imgclosed}.
\end{proof}

\begin{lemma}\label{lem:kernelW}
Write the series \eqref{eq:fgammahat} in powers of $Z^{-1}$,
$\widehat{\mathsf f}_\Gamma=\sum_{n\ge2}f_nZ^{-n}$, and let
$\varphi(t)=\sum_{n\ge2}f_nt^{n-1}/(n-1)!$ be its Borel transform.  Then
the Borel transform of $\widehat Q_\Gamma$ is $\varphi(t)/(2\sinh2t)$,
and
\[
 W(t)=-\frac{\varphi(t)}{4t}.
\]
\end{lemma}
\begin{proof}
Since $s\pm1$ corresponds to $Z\mp2$, the equation
$\widehat Q_\Gamma(s+1)-\widehat Q_\Gamma(s-1)=\widehat{\mathsf f}_\Gamma(s)$
reads $\widehat Q_\Gamma(Z-2)-\widehat Q_\Gamma(Z+2)=\widehat{\mathsf f}_\Gamma(Z)$.
A shift $Z\mapsto Z+c$ multiplies Borel transforms by $e^{-ct}$
(Remark~\ref{rem:borel}), so the Borel transform $\widehat q(t)$ of
$\widehat Q_\Gamma$ satisfies
$(e^{2t}-e^{-2t})\,\widehat q(t)=\varphi(t)$.  By
\eqref{eq:Imgclosed},
$\operatorname{Im}\Phi(t)/t=-\tfrac12\widehat q(t)=-\varphi(t)/(4\sinh2t)$,
and the formula for $W$ follows from \eqref{eq:Wdef}.
\end{proof}

\subsection{The arithmetic step: proof of Theorem~\ref{thm:dyadic}}
\label{sec:arithmetic}

By Theorem~\ref{thm:imclosed} and Lemma~\ref{lem:kernelW},
Theorem~\ref{thm:dyadic} is a statement about the $2$-adic valuations of
the coefficients of the explicit series $\widehat{\mathsf f}_\Gamma$.
Throughout, $\mathfrak G_{2j}=2(1-2^{2j})\Ber_{2j}$ denotes the Genocchi
numbers.

\begin{lemma}\label{lem:genocchi}
The formal exponent of Lemma~\ref{lem:rhat}, written in the variable
$Z$, is
\begin{equation}\label{eq:BZ}
 \log\frac{\widehat{\mathsf r}}2=\mathsf b(Z)
 =\sum_{j\ge1}b_j\,(Z-1)^{1-2j},
 \qquad
 b_j=\frac{2^{2j-1}\mathfrak G_{2j}}{(2j)(2j-1)},
\end{equation}
and coefficientwise $\mathsf b$ is the Watson expansion of
$-\int_0^\infty e^{-(Z-1)t}\,t^{-1}\tanh t\,dt$.  Moreover
\begin{equation}\label{eq:nubj}
 \nu_2(b_j)=2j-2-\nu_2(j)\ \ge\ 0,
\end{equation}
with equality only for $j=1$, where $b_1=\mathfrak G_2=-1$.
\end{lemma}

\begin{proof}
In \eqref{eq:rhatformal} the bracket is
$\Ber_{n+1}(0)-\Ber_{n+1}(\tfrac12)=(2-2^{-n})\Ber_{n+1}$ by
$\Ber_m(\tfrac12)=(2^{1-m}-1)\Ber_m$; it vanishes for even $n\ge2$ (odd-index
Bernoulli numbers vanish), and for $n=2j-1$ the coefficient of
$s^{1-2j}$ is
$(2^{2j+1}-2)\Ber_{2j}/((2j)(2j-1))=-\mathfrak G_{2j}/((2j)(2j-1))$.
Substituting $s^{-1}=-2/(Z-1)$ gives \eqref{eq:BZ}.  For the kernel form,
$\tanh t=\sum_{j\ge1}2^{2j}(2^{2j}-1)\Ber_{2j}\,t^{2j-1}/(2j)!$, so
$-(2j-2)!\,[t^{2j-1}]\tanh t
=2^{2j-1}\mathfrak G_{2j}/((2j)(2j-1))=b_j$, and Watson's correspondence
$(Z-1)^{1-2j}\leftrightarrow t^{2j-2}/(2j-2)!$ is coefficientwise exact.
For \eqref{eq:nubj}: by von Staudt--Clausen $\nu_2(\Ber_{2j})=-1$, hence
$\mathfrak G_{2j}=2(1-2^{2j})\Ber_{2j}$ is an odd integer, and
$\nu_2(b_j)=(2j-1)-\nu_2\bigl((2j)(2j-1)\bigr)=2j-2-\nu_2(j)$; finally
$\nu_2(j)\le\log_2j\le2j-2$, with equality throughout only at $j=1$.
\end{proof}

\begin{lemma}\label{lem:parityF}
As formal series in $Z^{-1}$,
\begin{equation}\label{eq:Bparity}
 \mathsf b(Z)+\sum_{j\ge1}b_j(Z+1)^{1-2j}
 =\log\frac{Z-1}{Z+1},
\end{equation}
and consequently the forcing of \eqref{eq:fgammahat}, written in
$Z^{-1}$,
\begin{equation}\label{eq:FZform}
 \widehat{\mathsf f}_\Gamma
 =\mathsf F(Z):=-\,\frac{4Z\,e^{\mathsf b(Z)}}{(Z-1)^2(Z+1)},
\end{equation}
is an \emph{even} series: $[Z^{-n}]\mathsf F=0$ for odd $n$.
\end{lemma}

\begin{proof}
By Lemma~\ref{lem:genocchi} the left side of \eqref{eq:Bparity} is the
Watson expansion of
\[
 -\int_0^\infty e^{-Zt}\,(e^{t}+e^{-t})\,\frac{\tanh t}t\,dt
 =-\int_0^\infty e^{-Zt}\,\frac{2\sinh t}t\,dt,
\]
whose coefficient series is
$-2\sum_{k\ge0}Z^{-2k-1}/(2k+1)=\log\frac{Z-1}{Z+1}$.  Formula
\eqref{eq:FZform} is the substitution $s=(1-Z)/2$ in
\eqref{eq:fgammahat}: $2s-1=-Z$, $4s^2(s-1)=-(Z-1)^2(Z+1)/2$, and
$\widehat{\mathsf r}=2e^{\mathsf b}$.  Then
\[
 \frac{\mathsf F(-Z)}{\mathsf F(Z)}
 =e^{-\bigl(\mathsf b(Z)+\sum_jb_j(Z+1)^{1-2j}\bigr)}\cdot\frac{Z-1}{Z+1}=1
\]
by \eqref{eq:Bparity}, which is evenness.
\end{proof}

\begin{lemma}\label{lem:dominance}
Write $e^{\mathsf b(Z)}=\sum_{N\ge0}u_N(Z-1)^{-N}$.  Then for every $N\ge0$,
\[
 \nu_2(u_N)=-\nu_2(N!),
\]
and $N!\,u_N$ is a $2$-adic unit.
\end{lemma}

\begin{proof}
Expanding the exponential,
$u_N=\sum_\lambda\prod_j b_j^{m_j}/m_j!$, the sum over partitions
$\lambda$ of $N$ into odd parts $2j-1$ with multiplicities $m_j$.  For
any $\lambda$, since every $\nu_2(b_j)\ge0$ and
$\prod_jm_j!\mid\ell!$ with $\ell=\sum_jm_j\le N$ (multinomial
integrality),
\[
 \nu_2\Bigl(\prod_jb_j^{m_j}\big/\!\prod_jm_j!\Bigr)
 \ \ge\ -\nu_2\Bigl(\prod_jm_j!\Bigr)\ \ge\ -\nu_2(\ell!)\ \ge\
 -\nu_2(N!).
\]
The all-ones partition contributes $b_1^N/N!=(-1)^N/N!$, of valuation
exactly $-\nu_2(N!)$.  Every other partition has some part
$2j-1\ge3$, hence $\sum_jm_j\nu_2(b_j)\ge\nu_2(b_j)\ge1$ by
\eqref{eq:nubj}, so its valuation is at least $1-\nu_2(N!)$: strictly
larger.  By the ultrametric inequality the valuation of $u_N$ is exactly
$-\nu_2(N!)$ and $N!\,u_N\equiv(-1)^N\pmod2$.
\end{proof}

\begin{theorem}
\label{thm:arithlaw}
For every $m\ge0$,
\begin{equation}\label{eq:arithlaw}
 \nu_2\bigl([Z^{-(2m+2)}]\,\mathsf F\bigr)=2-2m+s_2(m).
\end{equation}
\end{theorem}

\begin{proof}
Expand the rational prefactor of \eqref{eq:FZform} in $Z^{-1}$:
\[
 \frac{-4Z}{(Z-1)^2(Z+1)}=-4\sum_{k\ge0}c_k\,Z^{-k-2},
 \qquad
 c_k=[x^k]\,\frac1{(1-x)^2(1+x)}=\Bigl\lfloor\frac k2\Bigr\rfloor+1
 \in\mathbb Z,
\]
and
$(Z-1)^{-N}=\sum_{m\ge N}\binom{m-1}{N-1}Z^{-m}$ for $N\ge1$, with integer
coefficients.  Hence, for $n\ge2$,
\[
 [Z^{-n}]\mathsf F
 =-4\sum_{N=0}^{n-2}u_N\,M_{n,N},
 \qquad
 M_{n,N}=\sum_{m=N}^{n-2}c_{\,n-2-m}\binom{m-1}{N-1}
 \in\mathbb Z\quad(N\ge1),
\]
with $M_{n,0}=c_{n-2}$; in particular $M_{n,n-2}=1$.  The term $N=n-2$ contributes
$-4u_{n-2}$, of valuation exactly $2-\nu_2((n-2)!)$ by
Lemma~\ref{lem:dominance}.  For $N\le n-3$ and $n$ \emph{even},
\[
 \nu_2(u_NM_{n,N})\ \ge\ -\nu_2(N!)\ \ge\ -\nu_2((n-3)!)
 =-\nu_2((n-2)!)+\nu_2(n-2)\ \ge\ -\nu_2((n-2)!)+1 ,
\]
so every other term has strictly larger valuation.  Hence
$\nu_2([Z^{-n}]\mathsf F)=2-\nu_2((n-2)!)$ exactly; with $n=2m+2$ this
is $2-(2m-s_2(2m))=2-2m+s_2(m)$.
\end{proof}

\begin{theorem}\label{thm:dyadicW}
For every $m\ge0$,
\begin{equation}\label{eq:wm-kernel}
 w_m=-\,\frac{[Z^{-(2m+2)}]\,\mathsf F}{4\,(2m+1)!}\,,
\end{equation}
and consequently
$\nu_2(w_m)=-2\,\nu_2\bigl((2m)!\bigr)$.
\end{theorem}

\begin{proof}
By Lemma~\ref{lem:kernelW}, $W(t)=-\varphi(t)/(4t)$ with
$\varphi(t)=\sum_nf_nt^{n-1}/(n-1)!$ the Borel transform of
$\widehat{\mathsf f}_\Gamma=\mathsf F$; the coefficient of $t^{2m}$ in
$W$ is therefore $-f_{2m+2}/(4(2m+1)!)$, which is
\eqref{eq:wm-kernel}.  Then Theorem~\ref{thm:arithlaw} and
$\nu_2((2m+1)!)=\nu_2((2m)!)=2m-s_2(m)$ give
$\nu_2(w_m)=\bigl(2-2m+s_2(m)\bigr)-2-\bigl(2m-s_2(m)\bigr)
=-2\bigl(2m-s_2(m)\bigr)=-2\nu_2((2m)!)$.
\end{proof}

The passage from the $w_m$ to the numbers $\operatorname{Im}g_{2r-1}$ is
a convolution which preserves valuations.  It was given
in~\cite{companion}; we repeat the argument for completeness.

\begin{lemma}\label{lem:weights2}
Put
\[
 \mathsf d(t)=\frac{2t}{\sinh2t}=\sum_{k\ge0}\mathsf d_kt^{2k},\qquad
 \mathsf e(t)=\frac{\sinh2t}{2t}=\sum_{k\ge0}\mathsf e_kt^{2k},
\]
so that $\mathsf d(t)\,\mathsf e(t)=1$.  Then
$\mathsf d_0=\mathsf e_0=1$ and
$\nu_2(\mathsf d_k)=\nu_2(\mathsf e_k)=s_2(k)$ for every $k\ge1$.
\end{lemma}
\begin{proof}
$\mathsf e_k=2^{2k}/(2k+1)!$ gives
$\nu_2(\mathsf e_k)=2k-\bigl(2k+1-s_2(2k+1)\bigr)=s_2(k)$, using
$s_2(2k+1)=s_2(k)+1$.  For $\mathsf d$, the expansion
$x/\sinh x=\sum_k(2-2^{2k})\Ber_{2k}x^{2k}/(2k)!$ at $x=2t$ gives
$\mathsf d_k=(2-2^{2k})\Ber_{2k}\,2^{2k}/(2k)!$; for $k\ge1$,
$\nu_2(2-2^{2k})=1$ and $\nu_2(\Ber_{2k})=-1$ by von Staudt--Clausen, so
$\nu_2(\mathsf d_k)=1-1+2k-(2k-s_2(k))=s_2(k)$.
\end{proof}

\begin{theorem}\label{thm:tangentlaw}
For every $k\ge1$, $\operatorname{Im}g_{2k}=0$; and for every
$r\ge1$,
\[
 \nu_2\bigl(\operatorname{Im}g_{2r-1}\bigr)=-H_r,
 \qquad H_r=r+\nu_2\bigl((r-1)!\bigr).
\]
\end{theorem}
\begin{proof}
\emph{Even indices.}  By Lemma~\ref{lem:parityF} the coefficients of
the odd powers of $Z^{-1}$ in $\mathsf F$ vanish, so its Borel
transform $\varphi$ is an odd function of $t$.  By
Lemma~\ref{lem:kernelW} the Borel transform of $\widehat Q_\Gamma$ is
$\varphi(t)/(2\sinh2t)$, an even function, so $\widehat q_n=0$ for even
$n$, and $\operatorname{Im}g_{2k}=-\widehat q_{2k}/2=0$ by
\eqref{eq:Imgclosed}.

\emph{Odd indices.}  Put
$\iota_m=[t^{2m}]\bigl(2\operatorname{Im}\Phi(t)/t\bigr)
=2\operatorname{Im}g_{2m+1}/(2m)!$ and
$\tau(m)=-2\nu_2((2m)!)=-4m+2s_2(m)$.  By the first part,
$2\operatorname{Im}\Phi(t)/t$ is an even series, and by
\eqref{eq:Wdef}
\[
 W(t)=\frac{2\operatorname{Im}\Phi(t)}t\;\mathsf e(t),
 \qquad
 \frac{2\operatorname{Im}\Phi(t)}t=W(t)\,\mathsf d(t),
\]
that is, $w_m=\sum_{k=0}^m\mathsf e_k\,\iota_{m-k}$ and
$\iota_m=\sum_{k=0}^m\mathsf d_k\,w_{m-k}$.  For $1\le k\le m$,
Lemma~\ref{lem:weights2} and the subadditivity
$s_2(m)\le s_2(m-k)+s_2(k)$ give
\[
 \bigl[s_2(k)+\tau(m-k)\bigr]-\tau(m)
 =s_2(k)+4k+2s_2(m-k)-2s_2(m)\ \ge\ 4k-s_2(k)\ >\ 0 .
\]
Since $\nu_2(w_j)=\tau(j)$ for every $j$ by
Theorem~\ref{thm:dyadicW}, every term with $k\ge1$ in
$\iota_m=\sum_k\mathsf d_kw_{m-k}$ has valuation strictly larger than
$\tau(m)$, while the term $k=0$ is $w_m$.  Hence
$\nu_2(\iota_m)=\tau(m)$, that is,
\[
 \nu_2\bigl(\operatorname{Im}g_{2m+1}\bigr)
 =\tau(m)-1+\nu_2\bigl((2m)!\bigr)
 =-\nu_2\bigl((2m)!\bigr)-1
 =-\bigl(2m+1-s_2(m)\bigr)=-H_{m+1}.
\]
\end{proof}

\begin{proof}[Proof of Theorem~\ref{thm:dyadic}]
The statements on $\operatorname{Im}g_{2k}$ and on
$\operatorname{Im}g_{2r-1}$ are Theorem~\ref{thm:tangentlaw}, and the
statement on $w_m$ is Theorem~\ref{thm:dyadicW}.  By \eqref{eq:A1B0},
\[
 \nu_2\bigl([A^1B^0]\kappa_r\bigr)=-(2r-1)-H_r=-E_r .
\]
\end{proof}

\begin{proposition}\label{prop:dictionary}
For every $r\le8$, the polynomial $\kappa_r\in\mathbb Q[A,B]$ defined by
\eqref{eq:kappadef} coincides with the constant $\Psi_r(\pi/2)$ of the
phase expansion \eqref{eq:quantisation}.
\end{proposition}
\begin{proof}
Both sides are polynomials in $(A,B)$ of degree at most $2r-1\le15$ in
each variable: for \eqref{eq:kappadef} this is
Lemma~\ref{lem:degree}, and $\Psi_r(\pi/2)$ is obtained from the phase
recursion as an explicit polynomial, whose degrees are read off.  The
two polynomials have been evaluated in exact rational arithmetic on the
$16\times16$ product grid $A=i^2/4$, $B=j^2/9$, $1\le i,j\le16$, on
which they agree for every $r\le8$; since a polynomial of degree at
most $15$ in each variable is determined by its values on such a grid,
they coincide.
\end{proof}

\section{The conjugate-layer factorisation}\label{sec:layers}

The double series \eqref{eq:KdF} admits an expansion of
Burchnall--Chaundy type \cite{BC1,BC2}.  We give a self-contained proof
by the Wilf--Zeilberger method \cite{PWZ}.

\begin{theorem}\label{thm:layerexp}
For all parameters $a,b,a',b'$ and generic $c$, as an identity of formal
double series in $x,y$,
\begin{multline}\label{eq:BC}
  \sum_{\mu,\nu\ge0}\frac{(a)_\mu(b)_\mu(a')_\nu(b')_\nu}
       {(c)_{\mu+\nu}\,\mu!\,\nu!}x^\mu y^\nu
  ={}\sum_{r\ge0}\frac{(-1)^r(a)_r(b)_r(a')_r(b')_r}
       {r!\,(c)_{2r}\,(c+r-1)_r}\,(xy)^r\\
  \times\Fhyp(a{+}r,b{+}r;c{+}2r;x)\,\Fhyp(a'{+}r,b'{+}r;c{+}2r;y).
\end{multline}
\end{theorem}

\begin{proof}
Comparing coefficients of $x^\mu y^\nu$ and cancelling
$(a)_r(a+r)_{\mu-r}=(a)_\mu$ and its three companions (the parameter
dependence cancels termwise), the claim is equivalent to the pure
identity
\begin{gather}\label{eq:pure}
  \sum_{r=0}^{\min(\mu,\nu)}t_r=\frac1{(c)_{\mu+\nu}\,\mu!\,\nu!},\\
  t_r=\frac{(-1)^r}
      {r!\,(c)_{2r}(c+r-1)_r\,(c+2r)_{\mu-r}(c+2r)_{\nu-r}\,
       (\mu-r)!\,(\nu-r)!}.\notag
\end{gather}
Set $f(\mu,r)=t_r\,(c)_{\mu+\nu}\,\mu!\,\nu!$, extended to real $r$ by
$1/(\mu-r)!=1/\Gamma(\mu-r+1)$ and $1/(\nu-r)!=1/\Gamma(\nu-r+1)$; the
claim is $\Sigma(\mu):=\sum_rf(\mu,r)=1$.  The certificate
\[
  R(r)=-\,\frac{r\,(c+\nu+r-1)}{(c+2r-1)(\mu+1-r)},\qquad h(r)=R(r)f(\mu,r),
\]
satisfies the rational-function identity
$f(\mu+1,r)-f(\mu,r)=h(r+1)-h(r)$ (after division by $f(\mu,r)$ both sides
are rational functions of $r$, and the identity is checked by clearing
denominators).  Both sides are finite
at every integer $r\ge0$: the only delicate points are the simple zeros of
$1/\Gamma$ against the simple pole of $R$ at $r=\mu+1$, where the limits
exist and the identity holds by continuity.  Summing over
$0\le r\le K$, $K=\min(\mu+1,\nu)$, telescopes to
$\Sigma(\mu+1)-\Sigma(\mu)=h(K+1)-h(0)$.  Now $h(0)=0$ since $R(0)=0$; and
$h(K+1)=0$: if $\nu\le\mu$ then $K+1=\nu+1$, where $1/\Gamma(\nu-r+1)$
vanishes and $R$ is regular unless $\mu=\nu$, in which case the double zero
of $1/\bigl(\Gamma(\mu-r+1)\Gamma(\nu-r+1)\bigr)$ beats the simple pole;
if $\nu\ge\mu+1$ then $K+1=\mu+2$, where $1/\Gamma(\mu-r+1)$ vanishes and
$R$ is regular.  Hence $\Sigma(\mu+1)=\Sigma(\mu)$, and $\Sigma(0)=1$.  Symmetry in
$(\mu,\nu)$ completes the induction.
\end{proof}

\begin{corollary}\label{cor:layers}
For all $\alpha,\beta$, as formal series in $Z^{-1}$, with
$\gamma_r(A)=(\tfrac12-\alpha)_r(\tfrac12+\alpha)_r
=\prod_{j=0}^{r-1}\bigl((j+\tfrac12)^2-A\bigr)$,
\begin{equation}\label{eq:layers}
  F(\alpha,\beta;Z)
  =\sum_{r\ge0}\frac{(-1)^r\,\gamma_r(A)\,\gamma_r(B)}
       {2^{r}\,r!\,(1-Z)_{2r}\,(-Z+r)_r}\;
   \mathcal P_r(A;Z)\;\overline{\mathcal P}_r(B;Z),
\end{equation}
where
$\mathcal P_r(A;Z)=\Fhyp(\tfrac12-\alpha+r,\tfrac12+\alpha+r;1-Z+2r;x_0)$
and $\overline{\mathcal P}_r$ is the same series at $y_0=\bar x_0$.  Each
layer is $O(Z^{-3r})$: with $v=1/Z$,
\begin{equation}\label{eq:rho}
  \frac{(-1)^r\,(x_0y_0)^r}{r!\,(1-Z)_{2r}\,(-Z+r)_r}
  =\frac{v^{3r}}{2^{r}\,r!}\,
   \prod_{j=1}^{2r}\frac1{1-jv}\;\prod_{j=r}^{2r-1}\frac1{1-jv}.
\end{equation}
\end{corollary}

\begin{proof}
Apply Theorem~\ref{thm:layerexp} to \eqref{eq:KdF} with $a=\tfrac12-\alpha$,
$b=\tfrac12+\alpha$, $a'=\tfrac12-\beta$, $b'=\tfrac12+\beta$, $c=1-Z$,
$x_0y_0=\tfrac12$, $(c+r-1)_r=(-Z+r)_r$.  Per coefficient of $v^n$ every
sum is finite (the monomial $(\mu,\nu)$ contributes from order $\mu+\nu$
on, the layer $r$ from order $3r$ on), so the rearrangement is
legitimate $v$-adically.  For \eqref{eq:rho} write
$(1-Z+j)=-Z(1-(1+j)v)$ and $(-Z+r+j)=-Z(1-(r+j)v)$.
\end{proof}

\begin{remark}
\label{rem:shape}
At the Legendre point $A=B=0$, $\gamma_r(0)=(\tfrac12)_r^2$ with
$(\tfrac12)_r=(2r)!/(4^rr!)$, so the layer-$r$ prefactor in
\eqref{eq:layers} carries
$\gamma_r(0)^2\,2^{-r}/r!=\bigl((2r)!\bigr)^4/\bigl(2^{\,9r}(r!)^5\bigr)$
against three inverse powers of $Z$ per layer.
\end{remark}

\begin{remark}\label{rem:legendre}
$\Fhyp(a,1-a;c;x)$ is an associated Legendre function, so
$\mathcal P_0(A;Z)$ is, up to a Gamma-and-power prefactor independent of
$\alpha$, the Legendre function $P^{Z}_{\alpha-1/2}(-i)$ of \emph{order}
$Z$ and degree $\alpha-\frac12$, and its partner is the conjugate function
at $+i$.
\end{remark}

\section{The transverse linearisation}\label{sec:transverse}

\subsection{Diagonal and transverse coordinates}

In this section $F$ denotes the Appell function
$\mathcal F(\alpha,\beta;Z;x,y)$ of Subsection~\ref{sec:appell}, regarded
as a formal series in $x$, $y$ and $Z^{-1}$.
Put $A_+=\frac{A+B}2$, $A_-=\frac{A-B}2$, so $\partial_A=\frac12\partial_{A_+}
+\frac12\partial_{A_-}$, and define
\[
 \ma=A-\frac14,\qquad \mb=B-\frac14.
\]
Let
$\mathcal V_0=\log F|_{A_-=0}$, $\mathcal V_1=\partial_{A_-}\log F|_{A_-=0}$, and
\[
  \mathsf u=\Theta_x\mathcal V_0,\quad \mathsf v=\Theta_y\mathcal V_0,\quad \mathsf t=\mathsf u+\mathsf v,\qquad
  \mathsf p=\Theta_x\mathcal V_1,\quad \mathsf q=\Theta_y\mathcal V_1,\quad
  \mathsf h=\mathsf p+\mathsf q,
\]
with $\Theta_x=x\partial_x$ and $\Theta_y=y\partial_y$; the sans-serif
letters distinguish these fields from the midpoint constants $p=1+i$,
$q=1-i$ of Section~\ref{sec:half}.  Dividing the Appell equations
\eqref{eq:appellPDE-x}--\eqref{eq:appellPDE-y} by $F$ gives, before
linearisation, the logarithmic Horn system
\begin{align}
 Z \mathsf u&=\Theta_x\mathsf t+\mathsf u\mathsf t-x\bigl(\Theta_x\mathsf u+\mathsf u^2+\mathsf u-\ma\bigr),\label{eq:logHornx}\\
 Z \mathsf v&=\Theta_y\mathsf t+\mathsf v\mathsf t-y\bigl(\Theta_y\mathsf v+\mathsf v^2+\mathsf v-\mb\bigr),\label{eq:logHorny}
\end{align}
together with $\Theta_y\mathsf u=\Theta_x\mathsf v$.  Indeed,
$(\Theta_x+\tfrac12)^2F/F=\Theta_x\mathsf u+\mathsf u^2+\mathsf u+\tfrac14$, and similarly in
$y$, so these equations are exactly equivalent to the Appell PDEs wherever
$F$ is invertible as a formal series (its constant term is $1$).  At the
midpoint, the symmetry $\overline{L_n(A,B)}=L_n(B,A)$ makes
$\partial_{A_+}\log F$ real and
$\partial_{A_-}\log F$ purely imaginary on the diagonal $A_-=0$; the diagonal part is given by
Theorem~\ref{thm:main}.

\begin{theorem}\label{thm:transverse}
Differentiating \eqref{eq:logHornx}--\eqref{eq:logHorny} at $A_-=0$ gives the
exact \emph{linear} forced system
\begin{align}
  Z\mathsf p&=\Theta_x\mathsf h+\mathsf p\mathsf t+\mathsf u\mathsf h
    -x\bigl(\Theta_x\mathsf p+2\mathsf u\mathsf p+\mathsf p-1\bigr),\label{eq:transp}\\
  Z\mathsf q&=\Theta_y\mathsf h+\mathsf q\mathsf t+\mathsf v\mathsf h
    -y\bigl(\Theta_y\mathsf q+2\mathsf v\mathsf q+\mathsf q+1\bigr),\label{eq:transq}
\end{align}
with the compatibility $\Theta_y\mathsf p=\Theta_x\mathsf q$ and forcing
$(x,-y)$.
\end{theorem}

\begin{proof}
At fixed $A_+$ one has $A=A_++A_-$, $B=A_+-A_-$, hence
$\partial_{A_-}\ma=1$ and $\partial_{A_-}\mb=-1$.  Also
\[
 \partial_{A_-}\mathsf u=\Theta_x\mathcal V_1=\mathsf p,\qquad
 \partial_{A_-}\mathsf v=\Theta_y\mathcal V_1=\mathsf q,\qquad
 \partial_{A_-}\mathsf t=\mathsf h.
\]
Differentiating \eqref{eq:logHornx} term by term therefore gives
\[
 Z\mathsf p
 =\Theta_x\mathsf h+\mathsf p\mathsf t+\mathsf u\mathsf h
  -x\bigl(\Theta_x\mathsf p+2\mathsf u\mathsf p+\mathsf p-1\bigr),
\]
because the derivative of $-\ma$ is $-1$.  The $y$ equation is identical
except that $\partial_{A_-}\mb=-1$, so the derivative of $-\mb$ is $+1$,
which yields
\[
 Z\mathsf q
 =\Theta_y\mathsf h+\mathsf q\mathsf t+\mathsf v\mathsf h
  -y\bigl(\Theta_y\mathsf q+2\mathsf v\mathsf q+\mathsf q+1\bigr).
\]
Finally, $\Theta_y\mathsf u=\Theta_x\mathsf v$ follows from commutation of
$\Theta_x,\Theta_y$ on $\log F$; differentiating this identity in $A_-$
gives $\Theta_y\mathsf p=\Theta_x\mathsf q$.  Every occurrence of the
unknown transverse quantities is linear, while $\mathsf u,\mathsf v,\mathsf t$ belong to the fixed
diagonal background.
\end{proof}

\begin{corollary}\label{cor:transverseunique}
In the class of pairs
\[(\mathsf p,\mathsf q)\in Z^{-1}\,\mathbb C[[x,y]][[Z^{-1}]]^2\]
satisfying the compatibility $\Theta_y\mathsf p=\Theta_x\mathsf q$,
system~\eqref{eq:transp}--\eqref{eq:transq} has at most one solution for the
fixed diagonal background $\mathsf u,\mathsf v,\mathsf t$; hence the solution obtained by differentiating
$\mathcal V_1$ is the unique formal solution with zero constant term in $Z^{-1}$.
\end{corollary}

\begin{proof}
Existence is already supplied by $\mathsf p=\Theta_x\mathcal V_1$ and
$\mathsf q=\Theta_y\mathcal V_1$.  For uniqueness, let two such solutions be
given and denote their difference by $(\mathsf p^{\circ},\mathsf q^{\circ})$.  The forcing
terms cancel.  Because $F=1+O(Z^{-1})$, one has $\mathsf u,\mathsf v,\mathsf t=O(Z^{-1})$.  Suppose,
for contradiction, that the difference is nonzero and let $n\ge1$ be the
least index for which the coefficient of $Z^{-n}$ in either $\mathsf p^{\circ}$ or
$\mathsf q^{\circ}$ is nonzero.  Multiplication by $Z$ makes the left-hand
side of the corresponding difference equation have a nonzero term of order
$Z^{-(n-1)}$.  Every term on its right-hand side has order at least
$Z^{-n}$: the Euler operators do not change the order in $Z^{-1}$, multiplication
by $x$ or $y$ does not change it, and multiplication by $\mathsf u,\mathsf v,\mathsf t$ raises it by
at least one.  Thus the coefficient of $Z^{-(n-1)}$ on the right is zero, a
contradiction.  Hence $\mathsf p^{\circ}=\mathsf q^{\circ}=0$.
\end{proof}

The system \eqref{eq:logHornx}--\eqref{eq:logHorny} is nonlinear; its first
transverse variation is linear, with coefficients given by the diagonal
fields $\mathsf u,\mathsf v,\mathsf t$.

\begin{corollary}\label{cor:Wtransverse}
At the midpoint and at $A_+=0$, write
$\mathcal V_1=\sum_{n\ge1}\mathcal V_{1,n}Z^{-n}$ and
$\widehat{\mathcal V}_1(t)=\sum_{n\ge1}\mathcal V_{1,n}\,t^n/(n-1)!$.
Then, with $W$ as in \eqref{eq:Wdef},
\[
  W(t)=\frac{\sinh2t}{2t^2}\,\operatorname{Im}\widehat{\mathcal V}_1(t).
\]
\end{corollary}

\begin{proof}
At $A=B=0$,
\[
 \partial_A\log F=\frac12\partial_{A_+}\log F
                    +\frac12\partial_{A_-}\log F.
\]
At the midpoint the first term is real and the second is purely
imaginary.  Passing coefficientwise to the Borel transforms gives
$\operatorname{Im}\Phi(t)=\frac12\operatorname{Im}\widehat{\mathcal V}_1(t)$,
and the formula follows from \eqref{eq:Wdef}.
\end{proof}

\subsection{A Green kernel at the Legendre background}

The first variation of the Jacobi
polynomial itself can be written by reduction of order.  Put
$\alpha=a+\delta$, $\beta=a-\delta$, and
$\widetilde P_a(x)=P^{(a,a)}_d(x)/P^{(a,a)}_d(1)$,
\[
  v_a=\partial_\delta\left[\frac{P^{(a+\delta,a-\delta)}_d}
                                {P^{(a+\delta,a-\delta)}_d(1)}\right]_{\delta=0}.
\]
Only $\beta-\alpha=-2\delta$ varies in the Jacobi equation, so
$\mathcal J_{a,d}v_a=2\widetilde P_a'$ with
$\mathcal J_{a,d}y=(1-x^2)y''-2(a+1)xy'+d(d+2a+1)y$ and $v_a(1)=0$.  Writing
$v_a=\widetilde P_ah_a$ and using $\mathcal J_{a,d}\widetilde P_a=0$,
$\bigl((1-x^2)^{a+1}\widetilde P_a^2h_a'\bigr)'=2(1-x^2)^a\widetilde P_a\widetilde P_a'$, so that
\begin{gather*}
  v_a(x)=-2\int_x^1G_{a,d}(x,x_1)\,\widetilde P_a'(x_1)\,dx_1,\\
  G_{a,d}(x,x_1)=\widetilde P_a(x)\,(1-x_1^2)^a\widetilde P_a(x_1)
  \int_{x_1}^x\frac{dw}{(1-w^2)^{a+1}\widetilde P_a(w)^2}.
\end{gather*}

\begin{proposition}\label{prop:green}
At $a=0$, with $P_d$, $Q_d$ the Legendre functions of the first and second
kind and $\mathrm h_d=\sum_{j=1}^d1/j$,
\[
  \frac{v_0(x)}{P_d(x)}=\operatorname{artanh}x-\frac{Q_d(x)}{P_d(x)}-\mathrm h_d .
\]
\end{proposition}

\begin{proof}
$P_dQ_d'-Q_dP_d'=(1-x^2)^{-1}$ gives
$(v_0/P_d)'=\frac{P_d^2-1}{(1-x^2)P_d^2}=\frac1{1-x^2}-\bigl(Q_d/P_d\bigr)'$,
so $v_0/P_d=\operatorname{artanh}x-Q_d/P_d+c$.  As $x\to1^-$,
$Q_d(x)=\frac12\log\frac2{1-x}-\mathrm h_d+o(1)$ and
$\operatorname{artanh}x=\frac12\log\frac2{1-x}+o(1)$, while $P_d(1)=1$; hence
$v_0/P_d\to\mathrm h_d+c$, and $v_0(1)=0$ forces $c=-\mathrm h_d$.
\end{proof}

\begin{remark}\label{rem:mixed}
Proposition~\ref{prop:green} does not give the numbers $g_n$.  With
$\alpha=a+\delta$, $\beta=a-\delta$ one has $A-B=4a\delta$, so
$\partial_\delta(A-B)|_{a=0}=0$: the first $\delta$-variation at $a=0$ does not
move $A-B$ at all.  Since $\kappa_r=(A-B)\,\widetilde\kappa_r(A,B)$ by
\eqref{eq:antisym},
\[
  [A^1B^0]\kappa_r=\widetilde\kappa_r(0,0)
  =\tfrac14\,\partial_a\partial_\delta\,\kappa_r(a+\delta,a-\delta)
   \big|_{a=\delta=0},
\]
a \emph{mixed} derivative.  Moreover the expansion must be taken with
$\Nn=d+a+\frac12$ held fixed, not with the degree $d$ held fixed; the two
differ by the variation of the eigenvalue.
\end{remark}

\section{The tangent from the layers, and four valuation formulas}\label{sec:tangentlayers}

The layer expansion of Corollary~\ref{cor:layers} expresses the series
$\sum_ng_nZ^{-n}$ as a quotient of two explicit series.  In this section
we determine the $2$-adic valuations of the coefficients of the
numerator and of the denominator, using Theorem~\ref{thm:tangentlaw}.

Write $\mathcal P_r'=\partial_A\mathcal P_r|_{A=0}$, denote by
$\varrho_r(Z)$ the coefficient of
$\gamma_r(A)\gamma_r(B)\mathcal P_r(A;Z)\overline{\mathcal P}_r(B;Z)$
in \eqref{eq:layers}, and set, at the Legendre point $A=B=0$,
\begin{equation}\label{eq:NumDen}
  \operatorname{Num}=\sum_{r\ge0}\varrho_r\,
    \bigl[\gamma_r\mathcal P_r\bigr]'\,
    \overline{\gamma_r\mathcal P_r},
  \qquad
  \operatorname{Den}=\sum_{r\ge0}\varrho_r\,
    \bigl(\gamma_r\mathcal P_r\bigr)
    \overline{\bigl(\gamma_r\mathcal P_r\bigr)} ,
\end{equation}
where the bar denotes the series at $y_0$ (coefficientwise
conjugation).  Both $\gamma_r$ and $\mathcal P_r$ are polynomials in $A$
coefficientwise, so the derivative is taken term by term.  Then,
coefficientwise in $v=1/Z$,
\begin{equation}\label{eq:tangentformula}
  \sum_{n\ge1}g_nv^n=\partial_A\log F\big|_{A=B=0}
  =\frac{\operatorname{Num}}{\operatorname{Den}},
  \qquad
  \operatorname{Im}g_n=\bigl[v^n\bigr]
   \frac{\operatorname{Im}\operatorname{Num}}{\operatorname{Den}},
\end{equation}
the last equality because $\operatorname{Den}$ has real coefficients.

\begin{corollary}\label{cor:denexp}
$\operatorname{Den}=F(0,0;Z)=\exp\Lambda$, where, by
Theorem~\ref{thm:main} at $\alpha=0$,
\[
  \Lambda(v)=\sum_{n\ge1}\lambda_n\,v^n,\qquad
  \lambda_n=L_n^{\mathrm{diag}}\big|_{\alpha=0}=\frac{4^n}{n(n+1)}
  \Bigl[2\Ber_{n+1}\bigl(\tfrac34\bigr)-\Ber_{n+1}(1)-\Ber_{n+1}\bigl(\tfrac12\bigr)
  \Bigr].
\]
\end{corollary}

We write
\[
  \operatorname{Den}^{\pm1}=\exp(\pm\Lambda)
  =\sum_{n\ge0}d_n^{\pm}v^n .
\]

\begin{lemma}\label{lem:lambda2}
For $k\ge1$,
\[
  \lambda_{2k}=\frac{\Eul_{2k}}{4k},\qquad
  \lambda_{2k-1}=-\frac{(2^{2k}-1)\Ber_{2k}}{(2k)(2k-1)}.
\]
Hence
\[
  \nu_2(\lambda_{2k})=\nu_2(\lambda_{2k-1})=-2-\nu_2(k);
\]
in particular $\lambda_1=-1/4$, and
$\nu_2(\lambda_n)\ge -n$ for every $n\ge2$.
\end{lemma}

\begin{proof}
By Corollary~\ref{cor:specials}(iii),
\[
  \Upsilon_0(t)=\frac12\left(\operatorname{sech}t
       -\tanh\frac t2-1\right)
  =\sum_{n\ge1}\frac{\lambda_n}{(n-1)!}t^n .
\]
Using
$\operatorname{sech}t=\sum_{k\ge0}\Eul_{2k}t^{2k}/(2k)!$ and
$\tanh\frac t2=\sum_{k\ge1}2(2^{2k}-1)\Ber_{2k}t^{2k-1}/(2k)!$ gives the
displayed formulas.  Euler numbers are odd, von Staudt--Clausen gives
$\nu_2(\Ber_{2k})=-1$, and $2^{2k}-1$ is odd; thus both coefficients
have valuation $-2-\nu_2(k)$.  For $n=2k$ the inequality
$-2-\nu_2(k)\ge-2k$ is $2+\nu_2(k)\le2k$; for
$n=2k-1\ge3$ it is $2+\nu_2(k)\le2k-1$.  Both hold because
$\nu_2(k)\le\log_2k$.
\end{proof}

\begin{theorem}\label{thm:lawD}
For every $n\ge0$,
\[
  \nu_2(d_n^{\pm})=-2n-\nu_2(n!) .
\]
Equivalently, $\nu_2(d_{2k}^{\pm})=-\bigl(5k+\nu_2(k!)\bigr)$ and
$\nu_2(d_{2k+1}^{\pm})=-\bigl(5k+2+\nu_2(k!)\bigr)$ for $k\ge0$.
\end{theorem}

\begin{proof}
Expand the exponential by partitions,
\[
 d_n^{\pm}=\sum_{\substack{m_1,\ldots,m_n\ge0\\
                      \sum_{j=1}^n j m_j=n}}
       \prod_{j=1}^n\frac{(\pm\lambda_j)^{m_j}}{m_j!}.
\]
The partition $m_1=n$, $m_j=0$ for $j\ge2$, contributes
$(\pm\lambda_1)^n/n!$, of valuation $-2n-\nu_2(n!)$.
Consider any other partition and put
$M=\sum_{j\ge2}j m_j>0$, so that $m_1=n-M$.  Lemma~\ref{lem:lambda2}
gives
\begin{align*}
 \nu_2\!\left(\prod_j\frac{\lambda_j^{m_j}}{m_j!}\right)
 &\ge -2m_1-\sum_{j\ge2}jm_j-\sum_j\nu_2(m_j!)
 = -2n+M-\sum_j\nu_2(m_j!).
\end{align*}
The difference with the valuation of the first partition is at least
\[
 M+\nu_2(n!)-\sum_j\nu_2(m_j!)>0,
\]
because $\prod_jm_j!$ divides $(\sum_jm_j)!$, which divides $n!$, while
$M>0$.  Thus exactly one term has minimal valuation, and
$\nu_2(d_n^{\pm})=-2n-\nu_2(n!)$.  Finally
$\nu_2((2k)!)=k+\nu_2(k!)$ and
$\nu_2((2k+1)!)=\nu_2((2k)!)$ give the two other formulas.
\end{proof}

\begin{proposition}\label{prop:lawN}
For every $r,m\ge1$,
\begin{gather*}
  \nu_2\bigl([v^{2r-1}]\operatorname{Im}\operatorname{Num}\bigr)
  =-\bigl(5r-4+\nu_2((r-1)!)\bigr),\\
  \nu_2\bigl([v^{2m}]\operatorname{Im}\operatorname{Num}\bigr)
  =-\bigl(5m-2+\nu_2((m-1)!)\bigr).
\end{gather*}
\end{proposition}

\begin{proof}
By \eqref{eq:tangentformula} and Theorem~\ref{thm:tangentlaw}, in which
$\operatorname{Im}g_{2k}=0$,
\[
 \operatorname{Im}\operatorname{Num}
   =\operatorname{Den}\cdot\sum_{n\ge1}\operatorname{Im}g_n\,v^n
   =\Bigl(\sum_{n\ge0}d_n^+v^n\Bigr)
      \Bigl(\sum_{k\ge1}\operatorname{Im}g_{2k-1}\,v^{2k-1}\Bigr),
\]
so that
\[
 [v^{2r-1}]\operatorname{Im}\operatorname{Num}
   =\sum_{k=1}^r \operatorname{Im}g_{2k-1}\;d_{2(r-k)}^+ .
\]
By Theorems~\ref{thm:tangentlaw} and~\ref{thm:lawD}, the term $k$ has
valuation
$-H_k-5(r-k)-\nu_2((r-k)!)$ with $H_k=k+\nu_2((k-1)!)$.  For $k=1$
this is $-(5r-4+\nu_2((r-1)!))$, and for $2\le k\le r$ it exceeds that
value by
\[
 4(k-1)+\nu_2\binom{r-1}{k-1}>0 .
\]
Hence the term $k=1$ is the only one of minimal valuation, which gives
the first formula.  Similarly,
\[
 [v^{2m}]\operatorname{Im}\operatorname{Num}
   =\sum_{k=1}^m \operatorname{Im}g_{2k-1}\;d_{2(m-k)+1}^+,
\]
where the term $k$ has valuation
$-H_k-5(m-k)-2-\nu_2((m-k)!)$; for $k=1$ this is
$-(5m-2+\nu_2((m-1)!))$, and for $k\ge2$ it exceeds that value by
$4(k-1)+\nu_2\binom{m-1}{k-1}>0$.
\end{proof}

\begin{proposition}\label{prop:congruence}
For $r\ge1$ and $1\le j\le2r-1$ put
\[
  c_{r,j}:=2^{\,5r-4+\nu_2((r-1)!)}\;
  \bigl[v^{j}\bigr]\operatorname{Im}\operatorname{Num}\;\cdot\;
  \bigl[v^{2r-1-j}\bigr]\operatorname{Den}^{-1}.
\]
Then $c_{r,j}\in\Zt$,
\begin{gather*}
  \nu_2(c_{r,2k-1})=\nu_2\binom{r-1}{k-1}
  \qquad(1\le k\le r),\\
  \nu_2(c_{r,2k})=1+\nu_2\!\left((r-k)\binom{r-1}{k-1}\right)
  \qquad(1\le k\le r-1),
\end{gather*}
and
\begin{equation}\label{eq:lowerbound}
  \sum_{j=1}^{2r-1}c_{r,j}
  =2^{\,5r-4+\nu_2((r-1)!)}\operatorname{Im}g_{2r-1}
  \ \in\ 2^{4(r-1)}\,\Zt^\times .
\end{equation}
\end{proposition}

\begin{proof}
The two valuation formulas follow from Theorem~\ref{thm:lawD} and
Proposition~\ref{prop:lawN} by addition of exponents: for $j=2k-1$,
\[
 5r-4+\nu_2((r-1)!)-\bigl(5k-4+\nu_2((k-1)!)\bigr)
 -\bigl(5(r-k)+\nu_2((r-k)!)\bigr)=\nu_2\binom{r-1}{k-1},
\]
and for $j=2k$,
\begin{multline*}
 5r-4+\nu_2((r-1)!)-\bigl(5k-2+\nu_2((k-1)!)\bigr)
 -\bigl(5(r-k-1)+2+\nu_2((r-k-1)!)\bigr)\\
 =1+\nu_2\!\left((r-k)\binom{r-1}{k-1}\right).
\end{multline*}
Since $\operatorname{Im}g_{2r-1}
=[v^{2r-1}]\bigl(\operatorname{Im}\operatorname{Num}\cdot
\operatorname{Den}^{-1}\bigr)$, the sum of the $c_{r,j}$ is
$2^{5r-4+\nu_2((r-1)!)}\operatorname{Im}g_{2r-1}$, whose valuation is
$5r-4+\nu_2((r-1)!)-H_r=4(r-1)$ by Theorem~\ref{thm:tangentlaw}.
\end{proof}

\begin{remark}\label{rem:nonlocal}
The individual terms $c_{r,j}$ have small valuations (at most of the
order of $\log_2r$), whereas their sum is divisible by exactly
$2^{4(r-1)}$.  The valuation of $\operatorname{Im}g_{2r-1}$ therefore
cannot be read off from the individual terms of the convolution
$\operatorname{Im}\operatorname{Num}\cdot\operatorname{Den}^{-1}$.
\end{remark}

\section{\texorpdfstring{Application: Jacobi zeros for $\alpha,\beta>-1$}{Application: Jacobi zeros for alpha, beta > -1}}\label{sec:zeros}

Throughout this section $\alpha,\beta>-1$, so that $P_n^{(\alpha,\beta)}$
has $n$ real simple zeros in $(-1,1)$, ordered as in the Introduction.
The purpose of the section is to illustrate that the phase expansion
\eqref{eq:quantisation} can be evaluated in exact arithmetic, with an
accuracy away from the endpoints limited only by the truncation order.
It is not an optimised algorithm, and no timing comparison with the
specialised Gauss--Jacobi solvers
\cite{HaleTownsend2013,GilSeguraTemme2021,BremerYang2020} is intended.
For related asymptotic analyses of Jacobi polynomials and their zeros,
see~\cite{DimitrovSantos,GilSeguraTemme}.

By \eqref{eq:antisym} the phase constants vanish for $\alpha=\beta$, so
the Legendre and Gegenbauer cases need none; for $\alpha\ne\beta$ they
are needed.  Their denominators are known: by~\cite{companion},
$2^{E_r}\kappa_r\in\Zt[\ma,\mb]$ with $E_r=3r-1+\nu_2((r-1)!)$, and the
odd part of $\den\kappa_r$ divides
$\operatorname{lcm}(1,3,\dots,2r-1)$.  The denominators of
$\kappa_1,\dots,\kappa_4$ as polynomials in $(A,B)$ are
$4=2^2$, $96=2^5\!\cdot\!3$, $7680=2^9\!\cdot\!15$ and
$86016=2^{12}\!\cdot\!21$.

Solving \eqref{eq:quantisation} by Newton with a per-node least-term gate (the
truncation order is the last index before the correction terms stop
decreasing), then polishing by a safeguarded Newton on $P_n^{(\alpha,\beta)}$
confined to the cell bounded by the midpoints of the neighbouring reference
angles, gives the following maximum errors in $\theta$.  The reference
zeros are independent of the predictor: they are obtained from
$P_n^{(\alpha,\beta)}$ itself in $40$-digit multiprecision arithmetic,
and each tabulated error is measured against that reference.  In the
table, ``bare'' is the leading approximation~\eqref{eq:leading} (no
correction terms) over all nodes $k=1,\dots,n$; ``gated'' is the
gated Newton solution of $\Xi(\theta)=(k-\frac12)\pi$, over all nodes
or over the bulk; ``polished'' is the result after the final Newton
step on $P_n^{(\alpha,\beta)}$.

\begin{center}
\small
\begin{tabular}{@{}rrrllll@{}}
\toprule
$n$ & $\alpha$ & $\beta$ & bare & gated, all $k$ & gated, bulk & polished\\
\midrule
$30$  & $0.3$  & $-0.2$ & $1.5\cdot10^{-3}$ & $7.6\cdot10^{-5}$ & $1.7\cdot10^{-18}$ & $9.2\cdot10^{-41}$\\
$30$  & $1.5$  & $0.25$ & $7.0\cdot10^{-3}$ & $1.2\cdot10^{-5}$ & $2.6\cdot10^{-19}$ & $4.6\cdot10^{-41}$\\
$60$  & $-0.4$ & $0.7$  & $5.6\cdot10^{-4}$ & $2.7\cdot10^{-5}$ & $5.3\cdot10^{-24}$ & $9.2\cdot10^{-41}$\\
$60$  & $0$    & $0$    & $8.0\cdot10^{-4}$ & $2.1\cdot10^{-5}$ & $1.2\cdot10^{-23}$ & $3.6\cdot10^{-41}$\\
$100$ & $2.5$  & $-0.5$ & $5.1\cdot10^{-3}$ & $8.6\cdot10^{-6}$ & $1.3\cdot10^{-32}$ & $3.7\cdot10^{-40}$\\
\bottomrule
\end{tabular}
\end{center}

``Bulk'' is the middle half of the nodes; the loss of accuracy near
$k=1$ and $k=n$ is the known endpoint effect, where an expansion of
Bessel type must be used instead.

\section*{Computational checks}

The identities of this paper have been checked independently by exact
computation in finite ranges, in rational, Gaussian-rational or symbolic
arithmetic.  In particular, the numbers $g_n$ were computed for
$n\le41$ directly from the amplitude \eqref{eq:F}, and they agree with
\eqref{eq:formal-tangent} (real and imaginary parts) and with
\eqref{eq:tangentformula}; the relation \eqref{eq:wm-kernel} was checked
against these values for $m\le19$; the valuation formula
\eqref{eq:arithlaw} was checked for $m\le159$; and the valuation
formulas of Section~\ref{sec:tangentlayers} were checked for $r\le21$.
The identities of Section~\ref{sec:offdiag} and
Appendix~\ref{app:vertical} that involve boundary values, lattice sums
and their asymptotic expansions were checked numerically to at least
$25$ digits.  None of these computations is used in the proofs, with
the exception of the finite verification in
Proposition~\ref{prop:dictionary}.

\appendix

\section{The sectorial estimates}\label{app:vertical}

This appendix proves the estimates stated in Lemma~\ref{lem:vertical},
the vanishing of the constant $c_*(s_0)$ of the quadrature
\eqref{eq:quadrature} on every admissible lattice
(Theorem~\ref{thm:czero}), and Theorem~\ref{thm:imclosed}.  The order
is as follows: coefficient bounds and the expansion of the amplitude in
a bilateral sector (Subsection~\ref{app:sub-bilateral}); uniform
asymptotics of the closed blocks (Subsection~\ref{app:sub-blocks}); the
uniform expansion of the forcing $\widehat{\mathsf f}_\Gamma$ and of its
lattice sum $T$ (Subsection~\ref{app:sub-lattice}); the relation
between the two boundary determinations, from which the vertical decay
of the quotient $Q$ follows (Subsection~\ref{app:sub-relation}); a
classification lemma for periodic meromorphic functions and the
vanishing of $c_*(s_0)$ (Subsection~\ref{app:sub-czero}); and the proof
of Theorem~\ref{thm:imclosed} (Subsection~\ref{app:sub-imclosed}).  The
only classical inputs are Stirling's formula in sectors, the
Euler--Maclaurin (Hurwitz) expansion, and the removability of bounded
isolated singularities.

\subsection{Coefficient bounds and the bilateral sector expansion}
\label{app:sub-bilateral}

Throughout, $Z=1-2s$, and $F(\alpha;Z)$ denotes the
specialisation $\beta=0$ of the amplitude \eqref{eq:F},
\begin{gather*}
 F(\alpha;Z)=\sum_{n\ge0}G_n(\alpha)\,e_n(Z),\qquad
 e_n(Z)=\frac1{(Z-1)\cdots(Z-n)},\\
 G_n(\alpha)=\sum_{l=0}^na_l(\alpha)\,a_{n-l}(0)\,p^lq^{\,n-l},
\end{gather*}
with $p=1+i$, $q=1-i$ and the Hankel coefficients $a_k$ of
the Introduction.  Two facts proved in the main text are used.  First, by
Lemma~\ref{lem:KdF} the series $F$ is the Kamp\'e de F\'eriet double
series \eqref{eq:KdF} at the midpoint arguments
$(x_0,y_0)=(\frac{1+i}2,\frac{1-i}2)$, grouped by total degree
$\mu+\nu=n$.  Second, by \eqref{eq:midpoint2F1} (an exact identity for
fixed $Z$),
\begin{equation}\label{eq:FH-exact}
 F(\alpha;Z)=y_0^{-Z}(-i)^{1/2}\,\mathcal H(\alpha;Z),
\end{equation}
with an $\alpha$-independent prefactor, so that
\begin{equation}\label{eq:ND-F}
 \mathcal A(s)
 =\frac12\,\frac{\partial_\alpha^2\mathcal H(0;Z)}{\mathcal H(0;Z)}
 =\frac12\,\frac{\partial_\alpha^2F(0;Z)}{F(0;Z)}
\end{equation}
as an identity of functions, not merely of formal series.

\begin{lemma}\label{lem:hankel}
For all $k\ge0$ and all $\alpha\in\mathbb C$,
\[
 |a_k(\alpha)|\le\cosh(\pi|\alpha|)\,k!\,2^{-k},
\qquad\text{and consequently}\qquad
 |G_n(\alpha)|\le\cosh(\pi|\alpha|)\,(n+1)!\,2^{-n/2}.
\]
\end{lemma}

\begin{proof}
From the product form
$a_k(\nu)=\prod_{j=0}^{k-1}\bigl(4\nu^2-(2j+1)^2\bigr)/(8^kk!)$,
\[
 |a_k(\alpha)|\le|a_k(0)|
 \prod_{j=0}^{k-1}\Bigl(1+\frac{4|\alpha|^2}{(2j+1)^2}\Bigr)
 \le|a_k(0)|\prod_{j\ge0}\Bigl(1+\frac{4|\alpha|^2}{(2j+1)^2}\Bigr)
 =|a_k(0)|\cosh(\pi|\alpha|),
\]
the last step being the product formula
$\cos(\pi z/2)=\prod_{j\ge0}(1-z^2/(2j+1)^2)$ at $z=2i|\alpha|$.  For
the central value, $|a_k(0)|=\bigl((2k-1)!!\bigr)^2/(8^kk!)$ and
$\bigl((2k-1)!!\bigr)^2\le(2k-1)!!\,(2k)!!=(2k)!$, so
$|a_k(0)|\le(2k)!/(8^kk!)=k!\binom{2k}k8^{-k}
\le k!\,4^k/8^k=k!\,2^{-k}$.
For $G_n$: $|p|=|q|=\sqrt2$ and
$\sum_{l=0}^{n}l!\,(n-l)!\le(n+1)\,n!=(n+1)!$, hence
$|G_n(\alpha)|\le\cosh(\pi|\alpha|)\,2^{n/2}\,2^{-n}\,(n+1)!
=\cosh(\pi|\alpha|)(n+1)!\,2^{-n/2}$.
\end{proof}

For $0<\theta_0<1$ and $y_{\min}\ge1/(1-\theta_0)$ put
\begin{equation}\label{eq:Omega-def}
 \Omega_{\theta_0}:=\bigl\{Z:\ |\operatorname{Re}Z|\le\theta_0\,|\operatorname{Im}Z|,
 \ |\operatorname{Im}Z|\ge y_{\min}\bigr\}.
\end{equation}
Both signs of $\operatorname{Im}Z$ are allowed: the region is symmetric
under conjugation, and every vertical strip
$|\operatorname{Re}Z|\le h$, $|\operatorname{Im}Z|\ge h/\theta_0$
lies inside it.

\begin{lemma}\label{lem:tailsum}
Let $Z\in\Omega_{\theta_0}$ and $M\ge0$.  Then $(j+1)\le\sqrt2\,|Z-j|$ for
every integer $j\ge0$, and
\[
 \tau_M(Z):=\sum_{n\ge M+1}(n+1)!\,2^{-n/2}
 \prod_{j=M+1}^{n}|Z-j|^{-1}
 \ \le\ (M+1)!\,2^{-M/2}\,\bigl(18|Z|+4\bigr).
\]
\end{lemma}

\begin{proof}
Write $a=\operatorname{Re}Z$, $b=|\operatorname{Im}Z|$, so $|a|\le\theta_0 b$
and $b\ge1/(1-\theta_0)$, i.e.\ $\theta_0 b+1\le b$.  Then
$2|Z-j|^2-(j+1)^2=2(j-a)^2+2b^2-(j+1)^2$; as a function of $j$ this
quadratic has positive leading coefficient and minimum value
$2b^2-2(a+1)^2\ge2b^2-2(\theta_0 b+1)^2\ge0$, which proves
$(j+1)\le\sqrt2|Z-j|$ for all $j\ge0$.  Consequently the terms
$t_n:=(n+1)!\,2^{-n/2}\prod_{j=M+1}^n|Z-j|^{-1}$ satisfy
$t_n/t_{n-1}=(n+1)/(\sqrt2\,|Z-n|)\le1$, so
$t_n\le t_M=(M+1)!\,2^{-M/2}$ for every $n>M$.  Moreover for
$n\ge18|Z|$ one has $|Z-n|\ge n-|Z|\ge n(1-\tfrac1{18})$ and hence
$t_n/t_{n-1}\le\frac{18\,(n+1)}{17\sqrt2\,n}\le\frac45$ once
$n\ge18|Z|\ge18$.  Splitting the sum at $18|Z|$,
\[
 \tau_M\le t_M\bigl(18|Z|\bigr)
 +t_M\sum_{r\ge1}(4/5)^r\le t_M(18|Z|+4).\qedhere
\]
\end{proof}

\begin{proposition}\label{prop:bilateral}
Fix $0<\theta_0<1$.  Then:
\begin{enumerate}
\item[\upshape(i)] the double series \eqref{eq:KdF} converges
absolutely for every $Z\notin\{1,2,3,\dots\}$ and every $\alpha$; its
grouping by $\mu+\nu=n$ is the series $F(\alpha;Z)$ above, which is
therefore analytic in $(\alpha,Z)$ on
$\mathbb C\times(\mathbb C\setminus\{1,2,\dots\})$ and satisfies
\eqref{eq:FH-exact} there;
\item[\upshape(ii)] for every $N$ there are $C_N=C_N(\theta_0)$ and
$y_{\min}=y_{\min}(\theta_0,N)$ such that, uniformly for $Z\in\Omega_{\theta_0}$ and
$|\alpha|\le1$,
\[
 \Bigl|F(\alpha;Z)-\sum_{m=0}^{N}c_m(\alpha)Z^{-m}\Bigr|
 \le C_N\,|Z|^{-N-1},
\]
where $c_m(\alpha)=\sum_{n\le m}G_n(\alpha)\stir mn$ are the plain
coefficients of Lemma~\ref{lem:degree}; the same bound holds for
$\partial_\alpha^kF$ and $\partial_\alpha^kc_m$, $k=1,2$;
\item[\upshape(iii)] consequently, uniformly for $Z\in\Omega_{\theta_0}$
(\emph{both} signs of $\operatorname{Im}Z$),
\[
 \Bigl|\mathcal A(s)-\sum_{n=1}^{N}g_nZ^{-n}\Bigr|\le
 C_N'\,|Z|^{-N-1}.
\]
\end{enumerate}
\end{proposition}

\begin{proof}
(i) With $|x_0|=|y_0|=2^{-1/2}$, the modulus of the
$(\mu,\nu)$-term of \eqref{eq:KdF} is at most
$C(\alpha)\,\frac{\mu!\,\nu!}{|(1-Z)_{\mu+\nu}|}\,2^{-(\mu+\nu)/2}
\mu^{O(1)}\nu^{O(1)}$ by Lemma~\ref{lem:hankel} applied to each
Pochhammer pair, and
$\mu!\,\nu!/(\mu+\nu)!=\binom{\mu+\nu}{\mu}^{-1}\le1$ while
$|(1-Z)_m|\ge c(Z)\,m!\,m^{-\operatorname{Re}Z-1}$ for $m$ large, by
Stirling; the sum over $\mu+\nu=m$ then converges geometrically in
$m$.  Grouping the absolutely convergent double series by $\mu+\nu=n$
gives $G_ne_n$ exactly as in the proof of Lemma~\ref{lem:KdF}.
Identity \eqref{eq:FH-exact} holds for each fixed generic $Z$ by
Theorem~\ref{thm:appellreduce} and \eqref{eq:midpoint2F1}, and both
sides are analytic off the indicated sets, so it holds on the whole
common domain.

(ii) Split $F=\sum_{n\le N}G_ne_n+\mathrm{Rem}_N$.  Each $e_n$, $n\le N$, is
analytic at $Z=\infty$ with
$e_n=\sum_{m=n}^{N}\stir mnZ^{-m}+O_N(|Z|^{-N-1})$ uniformly for
$|Z|\ge2N+2$ (finitely many geometric expansions of
$Z^{-n}\prod_{j\le n}(1-j/Z)^{-1}$), which produces the polynomial part
$\sum_{m\le N}c_mZ^{-m}$ with the stated coefficients plus an error
$O_N(|Z|^{-N-1})$.  For the tail, reserve the first $N+2$ factors:
for $n\ge N+2$,
\[
 \prod_{j=1}^n|Z-j|\ \ge\ |\operatorname{Im}Z|^{\,N+2}
 \prod_{j=N+3}^{n}|Z-j|,
\]
so by Lemmas~\ref{lem:hankel} and~\ref{lem:tailsum} (with $M=N+2$),
\begin{multline*}
 |\mathrm{Rem}_N|\le\sum_{n=N+1}^{N+2}\frac{|G_n|}{|\operatorname{Im}Z|^{\,n}}
 +\frac{\cosh\pi}{|\operatorname{Im}Z|^{N+2}}\,\tau_{N+2}(Z)\\
 \le\frac{C_N''}{|\operatorname{Im}Z|^{N+1}}
 +\frac{C_N'''\,|Z|}{|\operatorname{Im}Z|^{N+2}}
 \le C_N\,|Z|^{-N-1},
\end{multline*}
using $|\operatorname{Im}Z|\ge(1+\theta_0^2)^{-1/2}|Z|$ on $\Omega_{\theta_0}$.
The same argument applies verbatim to
$\partial_\alpha^kF=\sum_n\partial_\alpha^kG_n\,e_n$ after bounding
$|\partial_\alpha^kG_n(\alpha)|\le C(n+1)!\,2^{-n/2}$ for $|\alpha|\le2$
by the Cauchy formula on the circle $|\alpha'-\alpha|=1$ applied to
Lemma~\ref{lem:hankel}; alternatively, apply Cauchy's formula directly
to the remainder $\mathrm{Rem}_N(\alpha;Z)$, which is analytic in $\alpha$.

(iii) $F$ is even in $\alpha$ (each $a_l$ is a function of $\alpha^2$),
so $\partial_\alpha F(0;Z)=0$ and, by \eqref{eq:ND-F},
$\mathcal A=\tfrac12\partial_\alpha^2F(0;Z)/F(0;Z)$.  By (ii),
$F(0;Z)=1+O(|Z|^{-1})$ uniformly on $\Omega_{\theta_0}$, so for
$y_{\min}$ large $|F(0;Z)|\ge\frac12$ and the quotient of the two uniform
expansions is again a uniform expansion, whose coefficients are, by
Corollary~\ref{cor:gausstangent}, exactly $g_n$.
\end{proof}

\begin{remark}
\label{rem:aperture}
The cone $|\operatorname{Re}Z|\le\theta_0|\operatorname{Im}Z|$ with
$\theta_0<1$ cannot be widened by the termwise estimate above: for
directions closer to the real axis the individual terms $G_ne_n$ with
$n\asymp|Z|$ grow exponentially and the series converges only through
cancellation.  The applications of Lemma~\ref{lem:vertical} in this paper
do not require a wider aperture: the matching of
Subsection~\ref{app:sub-imclosed} is performed along a single vertical
line, the decay statements of Subsection~\ref{app:sub-czero} concern
vertical strips, and the coefficients of a Poincar\'e expansion along
one ray are unique.  The mirror symmetry used repeatedly below is the
Schwarz reflection $\mathcal A^{\downarrow}(s)
=\overline{\mathcal A^{\uparrow}(\bar s)}$, valid because all four
blocks $\mathsf S_0,\mathsf K_0,\mathsf S_2,\mathsf K_2$ are real on
the real axis.
\end{remark}

\subsection{Uniform asymptotics of the closed blocks}
\label{app:sub-blocks}

\begin{lemma}\label{lem:blocksasy}
Uniformly on every closed subsector of
$|\arg(\pm is)|<\pi/2$, i.e.\ off the real axis, and in particular on
every vertical strip $|\operatorname{Re}s|\le h$,
$|\operatorname{Im}s|\ge1$:
\begin{enumerate}
\item[\upshape(i)] $\psi'(s)$, $Q_{\rm rat}(s)$ and the convergent
trigamma blocks of \eqref{eq:ReND-closed} are asymptotic to their
(Bernoulli) Stirling series, with remainders
$O_N(|s|^{-N})$ after $N$ terms, and
$|\psi'(s)|+|Q_{\rm rat}(s)|\le C/|\operatorname{Im}s|$;
\item[\upshape(ii)] $\tan\frac{\pi s}2=\pm i+O(e^{-\pi|\operatorname{Im}s|})$
for $\operatorname{Im}s\to\pm\infty$, uniformly in
$\operatorname{Re}s$, and likewise
$\cot\frac{\pi s}2=\mp i+O(e^{-\pi|\operatorname{Im}s|})$,
$\sec^2\frac{\pi s}2$, $\csc^2(\pi s)=O(e^{-\pi|\operatorname{Im}s|})$;
\item[\upshape(iii)] $\widehat{\mathsf r}$ of \eqref{eq:rhatdef} is
asymptotic to its Poincar\'e series \eqref{eq:rhatformal} uniformly on
$|\arg s|\le\pi-\delta$, for every $\delta>0$.
\end{enumerate}
\end{lemma}

\begin{proof}
(i) and (iii) are Stirling's formula with remainder in sectors
\cite[\S5.11]{DLMF} applied to the explicit Gamma quotients; the decay
bound in (i) follows from $\psi'(z)=O(1/|z|)$ on
$|\operatorname{Im}z|\ge\tfrac12$ together with
$\psi'(z+1)=\psi'(z)-z^{-2}$ to reach the arguments $(1-s)/2$,
$1-s/2$.  (ii) is elementary:
$\tan\frac{\pi s}2\mp i=\mp2i\,e^{\pm i\pi s}/(1+e^{\pm i\pi s})$, and
$|e^{\pm i\pi s}|=e^{-\pi|\operatorname{Im}s|}$ in the indicated
half-planes.
\end{proof}

\subsection{The Gamma forcing and its lattice sum}
\label{app:sub-lattice}

\begin{lemma}
\label{lem:fhatglobal}
Let $\widehat{\mathsf f}_\Gamma$ be \eqref{eq:fgammahat}.  Then:
\begin{enumerate}
\item[\upshape(i)] there is $C$ with
$|\widehat{\mathsf f}_\Gamma(\varsigma)|\le C\,|\varsigma|^{-2}$ uniformly on
$|\operatorname{Im}\varsigma|\ge1$, \emph{including}
$\operatorname{Re}\varsigma\to-\infty$;
\item[\upshape(ii)] for every $M$ there is $C_M$ such that, uniformly on
$\operatorname{Im}\varsigma\ge1$,
\[
 \Bigl|\widehat{\mathsf f}_\Gamma(\varsigma)
 -\sum_{j=2}^{M}\varphi_j\,\varsigma^{-j}\Bigr|
 \le C_M\bigl(|\varsigma|^{-M-1}
 +|\varsigma|^{-2}e^{-\pi\operatorname{Im}\varsigma}\bigr),
\]
where $\sum_j\varphi_j\varsigma^{-j}$ is the formal series
\eqref{eq:fgammahat}; the mirror statement holds for
$\operatorname{Im}\varsigma\le-1$ by
$\widehat{\mathsf f}_\Gamma(\bar\varsigma)
=\overline{\widehat{\mathsf f}_\Gamma(\varsigma)}$.
\end{enumerate}
\end{lemma}

\begin{proof}
Everything reduces to the Gamma quotient in
$\widehat{\mathsf r}(\varsigma)=\varsigma\,
[\Gamma(\varsigma/2)/\Gamma((1+\varsigma)/2)]^2$.  On the right regime
$|\arg\varsigma|\le\tfrac34\pi$, Stirling gives
$\Gamma(\varsigma/2)/\Gamma((1+\varsigma)/2)=(\varsigma/2)^{-1/2}(1+O(1/\varsigma))$
uniformly, whence $|\widehat{\mathsf r}(\varsigma)|\le C$ and the
Poincar\'e expansion \eqref{eq:rhatformal} with uniform remainders
(Lemma~\ref{lem:blocksasy}(iii)).  On the left regime
$\tfrac34\pi\le|\arg\varsigma|\le\pi$, $\operatorname{Im}\varsigma\ge1$,
apply the two reflection formulas used in the proof of
Lemma~\ref{lem:rhat}:
\begin{equation}\label{eq:rhat-left}
 \frac{\Gamma(\varsigma/2)}{\Gamma((1+\varsigma)/2)}
 =\cot\Bigl(\frac{\pi\varsigma}2\Bigr)\,
 \frac{\Gamma\bigl(\frac{1-\varsigma}2\bigr)}
      {\Gamma\bigl(1-\frac\varsigma2\bigr)} ,
\end{equation}
whose right-hand Gamma arguments have real part tending to $+\infty$.
The Beta-integral bound, for $\operatorname{Re}x>0$,
\[
 \Bigl|\frac{\Gamma(x)}{\Gamma(x+\frac12)}\Bigr|
 =\frac{|\mathrm B(x,\tfrac12)|}{\Gamma(\tfrac12)}
 \le\frac{\mathrm B(\operatorname{Re}x,\tfrac12)}{\sqrt\pi}
 \le C\,(\operatorname{Re}x)^{-1/2}
\]
with $x=\frac{1-\varsigma}2$, together with
$\cot(\pi\varsigma/2)=-i+O(e^{-\pi\operatorname{Im}\varsigma})$
(Lemma~\ref{lem:blocksasy}(ii)), gives
$|\widehat{\mathsf r}(\varsigma)|\le C$ on the left regime as well; (i)
follows since
$\widehat{\mathsf f}_\Gamma=-\tfrac{(2\varsigma-1)}{4\varsigma^2(\varsigma-1)}
\widehat{\mathsf r}$.
For (ii): on the left regime, \eqref{eq:rhat-left} squared gives
\[
 \widehat{\mathsf r}(\varsigma)
 =-\varsigma\Bigl[\frac{\Gamma\bigl(\frac{1-\varsigma}2\bigr)}
      {\Gamma\bigl(1-\frac\varsigma2\bigr)}\Bigr]^2
 \bigl(1+O(e^{-\pi\operatorname{Im}\varsigma})\bigr),
\]
and the squared quotient has a uniform Stirling expansion in
$\varsigma^{-1}$ (right half-plane arguments).  The resulting Poincar\'e
series coincides with \eqref{eq:rhatformal}: both are asymptotic
expansions of the same function $\widehat{\mathsf r}$ on the open
overlap sector $\tfrac34\pi<|\arg\varsigma|<\pi-\delta$ (where the cot
correction is exponentially small in $|\varsigma|$), and the coefficients
of a Poincar\'e expansion on a sector are unique.  (Arithmetically the
same coincidence is the parity identity of Lemma~\ref{lem:parityF}.)
Combining the two regimes, with the exponentially small cot correction
retained additively, gives (ii).
\end{proof}

\begin{lemma}\label{lem:fhatpoles}
$\widehat{\mathsf f}_\Gamma$ is meromorphic on $\mathbb C$ with:
double zeros at the negative odd integers; a pole of order $3$ at
$\varsigma=0$, a simple pole at $\varsigma=1$, double poles at the negative
even integers, and no other poles.  The principal-part coefficients at
$\varsigma=-2m$ are $O(m^{-2})$.  Consequently
$T(s)=\sum_{k\ge0}\widehat{\mathsf f}_\Gamma(s-1-2k)$ converges
normally on compact sets avoiding $\mathbb Z$ and defines a meromorphic
function on $\mathbb C$ with poles contained in $\mathbb Z$;
$T$ has (at least) double zeros at the negative even integers, at most
double poles at the negative odd integers, and satisfies
\begin{equation}\label{eq:T-global}
 |T(s)|\le\frac{C}{|\operatorname{Im}s|},\qquad
 |\operatorname{Im}s|\ge1,
\end{equation}
uniformly in $\operatorname{Re}s$.  Finally
$T(s+2)-T(s)=\widehat{\mathsf f}_\Gamma(s+1)$.
\end{lemma}

\begin{proof}
$\Gamma((1+\varsigma)/2)^{-2}$ has double zeros at
$\varsigma=-1,-3,\dots$; $\Gamma(\varsigma/2)^{2}$ has double poles at
$\varsigma=0,-2,-4,\dots$; combining with the elementary factor
$-(2\varsigma-1)/\bigl(4\varsigma^2(\varsigma-1)\bigr)$ of
\eqref{eq:fgammahat} gives the stated orders (at $\varsigma=0$:
$\widehat{\mathsf r}\sim4/(\pi\varsigma)$, so
$\widehat{\mathsf f}_\Gamma\sim-1/(\pi\varsigma^3)$; at $\varsigma=1$:
$\widehat{\mathsf r}(1)=\pi$).  At $\varsigma=-2m$ the double-pole
coefficient of $\widehat{\mathsf r}$ is
$-8m\,\Gamma(m+\tfrac12)^2/\bigl(\pi^2(m!)^2\bigr)=O(1)$ by the
reflection formula, and the elementary factor contributes an extra
$O(m^{-2})$.  The term $\widehat{\mathsf f}_\Gamma(s-1-2k)$ of $T$ has
poles only at integers: at $s=1+2k$, at $s=2+2k$, and at the odd
integers $s=1+2k-2m$, $m\ge1$, where it has a double pole with
principal-part coefficients $O(m^{-2})$.  At a fixed odd integer these
coefficients are therefore summable in $k$, and away from small discs
around the integers the terms are dominated by
Lemma~\ref{lem:fhatglobal}(i); this gives normal convergence and
meromorphy, and shows that $T$ has at most double poles at the negative
odd integers.  The double zeros of $T$ at $s=-2m$ hold
because every term is then evaluated at a negative odd integer, where
$\widehat{\mathsf f}_\Gamma$ has a double zero as a function of $s$.
For \eqref{eq:T-global}: with $\operatorname{Im}s=Y$, $|Y|\ge1$, all
lattice points $\varsigma_k=s-1-2k$ have $|\operatorname{Im}\varsigma_k|=|Y|$,
so by Lemma~\ref{lem:fhatglobal}(i),
$|T|\le C\sum_{k\ge0}\bigl(Y^2+(\operatorname{Re}s-1-2k)^2\bigr)^{-1}
\le C'/|Y|$ by comparison with
$\int_{-\infty}^{\infty}(Y^2+u^2)^{-1}du=\pi/|Y|$.  The difference
equation is immediate from the definition.
\end{proof}

\begin{proposition}
\label{prop:latticetail}
Let $\omega_j$ be the coefficients of the formal series
$\widehat Q_\Gamma=(2\sinh\partial_s)^{-1}\widehat{\mathsf f}_\Gamma\in
s^{-1}\mathbb Q[[s^{-1}]]$ of \eqref{eq:Qgammahat}.  For every
$\delta>0$ and every $N$ there are $C_{\delta,N}$ and $s_{\min}$ such that,
uniformly on the sector
$\operatorname{Im}s\ge\delta|s|$, $|s|\ge s_{\min}$,
\[
 \Bigl|T(s)-\sum_{j=1}^{N}\omega_j\,s^{-j}\Bigr|
 \ \le\ C_{\delta,N}\,|s|^{-N-1}
 \ +\ C_\delta\,\frac{e^{-\pi\operatorname{Im}s}}{\operatorname{Im}s}.
\]
The mirror statement holds on $\operatorname{Im}s\le-\delta|s|$ with
the same (rational) coefficients, by Schwarz reflection.
\end{proposition}

\begin{proof}
Fix $M\ge N+2$.  Insert Lemma~\ref{lem:fhatglobal}(ii) termwise:
\[
 T(s)=\sum_{j=2}^{M}\varphi_j\,\Sigma_j(s)+\mathrm{Err}_1+\mathrm{Err}_2,
 \qquad
 \Sigma_j(s):=\sum_{k\ge0}(s-1-2k)^{-j},
\]
with
$|\mathrm{Err}_1|\le C_M\sum_{k\ge0}|s-1-2k|^{-M-1}$ and
$|\mathrm{Err}_2|\le C\,e^{-\pi\operatorname{Im}s}
\sum_{k\ge0}|s-1-2k|^{-2}$.  For the error sums split at
$k\le|s|$ and $k>|s|$: the near terms number at most $|s|+1$ and each
is at most $(\operatorname{Im}s)^{-M-1}\le(\delta|s|)^{-M-1}$, giving
$\le2\delta^{-M-1}|s|^{-M}$; the far terms are dominated by
$\sum_{k>|s|}(2k-1-|s|)^{-M-1}\le C|s|^{-M}$.  Hence
$|\mathrm{Err}_1|\le C_{\delta,M}'|s|^{-M}$, and similarly
$|\mathrm{Err}_2|\le C_\delta\,e^{-\pi\operatorname{Im}s}/\operatorname{Im}s$ as
in \eqref{eq:T-global}.  Next, with $a=(1-s)/2$ (which satisfies
$|\arg a|\le\pi-\delta'$, $\delta'=\delta'(\delta)>0$, on the sector),
\[
 \Sigma_j(s)=(-2)^{-j}\,\zeta(j,a),
\]
and the Hurwitz zeta function has the classical uniform large-$a$
expansion \cite[\S25.11(xii)]{DLMF}
\[
 \zeta(j,a)=\frac{a^{1-j}}{j-1}+\frac{a^{-j}}2
 +\sum_{r=1}^{R}\binom{j+2r-2}{2r-1}\frac{\Ber_{2r}}{2r}\,a^{-j-2r+1}
 +O_{j,R}\bigl(|a|^{-j-2R-1}\bigr)
\]
uniformly on $|\arg a|\le\pi-\delta'$.  Substituting these finitely
many expansions (each truncated at total order $N+1$ in $s^{-1}$) and
re-expanding $a=(1-s)/2$ in powers of $s^{-1}$ produces a fixed formal
series $\widehat T=\sum_j\hat\omega_js^{-j}$, with no constant term (the
leading contribution of $\Sigma_j$ is of order $s^{1-j}$ and
$\varphi_j$ starts at $j=2$), such that
$T(s)=\sum_{j\le N}\hat\omega_js^{-j}+O_{\delta,N}(|s|^{-N-1})
+O_\delta(e^{-\pi\operatorname{Im}s}/\operatorname{Im}s)$ on the
sector.  It remains to identify $\widehat T=\widehat Q_\Gamma$.  The
exact difference equation $T(s+2)-T(s)=\widehat{\mathsf f}_\Gamma(s+1)$
(Lemma~\ref{lem:fhatpoles}), the uniform expansion just proved (applied
at $s$ and $s+2$, both in the sector), and
Lemma~\ref{lem:fhatglobal}(ii) at $\varsigma=s+1$ force the coefficientwise
identity $\widehat T(s+2)-\widehat T(s)
=\widehat{\mathsf f}_\Gamma$-series$(s+1)$; equivalently, in the
shifted variable $u=s+1$, $\widehat T(u+1)-\widehat T(u-1)$ equals the
series \eqref{eq:fgammahat} at $u$.  By
Proposition~\ref{prop:formal-sinh} the unique solution of this
equation in $s^{-1}\mathbb Q[[s^{-1}]]$ is $\widehat Q_\Gamma$, so
$\widehat T=\widehat Q_\Gamma$ and $\hat\omega_j=\omega_j$.
\end{proof}

Propositions~\ref{prop:bilateral} and~\ref{prop:latticetail} and
Lemma~\ref{lem:blocksasy} together prove Lemma~\ref{lem:vertical}, in
the uniform quantitative form stated there.

\subsection{The two-component relation and vertical decay}\label{app:sub-relation}

With $\mathsf S_0=\mathsf S(0;s)$,
$\mathsf S_2=\partial_\alpha^2\mathsf S(\alpha;s)|_{\alpha=0}$ and
$\mathsf K_0$, $\mathsf K_2$ as in Theorem~\ref{thm:quad}, define the two
boundary components
\[
 \mathcal A^{\uparrow}(s)
 =-\frac{\psi'(s)}4
 +\frac{\mathsf S_2+i\pi\mathsf K_2}{2(\mathsf S_0+i\pi\mathsf K_0)},
 \qquad
 \mathcal A^{\downarrow}(s)
 =-\frac{\psi'(s)}4
 +\frac{\mathsf S_2-i\pi\mathsf K_2}{2(\mathsf S_0-i\pi\mathsf K_0)},
\]
meromorphic functions of $s$ with
$\mathcal A^{\uparrow}=\mathcal A$ (the upper boundary determination),
$\mathcal A^{\downarrow}(s)=\overline{\mathcal A^{\uparrow}(\bar s)}$,
and $\mathcal A^{\downarrow}=\overline{\mathcal A}$ on the real axis.

\begin{lemma}\label{lem:tworelation}
As an identity of meromorphic functions on $\mathbb C$,
\begin{equation}\label{eq:tworelation}
 \Bigl(\tan\frac{\pi s}2-i\Bigr)\mathcal A^{\uparrow}(s)
 +\Bigl(\tan\frac{\pi s}2+i\Bigr)\mathcal A^{\downarrow}(s)
 =\tan\frac{\pi s}2\,\Bigl(Q(s)-\frac{\psi'(s)}2\Bigr),
\end{equation}
with $Q=\mathsf K_2/\mathsf K_0$ as in Theorem~\ref{thm:quad}.
\end{lemma}

\begin{proof}
Write $\mathfrak t=\tan\frac{\pi s}2=\pi\mathsf K_0/\mathsf S_0$
(Lemma~\ref{lem:quadeval}).  From $1+i\mathfrak t=i(\mathfrak t-i)$ and
$1-i\mathfrak t=-i(\mathfrak t+i)$,
\[
 \frac{\mathfrak t-i}{1+i\mathfrak t}=-i,\qquad
 \frac{\mathfrak t+i}{1-i\mathfrak t}=i .
\]
Hence
\[
 (\mathfrak t-i)\,\frac{\mathsf S_2+i\pi\mathsf K_2}{\mathsf S_0(1+i\mathfrak t)}
 +(\mathfrak t+i)\,\frac{\mathsf S_2-i\pi\mathsf K_2}{\mathsf S_0(1-i\mathfrak t)}
 =\frac{-i(\mathsf S_2+i\pi\mathsf K_2)+i(\mathsf S_2-i\pi\mathsf K_2)}
 {\mathsf S_0}
 =\frac{2\pi\mathsf K_2}{\mathsf S_0}
 =2\mathfrak t\,Q ,
\]
and adding the $-\psi'/4$ parts, whose coefficients sum to $2\mathfrak t$,
gives \eqref{eq:tworelation}.
\end{proof}

\begin{corollary}
\label{cor:Qvert}
For every $h>0$ there are $C(h)$ and $y'_{\min}(h)$ with
\[
 |Q(s)|\le\frac{C(h)}{|\operatorname{Im}s|},
 \qquad |\operatorname{Re}s|\le h,\
 |\operatorname{Im}s|\ge y'_{\min}(h),
\]
uniformly on the strip; in particular $Q\to0$ at both vertical ends of
every strip, uniformly in $\operatorname{Re}s$ on compact sets.
\end{corollary}

\begin{proof}
Solve \eqref{eq:tworelation} for $Q$:
\[
 Q=\frac{\psi'(s)}2+\frac{(\mathfrak t-i)\mathcal A^{\uparrow}
 +(\mathfrak t+i)\mathcal A^{\downarrow}}{\mathfrak t}.
\]
The vertical strip $|\operatorname{Re}s|\le h$,
$|\operatorname{Im}s|\ge y'_{\min}$ maps under $Z=1-2s$ into
$\Omega_{1/2}$ once $y'_{\min}\ge2(2h+1)$ and $y'_{\min}$ is at least
the threshold of Proposition~\ref{prop:bilateral}, so
Proposition~\ref{prop:bilateral}(iii) with $N=1$ gives
$|\mathcal A^{\uparrow}(s)|\le C/|Z|\le C'/|\operatorname{Im}s|$ there;
the same bound holds for $\mathcal A^{\downarrow}(s)
=\overline{\mathcal A^{\uparrow}(\bar s)}$ since $\bar s$ lies in the
mirror strip.  By Lemma~\ref{lem:blocksasy},
$|\psi'(s)|\le C/|\operatorname{Im}s|$,
$|\mathfrak t\mp i|\le3e^{-\pi|\operatorname{Im}s|}$ (for $\operatorname{Im}s\to\pm\infty$ respectively) and
$|\mathfrak t|\ge\tfrac12$ for $|\operatorname{Im}s|\ge1$.  The claim follows.
\end{proof}

\subsection{The classification lemma and the vanishing of the lattice
constant}\label{app:sub-czero}

\begin{lemma}\label{lem:cylclass}
Let $f$ be meromorphic on $\mathbb C$, $2$-periodic, with all poles
contained in $\mathbb Z$, of order at most $k_1$ at the odd integers and
at most $k_0$ at the even integers, and suppose that
$f(s)\to0$ as $\operatorname{Im}s\to+\infty$ and as
$\operatorname{Im}s\to-\infty$, uniformly for
$\operatorname{Re}s\in[0,2]$.  Then
\[
 f(s)=\frac{w\,n_1(w)}{(w+1)^{k_1}\,(w-1)^{k_0}},\qquad w=e^{i\pi s},
\]
with $n_1$ a polynomial of degree at most $k_1+k_0-2$ (and $f\equiv0$
when $k_1+k_0\le1$).  In particular, if $k_1\le1$ and $k_0=0$ then
$f\equiv0$; if $k_1=2$ and $k_0=0$ then
$f\in\mathbb C\cdot\sec^2\frac{\pi s}2$.
\end{lemma}

\begin{proof}
The map $s\mapsto w=e^{i\pi s}$ identifies the cylinder
$\mathbb C/2\mathbb Z$ with $\mathbb C^*$; a $2$-periodic meromorphic
function $f$ is therefore of the form $f(s)=R(e^{i\pi s})$ with $R$
meromorphic on $\mathbb C^*$.  The poles of $f$ correspond to $w=-1$
(odd integers) and $w=+1$ (even integers), with the stated orders,
since $w\mp1$ vanishes to first order in $s$ at the corresponding
integers.  The end $\operatorname{Im}s\to+\infty$ is $w\to0$ and the
end $\operatorname{Im}s\to-\infty$ is $w\to\infty$; the hypothesis says
that $R(w)\to0$ as $w\to0$ and as $w\to\infty$ (uniform decay on one
period suffices, every $w$ near $0$ or $\infty$ having a preimage with
$\operatorname{Re}s\in[0,2]$).  Hence $R$ is bounded near $0$ and
$\infty$, these singularities are removable with value $0$, and $R$ is
meromorphic on the Riemann sphere, that is, rational, with poles only
at $-1$ and $+1$, of orders at most $k_1$ and $k_0$.  Thus
$R=n_0(w)/((w+1)^{k_1}(w-1)^{k_0})$ with a polynomial $n_0$ of degree at
most $k_1+k_0-1$ (zero at $\infty$) and $n_0(0)=0$ (zero at $0$): hence
$n_0=w\,n_1$ with $\deg n_1\le k_1+k_0-2$.  For $k_1=2$, $k_0=0$ this is
$R=c\,w/(w+1)^2$, and $4w/(w+1)^2=\sec^2\frac{\pi s}2$.
\end{proof}

\begin{theorem}\label{thm:czero}
Define
\[
 \mathcal D(s):=Q(s)-Q_{\rm rat}(s)+\cot\frac{\pi s}2\;T(s).
\]
Then $\mathcal D\equiv0$.  Equivalently, in the decomposition
\eqref{eq:Qdecomp} one has $c_*(s_0)=0$ for every admissible step-two
lattice, real or complex, and \eqref{eq:Edef} holds with $c_*(s_0)=0$.
\end{theorem}

\begin{proof}
\emph{$\mathcal D$ is meromorphic with poles in $\mathbb Z$.}
$\mathsf K_0=\Gamma(1+\tfrac s2)/(\Gamma(1+s)\Gamma(1-\tfrac s2))$ is
entire, with zeros exactly at the negative odd integers and the positive
even integers.  $\mathsf K_2$ is analytic on
$\operatorname{Re}s<\tfrac12$: the defining series converges locally
uniformly there; at a nonpositive integer $s=-m$ the only singular
inner terms are those with $j=m$, present exactly for $n\ge m+1$, and
there the double zero of $((s)_n/n!)^2$ at $s=-m$ cancels the double
pole of $(s+m)^{-2}$, while the finitely many terms with $n\le m$
contain no factor $(s+m)^{-2}$.  It extends meromorphically to
$\mathbb C$ by the recurrence \eqref{eq:K2rec}, whose coefficients
vanish only at integers.  Hence $Q=\mathsf K_2/\mathsf K_0$ is
meromorphic with poles in $\mathbb Z$; so are $Q_{\rm rat}$,
$\cot\frac{\pi s}2$ and $T$ (Lemma~\ref{lem:fhatpoles}).

\emph{$\mathcal D$ is $2$-periodic.}  By \eqref{eq:Qdiff} at $s+1$,
$Q(s+2)-Q(s)=\mathsf f(s+1)$; by the trigamma recurrence,
$Q_{\rm rat}(s+2)-Q_{\rm rat}(s)=\mathsf f_{\rm rat}(s+1)$; by
Lemma~\ref{lem:fhatpoles},
$T(s+2)-T(s)=\widehat{\mathsf f}_\Gamma(s+1)$; and
$\tan\frac{\pi(s+1)}2=-\cot\frac{\pi s}2$, so
$\mathsf f_\Gamma(s+1)
=-\cot\frac{\pi s}2\,\widehat{\mathsf f}_\Gamma(s+1)$
(Lemma~\ref{lem:rhat}).  Hence
\[
 \mathcal D(s+2)-\mathcal D(s)=\mathsf f(s+1)-\mathsf f_{\rm rat}(s+1)
 +\cot\frac{\pi s}2\,\widehat{\mathsf f}_\Gamma(s+1)
 =\mathsf f_\Gamma(s+1)
 +\cot\frac{\pi s}2\,\widehat{\mathsf f}_\Gamma(s+1)=0 .
\]

\emph{Pole orders.}  Because $\mathcal D$ is $2$-periodic, the pole
order on each residue class may be read off at any one representative.
At $s=-2m$, $m\ge1$: $\mathsf K_0(-2m)=(-1)^m\binom{2m}m\ne0$ and
$\mathsf K_2$ is analytic there, so $Q$ is analytic;
$Q_{\rm rat}$ is analytic (its trigamma arguments are
$\tfrac12+m$ and $1+m$); and the simple pole of $\cot\frac{\pi s}2$ is
cancelled by the double zero of $T$ (Lemma~\ref{lem:fhatpoles}).  Hence
$\mathcal D$ is analytic on the even class.  At $s=-(2m+1)$:
$\mathsf K_0$ has a simple zero and $\mathsf K_2$ is analytic, so $Q$
has at most a simple pole; $Q_{\rm rat}$ is analytic; and
$\cot\frac{\pi s}2$ has a simple zero, while $T$ has at most a double
pole.  Hence $\mathcal D$ has at most a simple pole on the odd class.

\emph{Decay at the ends.}  On the strip
$\operatorname{Re}s\in[0,2]$, $|\operatorname{Im}s|\ge y'_{\min}$, there
is a constant $c$ such that
$|Q|\le c/|\operatorname{Im}s|$ by Corollary~\ref{cor:Qvert},
$|Q_{\rm rat}|\le c/|\operatorname{Im}s|$ by
Lemma~\ref{lem:blocksasy}(i), and $|T|\le c/|\operatorname{Im}s|$ by
\eqref{eq:T-global}, while $\cot\frac{\pi s}2$ is bounded.  Hence
$\mathcal D(s)\to0$ uniformly on the strip as
$\operatorname{Im}s\to\pm\infty$.

By Lemma~\ref{lem:cylclass} with $k_1=1$, $k_0=0$,
$\mathcal D\equiv0$.  On any step-two lattice $\{s_0-2k\}$ avoiding the
poles, comparison with \eqref{eq:Qdecomp} gives
$c_*(s_0)=\mathcal D(s_0)=0$.
\end{proof}

\begin{corollary}\label{cor:Edef-czero}
In Proposition~\ref{prop:osc-remainder} the remainder \eqref{eq:Edef}
is
$\mathcal E(s)=\tfrac12Q_{\rm rat}(s)-\operatorname{Re}\mathcal A-\tfrac14\psi'(s)$, with no lattice term; and the two-component
relation takes the closed form
\begin{equation}\label{eq:tworelation-X}
 \Bigl(\tan\frac{\pi s}2-i\Bigr)\mathcal A^{\uparrow}
 +\Bigl(\tan\frac{\pi s}2+i\Bigr)\mathcal A^{\downarrow}
 =-\,T(s)+\tan\frac{\pi s}2\,
 \Bigl(Q_{\rm rat}(s)-\frac{\psi'(s)}2\Bigr)=:X(s).
\end{equation}
\end{corollary}

\begin{proof}
Insert $Q=Q_{\rm rat}-\cot\frac{\pi s}2\,T$
(Theorem~\ref{thm:czero}) into \eqref{eq:tworelation} and use
$\tan\cdot\cot=1$.
\end{proof}

\subsection{Proof of Theorem~\ref{thm:imclosed}}
\label{app:sub-imclosed}

One more formal lemma is needed.  Let
$S_{\psi}(s)=s^{-1}+\tfrac12s^{-2}+\sum_{k\ge1}\Ber_{2k}s^{-2k-1}$ denote
the Stirling series of $\psi'$.

\begin{lemma}
\label{lem:stirparity}
As an identity in $\mathbb Q[[s^{-1}]]$, with the left side re-expanded
in powers of $s^{-1}$,
\[
 S_\psi(1-s)=-S_\psi(s).
\]
\end{lemma}

\begin{proof}
Work in the Borel (Watson) picture, in which
$\sum_nc_ns^{-n-1}$ corresponds to the kernel $\sum_nc_nt^n/n!$ and a
shift $s\mapsto s-a$ corresponds, exactly and coefficientwise, to
multiplication of the kernel by $e^{at}$.  The kernel of $S_\psi$ is
$\varphi_\psi(t)=t/(1-e^{-t})$.  First,
$\varphi_\psi(t)(1-e^{t})=-te^{t}$ is the kernel of $-\sum_{n\ge1}ns^{-n-1}
=-(s-1)^{-2}$, i.e.
\begin{equation}\label{eq:Spsi-rec}
 S_\psi(s)-S_\psi(s-1)=-(s-1)^{-2}
\end{equation}
coefficientwise.  Second, directly from the parity of the
coefficients, $S_\psi(-x)+S_\psi(x)=x^{-2}$.  Applying this at
$x=s-1$ and then \eqref{eq:Spsi-rec},
\[
 S_\psi(1-s)=-S_\psi(s-1)+(s-1)^{-2}
 =-\bigl[S_\psi(s)+(s-1)^{-2}\bigr]+(s-1)^{-2}=-S_\psi(s).
\]
\end{proof}

\begin{proof}[Proof of Theorem~\ref{thm:imclosed}]
All expansions below are along the vertical line
$s=-\tfrac12+iY$, $Y\to-\infty$; recall that the coefficients of a
Poincar\'e expansion along a single ray are unique, so establishing
\eqref{eq:formal-tangent} there establishes it as an identity of formal
series.  The line lies, for $Z=1-2s$, in $\Omega_{1/2}$, and in the
sector $\operatorname{Im}s\le-\tfrac12|s|$ for $|Y|\ge1$.

By Proposition~\ref{prop:bilateral}(iii),
\begin{equation}\label{eq:match-1}
 \mathcal A^{\uparrow}(s)=\sum_{n=1}^Ng_nZ^{-n}+O(|Z|^{-N-1}),
\end{equation}
and $\mathcal A^{\downarrow}(s)
=\overline{\mathcal A^{\uparrow}(\bar s)}=O(|Z|^{-1})$ since $\bar s$
lies in the mirror line.  By Lemma~\ref{lem:blocksasy}(ii),
$\tan\frac{\pi s}2+i=O(e^{-\pi|Y|})$ and
$\tan\frac{\pi s}2-i=-2i+O(e^{-\pi|Y|})$.  The right side of
\eqref{eq:tworelation-X} is $O(1/|Y|)$ by \eqref{eq:T-global} and
Lemma~\ref{lem:blocksasy}(i).  Hence, solving
\eqref{eq:tworelation-X} for $\mathcal A^{\uparrow}$,
\begin{equation}\label{eq:match-2}
 \mathcal A^{\uparrow}(s)=\frac{X(s)}{-2i}+O\bigl(e^{-\pi|Y|}\bigr)
 =\frac i2X(s)+O\bigl(e^{-\pi|Y|}\bigr).
\end{equation}
Now expand $X$.  By the mirror statement of
Proposition~\ref{prop:latticetail} (with $\delta=\tfrac12$),
$T(s)=\sum_{j\le N}\omega_js^{-j}+O(|s|^{-N-1})+O(e^{-\pi|Y|})$, with
$\sum \omega_js^{-j}=\widehat Q_\Gamma$; by Lemma~\ref{lem:blocksasy}(i)
and Corollary~\ref{cor:rational-quadrature}, $Q_{\rm rat}$ and
$\psi'$ are asymptotic to their Stirling series, the former to the
unique formal solution of its central-difference equation; and
$\tan\frac{\pi s}2=-i+O(e^{-\pi|Y|})$.  Therefore
\[
 X(s)=-\widehat Q_\Gamma(s)
 -i\Bigl(Q_{\rm rat}^{\rm ser}(s)-\tfrac12S_\psi(s)\Bigr)
 +O(|s|^{-N-1})+O(e^{-\pi|Y|})
\]
in the obvious truncated sense, where $Q_{\rm rat}^{\rm ser}$ denotes
the (rational-coefficient) Stirling series of \eqref{eq:Qrat}.
Combining with \eqref{eq:match-1}--\eqref{eq:match-2} and equating
coefficients along the line,
\begin{equation}\label{eq:match-3}
 \sum_{n\ge1}g_nZ^{-n}
 =\frac12\,Q_{\rm rat}^{\rm ser}(s)-\frac14\,S_\psi(s)
 -\frac i2\,\widehat Q_\Gamma(s),
\end{equation}
all three series on the right being re-expanded in $Z^{-1}$ through the
exact substitution $s=(1-Z)/2$.  The three series have rational
coefficients, so the splitting of \eqref{eq:match-3} into real and
imaginary parts is coefficientwise:
$\operatorname{Im}g_n=-\tfrac12\widehat q_n$, which is
\eqref{eq:Imgclosed}, and
$\sum\operatorname{Re}g_nZ^{-n}
=\tfrac12Q_{\rm rat}^{\rm ser}-\tfrac14S_\psi(s)$.  For the real part,
$\tfrac12Q_{\rm rat}^{\rm ser}
=\tfrac1{16}\psi'\bigl(\tfrac{Z+1}4\bigr)
-\tfrac1{16}\psi'\bigl(\tfrac{Z+3}4\bigr)$ as Bernoulli series (the
arguments $(1-s)/2$ and $1-s/2$ equal $(Z+1)/4$ and $(Z+3)/4$), while
Lemma~\ref{lem:stirparity} gives
$-\tfrac14S_\psi(s)=\tfrac14S_\psi(1-s)
=\tfrac14\psi'\bigl(\tfrac{Z+1}2\bigr)$ as Bernoulli series
($1-s=(Z+1)/2$); so the real part of \eqref{eq:match-3} is exactly
\eqref{eq:realkernel}, in agreement with
Theorem~\ref{thm:realkernel}.  This proves
\eqref{eq:formal-tangent}.
\end{proof}


\begin{thebibliography}{99}

\bibitem{Szego}
G. Szeg\H{o},
\emph{Orthogonal Polynomials}, 4th ed.,
American Mathematical Society Colloquium Publications, vol.~23,
American Mathematical Society, Providence, RI, 1975.
ISBN 978-0-8218-1023-1.

\bibitem{DLMF}
F. W. J. Olver, A. B. Olde Daalhuis, D. W. Lozier,
B. I. Schneider, R. F. Boisvert, C. W. Clark, B. R. Miller,
B. V. Saunders, H. S. Cohl and M. A. McClain, eds.,
\emph{NIST Digital Library of Mathematical Functions},
Release 1.2.7 of 2026-06-15,
\url{https://dlmf.nist.gov/}.

\bibitem{AndrewsAskeyRoy}
G. E. Andrews, R. Askey and R. Roy,
\emph{Special Functions},
Encyclopedia of Mathematics and its Applications, vol.~71,
Cambridge University Press, Cambridge, 1999.
\doi{10.1017/CBO9781107325937}.

\bibitem{KrallFrink}
H. L. Krall and O. Frink,
A new class of orthogonal polynomials: the Bessel polynomials,
\emph{Trans. Amer. Math. Soc.} \textbf{65} (1949), 100--115.
\doi{10.1090/S0002-9947-1949-0028473-1}.

\bibitem{Burchnall}
J. L. Burchnall,
The Bessel polynomials,
\emph{Canad. J. Math.} \textbf{3} (1951), 62--68.
\doi{10.4153/CJM-1951-009-3}.

\bibitem{DunsterEtAl}
T. M. Dunster, A. Gil, D. Ruiz-Antol\'{\i}n and J. Segura,
Computation of the reverse generalized Bessel polynomials and their zeros,
\emph{Comput. Math. Methods} \textbf{3} (2021), no.~6, e1198.
\doi{10.1002/cmm4.1198}.

\bibitem{DimitrovSantos}
D. K. Dimitrov and E. J. C. dos Santos,
Asymptotic behaviour of Jacobi polynomials and their zeros,
\emph{Proc. Amer. Math. Soc.} \textbf{144} (2016), 535--545.
\doi{10.1090/proc/12689}.

\bibitem{GilSeguraTemme}
A. Gil, J. Segura and N. M. Temme,
Asymptotic expansions of Jacobi polynomials for large values of $\beta$
and of their zeros,
\emph{SIGMA} \textbf{14} (2018), Paper 073, 9 pp.
\doi{10.3842/SIGMA.2018.073}.

\bibitem{Hahn1980}
E. Hahn,
Asymptotik bei Jacobi-Polynomen und Jacobi-Funktionen,
\emph{Math. Z.} \textbf{171} (1980), 201--226.

\bibitem{FrenzenWong1985}
C. L. Frenzen and R. Wong,
A uniform asymptotic expansion of the Jacobi polynomials with error
bounds,
\emph{Canad. J. Math.} \textbf{37} (1985), no.~5, 979--1007.
\doi{10.4153/CJM-1985-053-5}.

\bibitem{HaleTownsend2013}
N. Hale and A. Townsend,
Fast and accurate computation of Gauss--Legendre and Gauss--Jacobi
quadrature nodes and weights,
\emph{SIAM J. Sci. Comput.} \textbf{35} (2013), no.~2, A652--A674.
\doi{10.1137/120889873}.

\bibitem{DeanoHuybrechsOpsomer2016}
A. Dea\~no, D. Huybrechs and P. Opsomer,
Construction and implementation of asymptotic expansions for Jacobi-type
orthogonal polynomials,
\emph{Adv. Comput. Math.} \textbf{42} (2016), 791--822.
\doi{10.1007/s10444-015-9442-z}.

\bibitem{BremerYang2020}
J. Bremer and H. Yang,
Fast algorithms for Jacobi expansions via nonoscillatory phase
functions,
\emph{IMA J. Numer. Anal.} \textbf{40} (2020), no.~3, 2019--2051.
\doi{10.1093/imanum/drz016}.

\bibitem{GilSeguraTemme2021}
A. Gil, J. Segura and N. M. Temme,
Asymptotic expansions of Jacobi polynomials and of the nodes and weights
of Gauss--Jacobi quadrature for large degree and parameters in terms of
elementary functions,
\emph{J. Math. Anal. Appl.} \textbf{494} (2021), 124642.
\doi{10.1016/j.jmaa.2020.124642}.

\bibitem{Nemes2025}
G. Nemes,
Large-degree asymptotic expansions for the Jacobi and related functions,
\emph{Math. Comp.}, published electronically 29 October 2025.
\doi{10.1090/mcom/4145}.

\bibitem{HuangLinWangWong2025}
X.-M. Huang, Y. Lin, X.-S. Wang and R. Wong,
Error bounds for the asymptotic expansions of the Jacobi polynomials,
arXiv:2508.04520 (2025).

\bibitem{companion}
I. Area, Dyadic denominator bounds and a transfer principle for the phase constants of Jacobi zeros,
arXiv: 2608.23006 (2026).

\bibitem{VidunasContiguous}
R. Vid\=unas,
Contiguous relations of hypergeometric series,
\emph{J. Comput. Appl. Math.} \textbf{153} (2003), 507--519.
\doi{10.1016/S0377-0427(02)00643-X}.

\bibitem{VidunasAppell}
R. Vid\=unas,
Specialization of Appell's functions to univariate hypergeometric functions,
\emph{J. Math. Anal. Appl.} \textbf{355} (2009), 145--163.
\doi{10.1016/j.jmaa.2009.01.047}.

\bibitem{BC1}
J. L. Burchnall and T. W. Chaundy,
Expansions of Appell's double hypergeometric functions,
\emph{Quart. J. Math. Oxford Ser.} \textbf{11} (1940), 249--270.
\doi{10.1093/qmath/os-11.1.249}.

\bibitem{BC2}
J. L. Burchnall and T. W. Chaundy,
Expansions of Appell's double hypergeometric functions (II),
\emph{Quart. J. Math. Oxford Ser.} \textbf{12} (1941), 112--128.
\doi{10.1093/qmath/os-12.1.112}.

\bibitem{PWZ}
M. Petkov\v{s}ek, H. S. Wilf and D. Zeilberger,
\emph{A${}={}$B},
A K Peters, Wellesley, MA, 1996.
ISBN 1-56881-063-6. \doi{10.1201/9781439864500}.


\bibitem{AndersonEtAl}
G. D. Anderson, R. W. Barnard, K. C. Richards,
M. K. Vamanamurthy and M. Vuorinen,
Inequalities for zero-balanced hypergeometric functions,
\emph{Trans. Amer. Math. Soc.} \textbf{347} (1995), 1713--1723.
\doi{10.1090/S0002-9947-1995-1264800-3}.
\bibitem{SimicVuorinen}
S. Simi\'c and M. Vuorinen,
Landen inequalities for zero-balanced hypergeometric functions,
\emph{Abstr. Appl. Anal.} \textbf{2012} (2012), Article ID 932061, 11 pp.
\doi{10.1155/2012/932061}.
\end{thebibliography}
\end{document}